\documentclass[11pt]{article}
\newcommand{\comment}[1]		{}
\usepackage{amsmath,amssymb,amsthm}
\usepackage{rotating,pdflscape,afterpage}
\usepackage{array,xparse,patchcmd,xpatch}
\usepackage[sc,osf]{mathpazo}
	\AtBeginDocument{
		\DeclareSymbolFont{AMSb}{U}{msb}{m}{n}
		\DeclareSymbolFontAlphabet{\mathbb}{AMSb}
	}
\usepackage{graphicx}
\usepackage{extpfeil}
\usepackage{mathabx} 
\usepackage{fancyvrb}
\usepackage{float}
\usepackage{caption}
\newcommand{\mockalph}[1]{\!}
\usepackage{enumitem}

\usepackage[nouppercase]{scrlayer-scrpage}
\usepackage{avant}
\usepackage{mathrsfs}
\usepackage{mathtools}
\usepackage{faktor,xfrac}
\usepackage[all]{xypic}\xyoption{rotate}
	\newdir{ >}{{}*!/-7pt/\dir{>}}
	\newdir{i}{{}*!/-7pt/\dir{(}	}

\usepackage{thmtools}

\usepackage[utf8]{inputenc}
\usepackage[T1]{fontenc}
\usepackage{tocloft}
\makeatletter
\renewcommand{\l@figure}{\@dottedtocline{1}{1em}{3.5em}}
\renewcommand{\l@table}{\@dottedtocline{2}{1em}{3.5em}}
\makeatother
\newcommand*{\noaddvspace}{\renewcommand*{\addvspace}[1]{}}
\addtocontents{lof}{\protect\noaddvspace}

\usepackage[hyphens]{url}		%
\usepackage{hyperref}	
\usepackage{cleveref} 	

\makeatletter
\let\c@figure\c@table
\let\c@equation\c@table
\makeatother

\numberwithin{table}{section}
\numberwithin{figure}{section}
\newtheorem{abcthm}[table]{Theorem}

\newtheorem{theorem}[table]{Theorem}

\newtheorem{proposition}[table]{Proposition}

\newtheorem{corollary}[table]{Corollary}

\newtheorem{lemma}[table]{Lemma}

\newtheorem{claim}[table]{Claim}

\theoremstyle{definition}

\newtheorem{definition}[table]{Definition}

\newtheorem{notation}[table]{Notation}

\newtheorem{observation}[table]{Observation}

\newtheorem{conjecture}[table]{Conjecture}

\theoremstyle{remark}
\newtheorem{fact}[table]{Fact}

\newtheorem{example}[table]{Example}
\newtheorem{exercise}[table]{Exercise}

\newtheorem{problem}[table]{Problem}
\newtheorem{histrmks}[table]{Historical remarks}
\newtheorem{remark}[table]{Remark}
\newtheorem{remarks}[table]{Remarks}

\theoremstyle{plain}
\newtheorem*{thm*}{Theorem}
\newtheorem*{theorem*}{Theorem}
\newtheorem*{prop*}{Proposition}
\newtheorem*{proposition*}{Proposition}
\newtheorem*{lemma*}{Lemma}
\newtheorem*{corollary*}{Corollary}
\newtheorem*{cor*}{Corollary}

\theoremstyle{definition}
\newtheorem*{definition*}{Definition}
\newtheorem*{defn*}{Definition}
\newtheorem*{QQ*}{Question}
\newtheorem*{obs*}{Observation}
\newtheorem*{notation*}{Notation}

\theoremstyle{remark}
\newtheorem*{rmk*}{Remark}
\newtheorem*{remark*}{Remark}

\newtheorem*{examples*}{Examples}
\newtheorem*{example*}{Example}
\newtheorem*{EG*}{Example}
\newtheorem*{EGs*}{Examples}
\newtheorem*{fact*}{Fact}
\newtheorem*{prob*}{Problem}

\usepackage{etoolbox}
\usepackage[usenames,dvipsnames,svgnames,table]{xcolor}
\newcommand		{\defd}[1]	{\textcolor{RoyalBlue}{\textbf{\textit{#1}}}}
\newcommand		{\defm}[1]	{\textcolor{RoyalBlue}{#1}}

\tracingpatches
\makeatletter
\patchcmd{\@setref}{\bfseries ??}{\bfseries\color{red} FIX ME!}{}{}
\patchcmd{\@setcref}         {??}{\color{red} FIX ME!}{}{}
\patchcmd{\@setcref}         {??}{\color{red} FIX ME!}{}{}
\patchcmd{\@setcrefrange}    {??}{\color{red} FIX ME!}{}{}
\patchcmd{\@setcrefrange}    {??}{\color{red} FIX ME!}{}{}
\patchcmd{\@setcrefrange}    {??}{\color{red} FIX ME!}{}{}
\patchcmd{\@setcrefrange}    {??}{\color{red} FIX ME!}{}{}
\patchcmd{\@setcrefrange}    {??}{\color{red} FIX ME!}{}{}
\patchcmd{\@setcrefrange}    {??}{\color{red} FIX ME!}{}{}
\patchcmd{\@setnamecref}     {??}{\color{red} FIX ME!}{}{}
\patchcmd{\@setnamecref}     {??}{\color{red} FIX ME!}{}{}
\patchcmd{\@setcpageref}     {??}{\color{red} FIX ME!}{}{}
\patchcmd{\@setcpageref}     {??}{\color{red} FIX ME!}{}{}
\patchcmd{\@setcpagerefrange}{??}{\color{red} FIX ME!}{}{}
\patchcmd{\@setcpagerefrange}{??}{\color{red} FIX ME!}{}{}
\patchcmd{\@setcpagerefrange}{??}{\color{red} FIX ME!}{}{}
\patchcmd{\@setcpagerefrange}{??}{\color{red} FIX ME!}{}{}
\patchcmd{\@setcpagerefrange}{??}{\color{red} FIX ME!}{}{}
\patchcmd{\@cref}            {??}{\color{red} FIX ME!}{}{}
\patchcmd{\@setcite}{\bfseries ?!}{\bfseries\color{red} FIX ME!}{}{}
\patchcmd{\@citex}{\bfseries ?}{\color{red}{\@citeb}}{}{%
	\GenericWarning{}{Failed to patch \protect\@citex}}
\def\blx@citation@entry#1#2{%
	\blx@bibreq{#1}%
	\ifinlist{#1}{\blx@cites}
	{}
	{\listgadd{\blx@cites}{#1}%
		\blx@auxwrite\@mainaux{}{\string\abx@aux@cite{#1}}}%
	\ifinlistcs{#1}{blx@segm@\the\c@refsection @\the\c@refsegment}
	{}
	{\listcsgadd{blx@segm@\the\c@refsection @\the\c@refsegment}{#1}}%
	\blx@ifdata{#1}%
	{}%
	{\ifcsdef{blx@miss@\the\c@refsection}%
		{\ifinlistcs{#1}{blx@miss@\the\c@refsection}%
			{{\bfseries\color{red} cite:} }%
			{\blx@logreq@active{#2{#1}}}}%
		{\blx@logreq@active{#2{#1}}}}}
\def\blx@citeadd#1{%
	\ifcsdef{blx@keyalias@\the\c@refsection @#1}
	{\edef\blx@realkey{\csuse{blx@keyalias@\the\c@refsection @#1}}}
	{\def\blx@realkey{#1}}%
	\expandafter\blx@citation\expandafter{\blx@realkey}\blx@msg@cundefon
	\expandafter\blx@ifdata\expandafter{\blx@realkey}
	{\advance\blx@tempcnta\@ne
		\listeadd\blx@tempa{\blx@realkey}}
	{\ifnum\blx@tempcntb>\z@\multicitedelim\fi
		\expandafter\abx@missing\expandafter{\blx@realkey}%
		\advance\blx@tempcntb\@ne}}
\makeatother

\makeatletter
\DeclarePairedDelimiterX{\pmodx}[1]{(}{)}{{\operator@font mod}\mkern6mu#1}
\renewcommand{\pmod}{%
	\allowbreak
	\if@display\mkern18mu\else\mkern8mu\fi
	\pmodx
}
\makeatother

\makeatletter
\newcommand{\oset}[3][0ex]{%
\raisebox{.175ex}{$%
  \mathrel{\mathop{#3}\limits^{
    \vbox to#1{\kern-2\ex@
    \hbox{$\scriptstyle#2$}\vss}}}
    $}%
    }
\makeatother

\newcommand{\myred}{BrickRed}
\hypersetup{
		colorlinks			= true,
		urlcolor			= \myred,
		linkcolor			= \myred,
		citecolor			= PineGreen
}

\usepackage{xspace}

\usepackage[left=2.54cm,right=2.54cm,top=2.54cm,bottom=2.54cm]{geometry}
\usepackage{tikz}
\usetikzlibrary{matrix,fit,backgrounds,calc,arrows} 

\tikzstyle{image}=[rectangle,fill=Red!20,inner sep=-2pt]
\tikzstyle{nonzero}=[rectangle,fill=Navy!20,inner sep=0pt]
\tikzstyle{nonzerosm}=[rectangle,fill=Navy!20,inner sep=-2pt]

\makeatletter
\newbox\xrat@below
\newbox\xrat@above
\newcommand{\xrightarrowtail}[2][]{%
  \setbox\xrat@below=\hbox{\ensuremath{\scriptstyle #1}}%
  \setbox\xrat@above=\hbox{\ensuremath{\scriptstyle #2}}%
  \pgfmathsetlengthmacro{\xrat@len}{max(\wd\xrat@below,\wd\xrat@above)+.6em}%
  \mathrel{\tikz [>->,baseline=-.55ex]
                 \draw (0,0) -- node[below=-2pt] {\box\xrat@below}
                                node[above=-2pt] {\box\xrat@above}
                       (\xrat@len,0) ;}}
\newbox\xrat@below
\newbox\xrat@above
\renewcommand{\xtwoheadrightarrow}[2][]{%
  \setbox\xrat@below=\hbox{\ensuremath{\scriptstyle #1}}%
  \setbox\xrat@above=\hbox{\ensuremath{\scriptstyle #2}}%
  \pgfmathsetlengthmacro{\xrat@len}{max(\wd\xrat@below,\wd\xrat@above)+.6em}%
  \mathrel{\tikz [->>,baseline=-.55ex]
                 \draw (0,0) -- node[below=-2pt] {\box\xrat@below}
                                node[above=-2pt] {\box\xrat@above}
                       (\xrat@len,0) ;}}
\makeatother
\newcommand{\xmono}{\xrightarrowtail}

\newcommand{\xepi}{\xtwoheadrightarrow}
\newcommand{\epi}{\xepi{\phantom{\ \, }}}

\makeatletter
\newcommand{\presectionskip}{-1.5\baselineskip}
\newcommand{\postsectionskip}{0.3\baselineskip}
\usepackage{titlesec}
\renewcommand{\section}{\@startsection
  {chapter}{0}{0mm}
  {\presectionskip}
  {\postsectionskip}
  {\sffamily\huge}}
\renewcommand{\section}{\@startsection
  {section}{1}{0mm}
  {\presectionskip}
  {\postsectionskip}
  {\sffamily\LARGE}}
\renewcommand{\subsection}{\@startsection
  {subsection}{2}{0mm}
  {\presectionskip}
  {\postsectionskip}
  {\sffamily\Large}}
\renewcommand{\subsubsection}{\@startsection
  {subsubsection}{3}{0mm}
  {\presectionskip}
  {\postsectionskip}
  {\sffamily\normalsize}}
\renewcommand{\@seccntformat}[1]{\csname the#1\endcsname.\quad}
\newcommand\HUGE{\@setfontsize\Huge{30}{47}}
\titleformat{\chapter}[display]
{\sffamily\Large}
{Chapter {\HUGE\normalfont\thechapter}}    
{1em}
{\huge}
\titleformat{\section}
{\sffamily\Large}
{\normalfont{\@setfontsize\Large{18}{47}\thesection.}}    
{1em}
{}
\titleformat{\subsection}
{\sffamily\large}
{\normalfont{\Large\thesubsection}.}    
{1em}
{}
\titleformat{\subsubsection}
{\sffamily\normalsize}
{\normalfont{\large\thesubsubsection}.}    
{1em}
{}
\makeatother

\makeatletter
\def\smallunderbrace#1{\mathop{\vtop{\m@th\ialign{##\crcr
   $\hfil\displaystyle{#1}\hfil$\crcr
   \noalign{\kern3\p@\nointerlineskip}%
   \tiny\upbracefill\crcr\noalign{\kern3\p@}}}}\limits}
\makeatother

\newcommand{\bthm}{\begin{theorem}}
\newcommand{\ethm}{\end{theorem}}
\newcommand{\bprop}{\begin{proposition}}
\newcommand{\eprop}{\end{proposition}}
\newcommand{\bcor}{\begin{corollary}}
\newcommand{\ecor}{\end{corollary}}
\newcommand{\bconj}{\begin{conjecture}}
\newcommand{\econj}{\end{conjecture}}
\newcommand{\blem}{\begin{lemma}}
\newcommand{\elem}{\end{lemma}}
\newcommand{\bclm}{\begin{claim}}
\newcommand{\eclm}{\end{claim}}
\newcommand{\bpf}{\begin{proof}}
\newcommand{\epf}{\end{proof}}
\newcommand{\bdetails}{\begin{details}}
\newcommand{\edetails}{\end{details}}
\newcommand{\bdefi}{\begin{definition}}
\newcommand{\edefi}{\end{definition}}
\newcommand{\bdefn}{\begin{definition}}
\newcommand{\edefn}{\end{definition}}
\newcommand{\bex}{\begin{example}}
\newcommand{\eex}{\end{example}}
\newcommand{\bprob}{\begin{problem}}
\newcommand{\eprob}{\end{problem}}
\newcommand{\bob}{\begin{observation}}
\newcommand{\eob}{\end{observation}}
\newcommand{\bexer}{\begin{exercise}}
\newcommand{\eexer}{\end{exercise}}
\newcommand{\bexers}{\begin{exercises}}
\newcommand{\eexers}{\end{exercises}}
\newcommand{\brmk}{\begin{remark}}
\newcommand{\ermk}{\end{remark}}
\newcommand{\bhist}{\begin{histrmks}}
\newcommand{\ehist}{\end{histrmks}}
\newcommand{\brmks}{\begin{remarks}}
\newcommand{\ermks}{\end{remarks}}
\newcommand{\bverb}{\begin{verbatim}}
\newcommand{\everb}{\end{verbatim}}
\newcommand{\bntn}{\begin{notation}}
\newcommand{\entn}{\end{notation}}
\newcommand{\bfct}{\begin{fact}}
\newcommand{\efct}{\end{fact}}
\newcommand{\bfcts}{\begin{facts}}
\newcommand{\befcts}{\end{facts}}
\newcommand{\benum}{\begin{enumerate}}
\newcommand{\eenum}{\end{enumerate}}
\newcommand{\bitem}{\begin{itemize}}
\newcommand{\eitem}{\end{itemize}}

\ExplSyntaxOn
\NewDocumentEnvironment{adjunctions}{O{}}
{
	\cs_set_eq:cN {@arraycr} \farin_arraycr:
	\keys_set:nn { farin/adjunction } { #1 }
	\begin{array}
		{
			@{ \hspace { \dim_eval:n { \l_farin_left_shift_dim + \l_farin_padding_dim } } }
			r
			@{ {\farin_strut:} \l_farin_symbol_tl {} }
			l
			@{ \hspace { \dim_eval:n { \l_farin_right_shift_dim + \l_farin_padding_dim } } }
		}
	}
	{
	\end{array}
}
\keys_define:nn { farin/adjunction }
{
	leftshift       .dim_set:N = \l_farin_left_shift_dim,
	leftshift       .initial:n = 0pt,
	rightshift      .dim_set:N = \l_farin_right_shift_dim,
	rightshift      .initial:n = 0pt,
	padding         .dim_set:N = \l_farin_padding_dim,
	padding         .initial:n = 6pt,
	symbol          .tl_set:N  = \l_farin_symbol_tl,
	symbol          .initial:n = \longrightarrow,
	verticalspacing .dim_set:N  = \l_farin_vertspac_dim,
	verticalspacing .initial:n = {3pt},
}
\cs_new_protected:Npn \farin_strut:
{
	\vrule height \dim_eval:n { \ht\strutbox + 1.2\l_farin_vertspac_dim }
	depth  \dim_eval:n { \dp\strutbox + \l_farin_vertspac_dim }
	width 0pt
}
\makeatletter
\exp_args:NNo \cs_new:Npn \farin_arraycr:
{
	\@arraycr\hline
}
\makeatother
\ExplSyntaxOff

\let\hatt\^

\renewcommand	{\epsilon}	{\varepsilon}
\renewcommand	{\a}		{\alpha}
\renewcommand	{\b}		{\beta}
\renewcommand	{\d}		{\delta}

\renewcommand	{\l}		{\lambda}

\renewcommand	{\:}		{\colon}

\newcommand		{\quotientmed}[2]	{{\raisebox{.2em}{$#1$}}\  \!\!\big/\!\!\
										{\raisebox{-.2em}{$#2$}}}

\newcommand		{\qquotientmed}[2]	{{\raisebox{.2em}{$#1$}} \ \big/\!\!\!\!\!\big/ \
										{\raisebox{-.2em}{$#2$}}}

\newcommand		{\fs}		{{\mathfrak s}}
\newcommand		{\fk}		{{\mathfrak k}}
\newcommand		{\ft}		{{\mathfrak t}}
\newcommand		{\fg}		{{\mathfrak g}}

\newcommand		{\tG}		{\widetilde{G}}
\newcommand		{\tH}		{\widetilde{H}}

\newcommand		{\tN}		{\widetilde{N}}

\newcommand		{\wP}		{\widehat{P}}

\newcommand		{\SSS}	{Serre spectral sequence\xspace}

\newcommand		{\eqf}	{equivariantly formal\xspace}

\newcommand		{\eqfity}	{equivariant formality\xspace}
\newcommand		{\isotf}	{isotropy-formal\xspace}
\newcommand		{\isotfity}	{isotropy-formality\xspace}

\newcommand		{\exterior}	{\Lambda}
\newcommand		{\ext}		{\exterior}

\newcommand		{\ab}			{^{\mathrm{ab}}}
\renewcommand	{\th}			{^{\mathrm{th}}}

\newcommand		{\CDGA}		{\textsc{cdga}\xspace}

\newcommand	{\fa}{\f a}

\newcommand{\ccpair}{compact, connected pair\xspace}

\newcommand{\GK}{$\smash{(G,K)}$\xspace}
\newcommand{\GS}{$\smash{(G,S)}$\xspace}

\newcommand		{\dual}	{\mn^\vee}

\makeatletter
\newcommand{\subalign}[1]{%
  \vcenter{%
    \Let@ \restore@math@cr \default@tag
    \baselineskip\fontdimen10 \scriptfont\tw@
    \advance\baselineskip\fontdimen12 \scriptfont\tw@
    \lineskip\thr@@\fontdimen8 \scriptfont\thr@@
    \lineskiplimit\lineskip
    \ialign{\hfil$\m@th\scriptstyle##$&$\m@th\scriptstyle{}##$\crcr
      #1\crcr
    }%
  }
}
\makeatother

\newcommand		{\Ei}		{E_\infty}

\newcommand		{\ang}[1]			{\langle #1 \rangle}

\newcommand		{\eqn}[1]			{\begin{align*} #1 \end{align*}}
\newcommand		{\quation}[1]		{\begin{equation} #1 \end{equation}}

\newcommand		{\case}[1]			{\begin{cases} #1 \end{cases}}

\newcommand		{\hyref}[1]			{\hyperref[#1]{\ref*{#1}}}

\newcommand		{\bs}				{\bigskip}

\newcommand		{\mn}				{\mspace{-2mu}}
\newcommand		{\mnn}				{\mspace{-1mu}}
\newcommand		{\dsp}			{\displaystyle}

\newcommand		{\nd}				{\noindent}
\newcommand		{\ol}				{\overline}
\newcommand		{\os}			{\overset}

\newcommand		{\wh}			{\widehat}

\newcommand		{\wt}			{\widetilde}

\newcommand		{\mr}			{\mathrm}
\newcommand		{\bb}			{\mathbb}

\newcommand		{\f}			{\mathfrak}

\newcommand		{\g}			{\gamma}
\newcommand		{\e}			{\epsilon}
\newcommand		{\z}			{\zeta}

\newcommand		{\s}			{\sigma}

\newcommand		{\w}			{\omega}

\newcommand		{\G}			{\Gamma}
\newcommand		{\D}			{\Delta}

\DeclareSymbolFont{cmletters}{OT1}{cmr}{m}{n}
\DeclareMathSymbol{\Ups}{\mathalpha}{cmletters}{"7}
\renewcommand	{\Upsilon}{\Ups}

\newcommand		{\N}		{\bb N}
\newcommand		{\Z}		{\bb Z}
\newcommand		{\Q}		{\bb Q}
\newcommand		{\R}		{\bb R}
\newcommand		{\C}		{\bb C}

\newcommand		{\Quat}	{\bb H}
\newcommand 		{\HH}	{\mathbb{H}}
\newcommand		{\Oct}	{\bb O}

\newcommand		{\CP}		{\bb C \mr P}
\newcommand		{\HP}		{\bb H \mr P}
\newcommand		{\OP}	{\bb O \mr P}

\renewcommand	{\Im}		{\mr{Im}}
\DeclareMathOperator{\id}			{id{}}

\renewcommand 		{\H}		{H^*}
\newcommand 		{\HG}	{{\H_G}}

\newcommand 		{\HT}	{{\H_T}}

\newcommand 		{\HS}	{{\H_S}}

\let\union\cup%
\renewcommand	{\cup}		{\mspace{-1mu}\smile\mspace{-1mu}}
\let\inter\cap%
\renewcommand	{\cap}		{\mspace{-1mu}\frown\mspace{-1mu}}

\newcommand		{\Inter}		{\bigcap}
\newcommand		{\less}		{\setminus}
\newcommand		{\sub}		{\subseteq}
\newcommand		{\subn}		{\subsetneq}

\newcommand		{\dis}			{\amalg}

\newcommand		{\quot}		{\,/ \mn\mn /\,}

\renewcommand	{\-}		{^{-1}}

\renewcommand	{\.}		{\cdot}
\newcommand		{\x}		{\times}

\newcommand{\oplushigher}{\mathbin{\raisebox{.85pt}{$\displaystyle\oplus$}}}

\DeclareMathOperator*{\otimesvariable}{%
			\mathchoice {\raisebox{.85pt}{$\displaystyle\otimes$}}
						{\raisebox{.85pt}{$\otimes$}}
						{\raisebox{0.7pt}{$\scriptstyle\otimes$}}
						{\raisebox{0.2pt}{$\scriptscriptstyle\otimes$}}
						}
\newcommand		{\tensor}		{\otimesvariable}

\newcommand		{\direct}		{\oplushigher}
\newcommand		{\ox}			{\tensor}

\newcommand		{\+}			{\direct}

\newcommand		{\Direct}		{\bigoplus}

\DeclareMathOperator{\diag}		{diag}

\DeclareMathOperator{\rk}			{rk }
\DeclareMathOperator{\im}		{im }

\DeclareMathOperator{\coker}		{coker }

\DeclareMathOperator{\Stab}		{Stab }

\DeclareMathOperator{\Ad}		{Ad }

\DeclareMathOperator{\Homeo}	{Homeo}

\DeclareMathOperator{\Aut}		{Aut }

\newcommand		{\SO}		{\mr{SO}}

\newcommand		{\U}			{\mr{U}}
\newcommand		{\SU}			{\mr{SU}}

\newcommand		{\Sp}			{\mr{Sp}}
\newcommand		{\Spin}		{\mr{Spin}}

\newcommand		{\so}			{\mathfrak{so}}

\newcommand		{\spin}		{\mathfrak{spin}}

\newcommand		{\longto} 		{\longrightarrow}
\newcommand		{\lt}			{\longto}

\newcommand		{\lmt}			{\longmapsto}

\newcommand		{\inc}		{\hookrightarrow}
\newcommand		{\xinc}		{\xhookrightarrow}
\newcommand		{\longinc}		{\xinc[]{\ \ \ \ }}

\newcommand		{\longmono}	{\xmono[]{\ \ \ \ }}

\newcommand		{\longepi}	{\xepi[]{\ \ \ \ }}

\newcommand		{\simto}		{\xrightarrow{\sim}}

\newcommand		{\longsimto}	{\os\sim\longto}
\newcommand		{\isoto}		{\longsimto}

\newcommand		{\ceq}			{\coloneqq}
\newcommand		{\eqc}			{\eqqcolon}

\newcommand		{\ideal}			{\unlhd}

\newcommand		{\normal}			{\unlhd}

\newcommand		{\iso}				{\cong}
\newcommand		{\homeo}			{\approx}

\usepackage{rotating}
\usepackage{pdflscape}
\usepackage[greek,english]{babel}
\numberwithin{equation}{section}
\usepackage{multirow,makecell}

\usepackage{pifont}
\newcommand{\xmark}{\ding{55}}%
\newcommand{\Negative}{\qquad\qquad\qquad\ \ \ \xmark}
\newcommand{\Smallneg}{\qquad\ \ \ \xmark}
\newcommand{\yesgreater}{$>$}
\newcommand{\yesequal}{\quad \,\  $=$}
\theoremstyle{definition}
\newtheorem{discussion}[table]{Discussion}

\usepackage{lineno}
\newif\ifdebug                                                      
\newcommand {\revision}[1] {\ifdebug\textcolor{Red}{#1}\else{#1}\fi}

\newcommand{\printname}[1] 
{\smash{\makebox[0pt]{\hspace{-.0in}\raisebox{8pt}{\tiny #1}}}}

\newcommand {\labell}[1] %
{\ifdebug {\label{#1}\printname{#1}} \else {\label{#1}} \fi}

\renewcommand			{\H}		{H^*}
\newcommand			{\fh}		{\f h}
\newcommand			{\Hodd}		{H^{\mr{odd}}}
\newcommand			{\Heven}	{H^{\mr{even}}}

\newcommand			{\tGtH}		{$(\tG,\tH)$\xspace}

\newcommand			{\GH}		{$(G,H)$\xspace}
\newcommand			{\GHS}		{$(G,H_S)$\xspace}

\newcommand			{\Hl}		{{H_S}}
\newcommand			{\Hlp}		{H'}
\newcommand			{\Kl}		{{L_S}}
\newcommand			{\fkl}		{\f l} 
\newcommand			{\WH}		{{W_\Hl}}
\newcommand			{\WK}		{{W_\Kl}}

\newcommand			{\weight}		{\a}
\newcommand			{\genwt}	{\b}
\newcommand			{\simple}	{\g}

\newcommand			{\St}		{S^\perp}

\newcommand			{\te}		{\widetilde e}

\newcommand			{\Cen}		{Z}
\newcommand			{\Sym}		{\Sigma}
\newcommand			{\trivialgroup}	{1}

\newcommand			{\mrp}		{maximal regular pair\xspace}
\newcommand			{\rsp}		{rational sphere pair\xspace}
\newcommand			{\rspp}		{rational sphere product pair\xspace}
\newcommand			{\rsppp}	{rational sphere (product) pair\xspace}
\newcommand			{\ip}		{irreducible pair\xspace}
\newcommand			{\isp}		{irreducible subpair\xspace}

\newcommand			{\ve}		{virtually effective\xspace}

\newcommand			{\csa}		{centralizer-and-sphere argument\xspace}

\begin{document}
\title{\vspace{-1em}\huge Equivariant formality of corank-one isotropy actions 
and products of rational spheres}
\author{\Large Jeffrey D.~Carlson and Chen He}
\maketitle

\begin{abstract}
\small{
	We completely characterize the pairs 
	of connected Lie groups $G > K$
	such that $\rk G - \rk K = 1$
	and the isotropy action of $K$ on $G/K$
	is equivariantly formal.
	The analysis requires us
	to correct and extend 
	an existing partial classification of homogeneous quotients $G/K$
	with the rational homotopy type 
	of a product of an odd- and an even-dimensional
	sphere. 
}
\end{abstract}

\section{Introduction}
Among the most consequential algebraic conditions on 
a continuous group action $G \x X \lt X$
is surjectivity
of the map $\H(X_G;\Q) \longepi \H(X;\Q)$
induced by the fiber inclusion of the Borel fibration $X \to X_G \to BG$,
known by the trade name of \emph{equivariant formality}.
This notion was already
considered by Borel~\cite[Ch.~XII]{borel1960seminar} long before
being brought to its present prominence in well-known 
work of Goresky--Kottwitz--MacPherson~\cite{GKM1998},
and makes cohomology computations tractable
and powerful integral localization theorems applicable~%
\cite{BV1982,AB1984,jeffreykirwan1995}.

The most fundamental type of action
neither guaranteed to be equivariantly formal 
nor equivariantly informal 
is that of the translation action $k\.gK = kgK$ 
of a subgroup~$K$ of a Lie group~$G$
on the homogeneous space~$G/K$,
the (global) \emph{isotropy action}
whose tangent action at $1K$ 
is so important for local study of smooth actions.
One always assumes $G$ and $K$ are connected,
\revision{and then a reduction due to one of the 
authors~\cite[Prop.~3.1]{carlson2018eqftorus} 
allows one to assume they are compact as well.}
Then\revision, despite a number of equivalent characterizations~\cite{carlsonfok2018},
and a number of special classes of cases 
known to be equivariantly formal~\cite{brion1998eqcohom,shiga1996equivariant,goertschesnoshari2016},
one only has a classification of all equivariantly 
formal isotropy actions 
when $\rk G = \rk K$ or $\rk K = 1$~\cite{brion1998eqcohom,
	carlson2018eqftorus}.


In this work we add $\rk G - \rk K = 1$
to the fully classified range,
and find to our surprise 
that we require the classification
of effective transitive actions of compact Lie groups on spaces 
rationally homotopy equivalent to
a product $S^{\mr{odd}} \x S^{\mr{even}}$ of spheres. 
Such homogeneous rational sphere products 
have received extended consideration
as the subject of at least three Ph.D. dissertations,
and an AMS \emph{Memoirs}
volume on isoparametric hypersurfaces
admitting a transitive isometry group on one focal manifold%
~\cite{kamerich,kramer2002homogeneous,bletzsiebert2002,wolfrom2002}, 
but these invaluable and generally very
thorough analyses
unfortunately do not consider all cases we need and suffer
from minor errors and omissions
(see Discussions \ref{rmk:subdivision}, 
	\ref{rmk:comparison-Kamerich}, and \ref{rmk:comparison-Kramer}),
so that we must
revise and complete the classification ourselves.
This classification appears as part of \Cref{table:Kramer+Wolfrom}
and occupies \Cref{sec:cases}.

The classification appears at the end of a sequence of reductions,
which begins with 
a generic \ccpair $(G,K)$ with $\rk G - \rk K = 1$
and ends with an item in \Cref{table:Kramer+Wolfrom}.
The reduction involving rational sphere products
requires a
standard setup we maintain throughout. 

\begin{notation}\labell{def:sec1}
If $\defm G$ is a compact, connected Lie group with closed, connected subgroup $\defm K$ 
whose maximal torus $\defm S$ is of codimension one in a maximal torus $\defm T$ of $G$,
we write $\defm {W_G}$ for the Weyl group,
$\defm{w_0} = \defm{w_0^G} \in W_G$ for the longest word,
$\defm\fg$ for the Lie algebra,
$\defm{N_G(S)}$ and $\defm{\Cen_G(S)}$ respectively 
	for the normalizer and the centralizer,
$\defm N$ for the component group $\pi_0 N_G(S)$,
and  
$\defm{H_S}$ for the largest closed, connected 
	subgroup of $G$ containing $S$ as a maximal torus
	and centralizing the orthogonal complement to $\fs$ in $\ft$
	under a fixed $\Ad(G)$-invariant inner product.
    \revision{
    (See \Cref{def:LH} and the following material for much more on this group.)
    }
\end{notation}

\begin{definition}
	We call a pair of topological groups $(G,K)$ with $G > K$ \defd{isotropy-formal}
	if the isotropy action of $K$ on $G/K$ is equivariantly formal.
\end{definition}

%
%

\begin{theorem}\labell{thm:spheres}
	Let $(G,K)$ be a pair of compact, connected Lie groups 
	such that a maximal torus $S$ of $K$ is of codimension one 
	in a maximal torus $T$ of $G$.
\benum
\item\labell{item:sphere}
	If
	$G/K$ has
	the rational cohomology of 
	an odd-dimensional sphere, 
	then 
	$(G,K)$ is isotropy-formal.
\item\labell{item:product}
	If $G/K$ has the rational cohomology of
	a product $S^n \x S^m$,
	with $n$ odd and $m$ even, 
	then $(G,K)$ is \isotf if and only if
	$|N| \neq |W_K|$. 
\item\labell{item:converse}
	If $K = H_S$, 
	then $(G,K)$ is \isotf if and only if one of the two conditions above
	holds.
\item
	We have $N \ncong W_{H_S}$ if and only if 
	$w_0$
	stabilizes $\fs$ but $w_0|_{\fs}$ is not in $W_{H_S}$. 
\eenum
%
\end{theorem}

The entire procedure is as follows; 
\emph{irreducible} pairs are \GH such that no proper normal subgroup
of $G$ acts transitively on $G/H$,
and \emph{(virtually) effective} pairs those such that
$\ker\mn\big(G \to \Homeo(G/H)\mnn\big)$ is trivial (resp., finite).
Other vocabulary should be intuitive and will be elaborated over the course
of \Cref{sec:spheres}.

\vspace{-.05em}

\begin{theorem}\labell{thm:alg}
	Let \GK be a corank-one pair of compact, connected Lie groups,
	with maximal tori $(T,S)$ 
	as in \Cref{def:sec1}.
	To determine whether it is \isotf, we may do the following.
	\begin{description}
		\item[Step 1:] 
		Construct the associated \mrp \GHS per \Cref{thm:HS}. 
		\item[Step 2:]
		\ 
		\vspace{-2em}
		\bitem 
		\item
		If $\pi_1(G/H_S)$ is infinite, then \GK is \isotf.
		\item 
		If $\pi_1(G/H_S)$ is finite, 
		compute the (finite, even) number $\defm d = \dim_\Q \H(G/H_S)$.
		\eitem
		\item[Step 3:] 
		\ 		
		\vspace{-2em}
		\bitem 
		\item If $d = 2$, then \GK is \isotf.
		\item If $d \geq 6$, then \GK is not \isotf.
		\item If $d = 4$, then \GHS is a \rspp.	\\
		Take the effective \revision{quotient}~$(\ol G, \ol H)$ of \GHS,
		which is \isotf if and only if \GK is.

%
%

		\eitem
		
		\item[Step 4:]	
The pair $(\ol{G}, \ol{H})$ is irreducible if and only if it is in 
\Cref{table:Kramer+Wolfrom}
up to 
a finite covering (in the sense of \Cref{def:cover}).
		\bitem
\item If so, read from \Cref{table:Kramer+Wolfrom} whether it is \isotf.
\item If not, then  it is not isotropy-formal.
		\eitem
		
	\end{description}
\end{theorem}

\vspace{-0.4em}

\afterpage{%
\begin{landscape}
	\begin{table}[htbp]
		\caption{Irreducible {\ccpair}s \GH 
        with \revision{$\pi_1(G/H) = 0$ and}
            $\H(G/H) \iso \H(S^n \x S^m)$ 
			for $n\geq 3$ odd and $m\geq 2$ even}%
		\labell{table:Kramer+Wolfrom}
		\labell{table:main}
		\begin{center}
			\resizebox{\columnwidth}{!}
			{
				\begin{tabular}{	>{$}l<{$} >{$}r<{$}|
						>{$}l<{$}>{$}c<{$}|
						l|
						>{$}l<{$}|
						>{$}l<{$}|
						l|							
						>{$}l<{$}|
						>{$}l<{$} 
					}
					G && H &&
					\makecell{isotropy-formal?}
					& Z_G(H)^0  
					& H_S 
					& $H_S \geq H$?
					& (m,n) 
					& G/H\\
					\hline
					
					\SU(3)	&
					& \makecell{
						i_{p,q}\U(1),\\
						\quad p,q \mbox{ coprime}
					}& 
					& \makecell{
						If and only if \\
						$\{p,q\} = \{0,\pm 1\}$\\ 
						\phantom{
							$\{p,q\}$}\,or $\{1,-1\}$
					}
					& \makecell{
						i_{p,q}\U(1)\cdot \SU(2)_{p+2q,-q-2p}\\
						\phantom{\U(1)^2}\quad    \text{if $p\cdot q \in \{ 1,-2 \}$;}\\
						\U(1)^2 \quad \text{otherwise}
					}
					& \makecell{
						\SU(2)_{p,q} \\
						\phantom{H} \quad \text{if $p\cdot q \in \{ -1,0 \}$;}\\
						H \quad \text{otherwise}
					}
					& \makecell{
						\yesgreater\\
						\\
						\yesequal
					}
					&(2,5) 
					&
					\\
					
					\hline

					\Sp(2)	&		
					& \makecell{
						i_{p,q}\U(1),\\
						\quad p,q \mbox{ coprime}
					}&
					& \checkmark
					& \makecell{
						i_{p,q}\U(1)\cdot \Sp(1)_{q,-p}\\
						\phantom{\U(1)^2}\quad   \text{if $p\cdot q \in \{ 0,\pm 1 \}$;}\\
						\U(1)^2 \quad \text{otherwise}
					}
					& \makecell{
						\Sp(1)_{p,q} \\
						\phantom{H} \quad \text{if $p\cdot q \in\{ 0,\pm 1 \}$;}\\
						H \quad \text{otherwise}
					}
					& \makecell{
						\yesgreater\\
						\\
						\yesequal
					}&(2,7) 
					&
					\\
					
					\hline
					G_2	&
					& \makecell{
						i_{p,q}\U(1),\\
						\quad p,q \mbox{ coprime}
					}&
					& \checkmark
					& \makecell{
						i_{p,q}\U(1)\cdot \SU(2)_{q,-p}\\
						\phantom{\U(1)^2}\quad   \text{if $p\cdot q = 0$;}\\
						\U(1)^2 \quad \text{otherwise}
					}
					& \makecell{
						\SU(2)_{p,q} \\
						\phantom{H} \quad \text{if $p\cdot q = 0$;}\\
						H \quad \text{otherwise}
					}
					& \makecell{
						\yesgreater\\
						\\
						\yesequal
					}&(2,11) 
					&
					\\
					
					\hline

					G' \x \Sp(1)&
					&\makecell{
						H' \. i_{p,q}\U(1),										\\
						\quad p,q \mbox{ coprime}\\
						\quad \ \& \ \Cen_{G'}(H')^0 \neq \trivialgroup \\ 
					}
					&
					& \makecell{
						Unless  
						$(G',H') $\\
						$= \big(\SU(k\mn+\mn1),\SU(k)\mnn\big)$\\
						for $k \geq 2$
					}	
					& \U(1)^2 
					& \makecell{
						H' \. \Sp(1)_{p,q} 					
						\\
						\phantom{H} \quad \mbox{if } Z_{G'}(H')^0 = A_1			\\
						\phantom{H} \quad		
						\mbox{ \& } p\.q = \pm 1	;				\\
						H \quad \text{otherwise}
					}
					& \makecell{
						\yesgreater\\
						\\
						\\
						\yesequal
					}&(2,\dim G'/H') 
					&
					\\
					
					\hline

					K_1 \x K_2&& \multicolumn{2}{l|}{$H_1 \x H_2$} &
					\checkmark &
					\makecell{S^1 \x Z_{K_2}(H_2)^0 \mbox{ if }K_1 = A_1\\
					{}\phantom{S^1 \x {}}{} Z_{K_2}(H_2)^0 \mbox{ otherwise}}
					& K_1\x H_{S_2}
					&
					\!\!\!$\mn$ $H_{S_2} \mn \geq \mn H_2\mathrlap?$\!\!
					&(\dim K_1/H_1,\dim K_2/H_2) & K_1/H_1 \x K_2/H_2\\

					\hline
					\hline
					
					\SU(5)  && \multicolumn{2}{l|}{$\SU(2) \+ \SU(3)$}& 
					\Negative		&
					\U(1) & H
					&\yesequal
					&(4,9) & \widetilde G_2(\C^5)\\
					\SU(4) && \multicolumn{2}{l|}{$\SU(2) \+ \SU(2)$}&
					\checkmark
					& \U(1) & H
					&\yesequal
					&(4,5) & \widetilde G_2(\C^4)\cong V_2(\R^6)\\
					&& \multicolumn{2}{l|}{$\SO(4)$} & 
					\checkmark		&
					\trivialgroup & \SU(2) \+ \SU(2)
					&
					\Smallneg
					&(4,5) & \widetilde G_3(\R^6)\\
					\hline
					
					\SO(2k+1), 
					& k \geq 3 & \multicolumn{2}{l|}{$\SO(2k-2)$} &
					\checkmark		&
					\SO(3) & \SO(2k-1)
					&\yesgreater
					&(2k\mnn-\mn2,4k\mnn-\mn1) & V_3(\R^{2k+1})\\

					\Spin(9) && \multicolumn{2}{l|}{$\SU(4)$} &
					\checkmark		&
					\U(1) & H
					&
					\yesequal
					&(6,15) &  \\
					\Spin(7) && \multicolumn{2}{l|}{$\SU(3)$} &
					\checkmark		&
					\U(1) & H
					&
					\yesequal
					&(6,7) &  V_2(\R^8) \iso S^6 \x S^7 \\
					\SO(7) && \multicolumn{2}{l|}{$\SO(3)\+ \SO(3) \+ [1]$} &
					\checkmark		&
					\trivialgroup & \SO(5)
					&
					\Smallneg
					&(4,11) & \\
					&& \SO(3)\+ \SU(2)&&
					\Negative		
					&
					\U(1)
					& H
					&
					\yesequal
					&(4,11) & \\
					\hline
					
					\Sp(4) && \multicolumn{2}{l|}{$\SU(4)$} & 
					\checkmark		&
					\U(1)
					& H
					&
					\yesequal
					&(6,15) &  \\
					\Sp(3) && \multicolumn{2}{l|}{$\SU(3)$} & 
					\checkmark		&
					\U(1)& H
					&
					\yesequal
					&(6,7) &  \\
					&& \multicolumn{2}{l|}{$\Sp(1)\+\Sp(1)\+ [1]$} & 
					\checkmark		&
					[1] \+[1] \+\Sp(1) & \Sp(2) \+ [1]
					&
					\yesgreater
					&(4,11) &  \\
					&& 
					\makecell{
						\Sp(1)\+\D_2\Sp(1)	\\ 
						\qquad \mbox{(or }	\Sp(1)\+\SU(2)_2\mbox{)}
					} &&  		
					\Negative &
					[1]\+\SO(2) 
					& H
					&
					\yesequal
					&(4,11) &  \\
					&& \multicolumn{2}{l|}{$\Sp(1)\+\SU(2)_{10}$} &	
					\Negative	&
					\trivialgroup & \Sp(1)\+i_{3,1}\U(1)
					&
					\Smallneg
					&(4,11) &  \\
					&& \multicolumn{2}{l|}{$\SO(3)\.\D_3\Sp(1)$} &
					\Negative		&
					\trivialgroup & \SU(2)\.\D_3\U(1)
					&
					\Smallneg
					&(4,11) &  \\
					\hline
					
					\SO(2k), & k \geq 5 & \multicolumn{2}{l|}{$\SO(2k-2)$} & 
					\checkmark		&
					\SO(2) & H
					&
					\yesequal
					&(2k\mnn-\mn2,2k\mnn-\mn1) & V_2(\R^{2k})	\\
					\Spin(10) && \multicolumn{2}{l|}{$\SU(5)$} &
					\Negative		&
					\U(1)  & H
					&
					\yesequal
					&(6,15) &  	\Spin(9)/\SU(4)	\\
					\Spin(8) && \multicolumn{2}{l|}{$\SU(4)$} &
					\checkmark		&
					\U(1) & H
					&
					\yesequal
					&(6,7) & \Spin(7)/\SU(3)	\\
					\SO(8) && \multicolumn{2}{l|}{$\SO(6)$} &
					\checkmark		&
					\SO(2) & H
					&
					\yesequal
					&(6,7) & \Spin(7)/\SU(3)	\\
					\SO(6)	&&	\multicolumn{2}{l|}{$\SO(4)$} & 
					\checkmark		&
					\SO(2) & H
					&
					\yesequal
					&(4,5)& V_2(\R^6)\iso \widetilde G_2(\C^4) \\
					\hline 
					
					F_4 && \multicolumn{2}{l|}{$\Spin(7)$} 
					& \Negative
					&
					\SO(2) & H
					&
					\yesequal
					&(8,23) &  \\
					&& \multicolumn{2}{l|}{$\Sp(3)$} 
					& \Negative 
					&
					\Sp(1) & H
					&
					\yesequal
					&(8,23) &  \\
					\hline
					\hline
					
					\Sp(k) \x \Sp(2), 
					& k\geq 2 & \multicolumn{2}{l|}{$\Sp(k-1) \x \D\Sp(1) \x \Sp(1)$}
					& \Negative
					&
					\trivialgroup & \Sp(k-1) \x \D\U(1) \x \Sp(1)
					&
					\Smallneg
					&(4,4k\mnn-\mn1) & 
					\\
					\hline
					
					G_2 \x \Sp(2)&& \multicolumn{2}{l|}{$\SU(2)_1 \. \D\Sp(1) \. \Sp(1)$}
					& \Negative
					&
					\trivialgroup & \SU(2)_1 \. \D\U(1) \. \Sp(1)
					&
					\Smallneg
					&(4,11) &  \\
					
					&& \multicolumn{2}{l|}{$\SU(2)_3 \. \D\Sp(1) \. \Sp(1)$}
					& \Negative
					&
					\trivialgroup & \SU(2)_3 \. \D\U(1) \. \Sp(1)
					&
					\Smallneg
					&(4,11) &  \\
				\end{tabular}
			}
		\end{center}
	\end{table}%
\end{landscape}%
}

Evidently \Cref{thm:spheres} is responsible for Step 3.
The classification in the irreducible case in \Cref{sec:cases}
can be extended to the virtually effective case with the material of 
\Cref{sec:enlarge},
and \isotfity in this case is determined with some 
additional information involving~$\revision{\ol H}$ in \Cref{thm:alg}
and the quotient 
$Z_{\revision{\ol G}}(\revision{\ol H})^0 \revision{\ol H}/\revision{\ol H}$.
The considerations in Step 4 may be summarized as follows;
notation is explained in \Cref{rmk:table-explanation0} and 
\Cref{rmk:table-explanation}.


\begin{restatable}{theorem}{isotfclassification}\labell{thm:isotfclassification}
	Let
	\GH
	be a pair of 
	compact, connected Lie groups
	such that $G/H$ 
	has the rational homotopy type of 
	a product $S^n \x S^m$ with $n \geq 3$ odd and $m$ even.
	All irreducible such pairs
	(considered up to covers 
	in the sense of \Cref{def:cover})
	are listed in \Cref{table:Kramer+Wolfrom},
	and the isotropy-formal pairs among them are noted.
	All reducible but virtually effective such pairs
	(up to covers) are obtained
	as \[\Big(\mnn G \x P,(H \x \trivialgroup) \. 
						\big\{(p,p) : p \in P\big\}\mn\Big)\]
	for \GH in \Cref{table:Kramer+Wolfrom}
	and $P$ a rank-one subgroup of $Z_G(H)^0$
	not contained in $H$.
	Existence of such subgroups can be determined from the $Z_G(H)^0$
	column of the table.
	When such a $P$ exists, it is isomorphic 
	to one of $\U(1)$, $\SO(3)$, and $\SU(2)$,
	and the equivalence class of the pair depends only on
	the isomorphism type of $P$.
	Such a pair is \isotf if and only if $H_S > H$,
	as also noted in the table. 
\end{restatable}

\brmk
\revision{As the construction of the associated maximal pair,
the computation of the fundamental group and cohomology ring,
and the construction of 
the effective \revision{quotient}
and simply-connected cover of a pair
are all effectively computable tasks
and the table one consults is finite as well,
\Cref{thm:alg} in principle gives an effective algorithm for
determining whether a \ccpair $(G,K)$ is \isotf.
Thanks go to the referee for asking for this clarification.
}
\ermk

\smallskip

\nd\emph{Acknowledgments.} 
The first author thanks Federico Pasini 
for help with a program evaluating degree multisets 
for rational sphere products $G/H$,
Linus Kramer for verifying his analysis 
did not consider the cases $S^{\mr{even}}\x S^{\mr{odd}}$ with $\mr{even} \geq \mr{odd}$
and discussing the cases in which our results differed,
Oliver Goertsches for comments and suggestions on an early draft,
and Jason DeVito for useful conversations and literature references
on homogeneous spaces, 
especially the dissertation of Kamerich~\cite{kamerich}
and the argument in \Cref{rmk:example-discussion}\ref{rmk:S6S15'}.
The second author thanks Mychelle Parker for letting us use her 
Maple code to study various subalgebras of simple Lie algebras,
and thanks the  Fundamental Research Funds 
for the Central Universities of China (2020MS040, 2023MS078) for their support.
\revision{Both authors would additionally 
like to thank the anonymous referee for a careful 
and insightful reading
resulting in many corrections and clarifications.}

\section{Reductions and characterization}\labell{sec:spheres}

We have set ourselves the task of describing 
all the isotropy-formal pairs $(G,K)$ 
of compact, connected groups with $\rk K = \rk G - 1$.
We say $K$ and the pair $(G,K)$ are of \defd{corank one}.\footnote{\ 
	{N.B.} This terminology differs from Onishchik's usage~\cite[p.~207]{onishchik},
	in which the \emph{corank} is the rational dimension of the cokernel of the map
	$H^{\geq 1} G \to H^{\geq 1} K \epi H^{\geq 1}(K)\,/\,H^{\geq 1}(K)^2$.
}
We now embark on the voyage of reduction described in \Cref{thm:alg}.
As a first step, we may replace $K$ 
by its maximal torus $S$. 

\bthm[%
	{%
	\cite[Thm.~1.1]{carlson2018eqftorus}%
	\cite[Thm.~1.4, Prop.~3.12]{carlsonfok2018}%
	}]%
	\labell{thm:CF}
Let $(G,K)$ be a pair of compact, connected Lie groups
and $S$ a maximal torus of $K$. 
Then \GK is \isotf if and only if \GS is \isotf,
and if and only if either of the following hold:
\bitem
\item 
$G/S$ (equivalently, $G/K$) 
is formal in the sense of rational homotopy theory
and $\H(BG; \Q) \lt \H(BS; \Q)^{N_G(S)}$ is surjective;
\item 
$\H(BG; \Q) \lt \H(BS; \Q)^{N_G(S)}$ is surjective
and 
\revision{the action of the
component group $\pi_0 N_G(S)$ on the Lie algebra ${\fs}$ of $S$ 
induced from the conjugation action of $N_G(S)$ on $S$
 displays $\pi_0 N_G(S)$ as a reflection group.}
\eitem
\ethm

Of course, applying this result twice, we may replace $(G,K)$
in our considerations with any other pair $(G,H)$
such that $H$ shares a maximal torus with $K$.
There is a natural notion of equivalence of pairs:

\bdefn\labell{def:pairEquivalence}
Two pairs of Lie groups $(G_1,K_1)$ and $(G_2,K_2)$ are 
\defd{equivalent} if there is an isomorphism $f\:G_1\lt G_2$ such that 
$f(K_1)=K_2$. 
\edefn

\nd An equivalence induces a diffeomorphism
$G_1/K_1 \lt G_2/K_2$ taking the isotropy $K_1$-action to the isotropy $K_2$-action, so isotropy-formality depends only 
on equivalence classes of pairs.

Since conjugation by an element of $G$ is a Lie group automorphism,
all tori in $G$ lie in a maximal torus, and all maximal tori are conjugate,
we may assume the maximal torus $S$ of $K$ lies in 
some fixed maximal torus $T$ of $G$.
Then we may relate the Weyl group $W_G$ with respect to $T$
and the 
component group $N = \pi_0 N_G(S)$ of \Cref{def:sec1}.

\begin{definition}\labell{def:tN}
Let $G$ be a compact, connected Lie group with 
maximal torus $T$ containing a subtorus $S$.
We write $\defm{\tN}$ for the subgroup 
		$\big(\mnn N_G(S) \inter N_G(T)\mnn\big)/T$ of $W_G$.
\end{definition}

\begin{lemma}[{%
		\cite[Lem.~3.10, 4.3]{carlson2018eqftorus}%
	}]%
	\labell{thm:transition}
	With the notation set in \ref{def:sec1} and \ref{def:tN},
	the adjoint action of~$\tN$ on $\fs$ 
	factors as
$
	\tN \longepi N \longmono \Aut \fs. 
$
\end{lemma}

Thus $N$ can be identified as a subquotient of $W_G$.
If $S$ is a maximal torus of some closed, connected subgroup $K$ of $G$,
then $N_K(S)$ is contained in $N_G(S)$,
so $W_K$, viewed as a subgroup of $\Aut \fs$,
is a subgroup of $N$.
A well-known
characterization~\cite[IV.5.5, XII.3.4]{borel1960seminar} of \eqfity of 
the action of a torus $T$ on a compact space $X$
in terms of the sums of Betti numbers of $X$ and of $X^T$
reduces in our case $T = S$ and $X = G/K$ to the following:

\begin{theorem}[Shiga--Takahashi~%
				{\cite[Sec.~2]{shigatakahashi1995}};
				Goertsches--Noshari~%
				{\cite[Prop.~2.3]{goertschesnoshari2016}}; 
				Carlson~%
				{\cite[Prop.~3.11]{carlson2018eqftorus}} 
				for this phrasing]%
				\labell{thm:dim-criterion}
	With the notation set in \ref{def:sec1},
	we have
	\[
	\dim_\Q \H(G/K) 
		\geq 
	2^{\rk G - \rk K}
		[N:W_K]
		\mathrlap,
	\]
	and the left translation action of $S$ (hence $K$) on $G/K$ 
	is equivariantly formal
	if and only if 
	equality holds.
\end{theorem}

\subsection{Reduction to the regular case}\labell{sec:regular}
%
We now justify the group $H_S$ occurring in 
\Cref{def:sec1} and \Cref{thm:spheres,thm:alg}.

The essential special characteristic of the case 
$\dim S = \dim T - 1$
is that there exists a character $T \lt S^1$ whose kernel is $S$,
and the Lie algebra $\fs$ of $S$ can be 
identified as the kernel of the derivative of this character, 
a functional $\weight\: \ft \lt \R$
on the Lie algebra $ \ft$ of $T$.
We can normalize $\a$ to lie in a fundamental domain for
the Weyl group action as follows.
Conjugating $S$ by an element $w \in N_G(T)$
replaces the functional $\weight$ by $w\. \weight$, 
so that fixing a basis $ \D$ of simple roots for $G$,
we may assume $\weight$ lies in the 
{closed fundamental dual Weyl chamber} in $\ft\dual$
consisting of functionals 
whose invariant inner product
with each $\simple \in \D$ is nonnegative.
Equivalently, we may consider the unique vector 
${v} \in \ft$
such that $-B(v,-) = \a$, where $B$ is an $\Ad(G)$-invariant 
negative-definite bilinear form,
and then our normalization constrains $v$ to lie in the
{closed fundamental Weyl chamber} $\ol C$
of vectors on which each $\simple \in \D$ is nonnegative.

Thus, by \Cref{thm:CF}, isotropy-formality of corank-one pairs $(G,K)$ 
can be determined wholly in terms of the group $G$
and a nonzero vector $v$ in the closed fundamental Weyl
chamber with respect to some maximal torus 
and some basis of simple roots for $G$.
We fix this notation for the rest of this subsection.

\begin{notation}\labell{def:sec2}
	Let $G$ be a compact, connected Lie group 
	and $T$ a fixed maximal torus as in \Cref{def:sec1}.
	We write
\bitem
	\item $\defm \Delta$ for a basis of simple roots for $G$ in 
	the dual $\defm{\ft\dual}$ to the Lie algebra $\ft$ of $T$,
	\item $\defm{\ol C} \subn \ft$ for the closed fundamental Weyl chamber
				$\{u \in \ft : \g(u) \geq 0 \mbox{ for all } \g \in \D\}$,
	\item $\defm v \in \ft$ for a fixed nonzero vector in $\ol C$,	
	\item $\defm {\D_v} \subn \D$ for the set of simple roots annihilating $v$,
	\item $\defm{W_v}$ 
		for the stabilizer of $v$ 
		in 
		$W_G$,
		generated by reflections with respect to $\g \in \D_v$,\footnote{\
		See, \emph{e.g.}, Adams~\cite[Thm.~5.13(vii), notation from Defs.~4.13, 4.38]{adamsLiebook}.}
	\item $\defm B$ for a negative-definite $\Ad(G)$-invariant bilinear form on $\ft$,\footnote{\ This can be taken to be the Killing form if $G$ is semisimple
	and otherwise may be taken as the direct sum of the Killing form
	and the negative of an arbitrary inner product on the center $\f z(\f g)$.}
	\item $\defm \weight \in \ft\dual$ for the functional $-B(v,-)\: \ft \lt \R$,
	\item $\defm \fs = (\R v){}\defm{{}^\perp} < \ft$ 
		for the $B$-orthogonal complement of $\R v$ in $\ft$,
	\item $\defm S = \exp \fs$ for the connected subgroup of $T$
			whose Lie algebra is $\fs$.
\eitem
Although the logical dependency between the defined symbols is different
than in \Cref{def:sec1}, the notations are compatible
when $v$ is chosen such that $\exp \R v$ is closed (hence a circle) in $T$,
in which case $S$ is a complementary codimension-one subtorus of $T$.
\end{notation}

The stabilizer $W_v$ is closely related to the group $\tN$ 
of \Cref{def:tN} and \Cref{thm:transition}.
As the tangent space $\fs < \ft$ is defined to be $\ker \weight = (\R v)^\perp$,
and since $-B$
is $W$-invariant,
it follows an element $w \in W$ stabilizes $\fs$, 
and hence lies in $\tN$,
if and only if it stabilizes $\{\pm v\}$.
In symbols, $\tN = W_{\{\pm v\}}$.

%
%
%

\begin{corollary}\labell{thm:w0}
	With the notation set in \ref{def:sec1}, \ref{def:tN}, and \ref{def:sec2},
we have
\[
\tN 
=
\case{\ang{W_v,w_0} 		&	\mbox{if }	w_0\. v = -v\mathrlap,\\
	W_v					&	\mbox{otherwise}%
	\mathrlap,
}
\quad\qquad\mbox{and hence}\qquad\quad
N 
=
\case{\ang{W_v|_{\fs},w_0|_{\fs}} 		&	\mbox{if }	w_0\. v = -v\mathrlap,\\
	W_v|_{\fs}				&	\mbox{otherwise}%
	\mathrlap.
	}
\]
\end{corollary}


\bpf
Recall~\cite[Thm.~5.13, Cor.~5.16]{adamsLiebook}
that $W_G$ acts simply transitively on the collection of closed Weyl chambers
and that each orbit $W_{\revision{G}}\.v$ meets each closed Weyl chamber 
in precisely one point.
Since $-v$ lies in $-\ol C = w_0 \. \ol C$
and $w_0\. v$ is the unique point of the orbit~$W_G\. v$ 
lying in $w_0\.\ol C$,
it follows that if there exist any elements of $W_G$ 
sending $v$ to $-v$, 
then $w_0$ is among them. 
If so, the index~$[\tN : W_v]$ is $|\tN \. v| = \bigl|\{\pm v\}\bigr| = 2$,
and so $\tN = \ang{W_v,w_0}$;
otherwise, $[\tN : W_v] = \bigl|\{v\}\bigr| = 1$.
\epf

\if and only ifalse
\begin{proof}[Longer proof]
	Recall~\cite[Thm.~5.13]{adamsLiebook}
	that $W_v$ acts simply transitively on the collection of closed Weyl
	chambers containing $v$.
	Any element $w$ of $\tN$ either fixes $v$ or sends it to $-v$,
	which lies in the closed Weyl chamber $-\ol C$,
	so if there is some $w \in W$ for which $w\. v = -v$,
	then for some $w_v \in W_v$ we have $w_v w \. \ol C = - \ol C = w_0 \. C$.
	But since $W$ acts simply transitively 
	on the collection of all closed Weyl chambers,
	this implies $w_v w = w_0$,
	and hence $w_0 \.v  = w_v w \. v = w_v \. -v = -v$.
	Thus if any element of $W$ sends $v$ to $-v$,
	then $w_0$ does so, and if so,
	then $W_v$ is of index $2$ in $\tN = \ang{W_v,w_0}$.
\end{proof}
\fi

\begin{remark}
If we have $[\tN:W_v] = 2$, it does not necessarily follow
that $[N:W_v|_\fs] = 2$;
it can be the case that $w_0\.v = - w$ but $w_0|_\fs$ lies in $W_v|_\fs$.
\end{remark}

We follow Onishchik~\cite[pp.~61--5,~218--9]{onishchik} 
in associating another Lie group containing $S$.

\begin{definition}\labell{def:LH}
	With notation as in \Cref{def:sec1} and \ref{def:sec2},
	let 
	$\defm{S^\perp} \ceq \exp \R v$
	be the one-parameter subgroup of $T$ tangent to $\R v$ 
	at $1 \in T$,
	let
	$\defm{\Kl} \ceq Z_G(S^\perp)$ 
	be its centralizer in $G$,
	write $\defm{[\Kl,\Kl]}$ for the commutator subgroup
	of $\Kl$,
	and set $\defm{H_S} \ceq [\Kl,\Kl]\.S$.
	We call $(G,H_S)$ the \defd{\mrp} associated to $S$.
\end{definition}

The nomenclature is explained by the following definition and result.

\bdefn[Dynkin~\cite{dynkin1952}]\labell{def:regular}
A Lie subalgebra $\fk$ of $\fg$ 
is \defd{regular with respect to} a maximal abelian subalgebra~$\ft$ of $\fg$ 
if $[\ft,\fk]$ lies in $\fk$ 
(or in other words if $\ft$ lies 
in the normalizer $\mathfrak{n}_{\fg}(\fk)$) 
and \defd{regular} if $\mathfrak{n}_{\fg}(\fk)$ has the same rank as $\fg$. 
Equivalently, a closed, connected subgroup $K$ of $G$ is \defd{regular 
with respect to} a maximal torus~$T$ of $G$ 
if $T$ lies in the normalizer $N_G(K)$,
and \defd{regular} if $N_G(K)$ 
is of full rank in $G$.
\edefn


\begin{proposition}\labell{thm:HS}\labell{thm:Wv}
	With the notation set in \ref{def:sec2} and \ref{def:LH},
	we have the following.
\begin{enumerate}
\item\labell{thm:HS-1}
	The Weyl group $W_\Kl$ of $\Kl$ is the stabilizer $W_v$.
\item\labell{thm:HS-2}
	The group $\Hl$
	is the largest closed, connected subgroup
	of $G$ containing $S$ as a maximal torus
	and centralizing $S^\perp$.
	Equivalently, $\Hl$ 
	is the largest closed, connected subgroup of $G$ 
	regular with respect to $T$ and 
	containing $S$ as a maximal torus.
\item\labell{thm:WHS}
	The inclusion $\Hl \longinc \Kl$ 
	induces an isomorphism $W_{\Hl} \isoto W_\Kl$; 
	\emph{i.e.},
	 \hfill $W_{\Hl} = W_\Kl|{}_\fs = W_v|_\fs \iso W_v$.
\item\labell{thm:N}
With the notation set in \ref{def:sec1},
we have \hfill
$
N 
=
\case{
	\ang{W_{\Hl},w_0|_\fs} 		&	\mbox{if }	w_0\. v = -v,\\
	W_{\Hl}				&	\mbox{otherwise}.
}
$

Thus $N \neq W_{\Hl}$ 
holds if and only if both
$w_0\. v = -v$ and $w_0|_\fs \not \in W_{\Hl}$.
\end{enumerate}
\end{proposition}
\begin{proof}
	Writing $\defm\Phi$ for the set of nonzero roots of $G$,
	the adjoint action of $\ft$ on $\fg_\C = \fg\, \ox_\R \C $ has the root space 
	decomposition $\fg_\C = \ft_\C \+ \Direct_{\genwt \in \Phi} \C e_\b$ with coroot vectors 
	$\defm{h_\b} = [e_\b,e_{-\b}] \in i\ft$
	implicitly defined by 
	$\b = 2\frac{B(h_\b,-)}{B(h_\b,h_\b)}$.
	The associated real root space 
	decomposition is thus
\[
\fg = \ft \+ \Direct_{\genwt \in \Phi^+} 
\R\big\{e_\b+e_{-\b}, i(e_\b-e_{-\b})\big\}
\mathrlap.
\]

\bigskip	
\noindent 1.	
	Since $T$ is abelian and contains ${S^\perp}$,
	it centralizes $S^\perp$ and so lies in $\Kl$. 
	If we write $\defm{\Phi_\Kl} \subn \Phi$
	for the set of nonzero roots of $\Kl$, 
	then the complexified Lie algebra ${\f l}_\C$ of $\Kl$,
	as a $\ft$-representation, 
	decomposes as $\ft_\C \+ \Direct_{\genwt \in \Phi_\Kl} \C e_\genwt$.
	To determine $\Phi_\Kl$, 
		note that since $\f l$ centralizes $v$, 
	for each $\genwt \in \Phi_\Kl$ we have $\genwt(v) = 0$.
	The positive roots $\Phi^+$ are 
	$\N$-linear combinations of the simple roots $\d \in \D$,
	and $\d(v) \geq 0$ for each simple root because $v \in \ol C$,
	so $\genwt \in \Phi^+$ lies in $\Phi_\Kl$
	if and only if it is an $\N$-linear combination of 
	roots $\d \in \D_v$ vanishing on $v$.
	Thus we have $\Phi_\Kl = \Phi \inter \Z \D_v$.
	But then $\WK$ is the subgroup of $W_G$ 
	generated by reflections through the hyperplanes
	$\ker \g$ for $\g \in \D_v$,
	namely $W_v$.

\bigskip

\noindent 2.
Since $\Kl$ and $T$ centralize $S^\perp$, 
so evidently does $H_S = [\Kl,\Kl]\. S \leq \Kl$.
On the Lie algebra level, we have 
$
\fkl 
	= 
\ft \+ \Direct_{\genwt \in \Phi^+_\Kl} 
		\R\big\{ e_\b+e_{-\b}, i(e_\b-e_{-\b})\big\}
$,
so since $\big[i(e_\b-e_{-\b}),e_\b + e_{-\b}\big] = 2 ih_\b$
and $[h_\b,e_\b] = \b(h_\b)e_\b = 2e_\b$,
the derived subalgebra of $\fkl$ decomposes as
\[
[\fkl,\fkl]
 	=
	\ft_{[\fkl,\fkl]} 
	\+ 
\Direct_{\genwt \in \Phi^+_\Kl} \R\big\{e_\b+e_{-\b}, i(e_\b-e_{-\b})\big\}
\mathrlap,
\] 
where the Cartan algebra 
$\defm{\ft_{[\fkl,\fkl]}}$
is spanned by $i[e_{\b},e_{-\b}] = i h_\b \in \ft$ for $\b \in \Phi^+_\Kl$.

For $\b \in \Phi_\Kl$, since $\b(v)=0$, 
we have $h_\b\in v^\perp=\fs$, so that $\ft_{[\fkl,\fkl]} \leq \fs$. 
Since $[\fs,\f t_{\revision[\fkl,\fkl\revision]}] \leq [\fs,\fs] = 0$
and $[\fs,\R e_\b] \leq [\ft,\R e_\b] = \R e_\b$ for each $\b \in \Phi_\Kl$,
the root space decomposition of $[\fkl,\fkl]$
shows $\big[\fs, [\fkl,\fkl]\big] \leq [\fkl,\fkl]$,
meaning $[\fkl,\fkl] + \fs$ is a Lie subalgebra of $\fkl$ with Cartan algebra $\fs$. 
So $H_S = [\Kl,\Kl]\. S$ is a Lie group with maximal torus $S$
and hence is a proper, closed, connected Lie subgroup of $\Kl$.
	To see $H_S$ is the largest subgroup containing $S$ as
	a maximal torus and centralizing $S^\perp$,
	suppose $\Hlp$ is another such subgroup. 
	By definition, $\Hlp$
	lies in the centralizer $\Kl$ of $S^\perp$,
	so $[\Hlp,\Hlp]$ lies in $[\Kl,\Kl]$.
	As $S$ is a maximal torus of $\Hlp$, 
	it contains the identity component~$\defm{Z(\Hlp)^0}$ of the center~$\defm{Z(\Hlp)}$, 
	giving a containment
	$\Hlp =
	[\Hlp,\Hlp]Z(\Hlp)^0 \leq
	\Hl = [\Kl,\Kl]S$.
	
	To see $H_S$ can be equivalently characterized
	as the largest closed, connected subgroup of $G$ 
	regular with respect to $T$ and with maximal torus $S$, 
	let $H'$ be any  
	closed, connected subgroup of $G$ containing $S$ as a maximal torus. 
	If $H'$ centralizes $S^\perp$, 
	then both $S$ and $S^\perp$ normalize $H'$,
	and hence so does $T=SS^\perp$. 
	Conversely, if $T$ normalizes $H'$, 
	then the adjoint representation of~$\fg$
	restricts to a $\ft$-representation on $\fh'$. 
	The complexification of this representation decomposes as the 
	direct sum of $\fs_\C$ and weight spaces $\C e_\b$.
	Taking $h_\b \in \fs$ 
	as above, 
	$[e_\b,v]$ is a multiple of $B(h_\b,v)e_\b = 0$,
	and since $\fh'_\C$ is spanned by $\fs$ and the $e_\b$,
	it follows that $H'$ centralizes $S^\perp$.

\bigskip

\noindent 3.
	As $\Kl = Z_G(S^\perp)$ and $T = SS^\perp$,
	we have $N_\Kl(S) = N_\Kl(T)$.
	As $T$ is a maximal torus in $\Kl$,
	it in particular contains the identity component 
	$Z(\Kl)^0$ of the center of $\Kl$,
	and hence $\Hl \leq \Kl = [\Kl,\Kl]Z(\Kl)^0 \leq [\Kl,\Kl] T = \Hl T$.
	Since $T$ evidently normalizes $S$,
	we get $N_\Hl(S) \leq N_\Kl(S) = N_\Kl(T) \leq N_\Hl(S)T$;
	thus $N_\Hl(S)T = N_\Kl(T)T = N_\Kl(T)$.
	It follows
	\[
			W_{\Hl} 
		= 	\frac{N_{\Hl}(S)}	S 
		\lt \frac{N_{\Hl}(S)T}	T 
		= 	\frac{N_{\Kl}(T)}		T 
		= 	W_{\Kl}
	\] 
is surjective,
and since $T$ acts trivially by conjugation on $S$,
the injection $W_{\Hl} \longmono \Aut\, \fs$ factors through $W_{\Kl}$,
showing $W_{\Hl} \lt W_{\Kl}$ is injective.

\bs

\nd 4.
	The displayed equations follow from \Cref{thm:w0} and
	Items~\ref{thm:HS-1} and \ref{thm:WHS}.
	The new inequality comes from the fact  
	$w_0|_{\fs}$ may lie in $W_{H_S}$
	even if $w_0$ does not lie in $\WK$.
\end{proof}

Recall that in the statement of \Cref{thm:spheres},
we needed to understand the cohomology of $G/K$. 
We set yet more notation to facilitate 
a cohomological characterization of formality 
for homogeneous spaces.
\begin{notation}\labell{def:HG}
	Given a map $f\: B \lt A$ of graded $\Q$-algebras, we write
	\[
		\smash{A \defm\quot B
			\ceq 
		A\, \ox_B\, \Q} 
			\,=\, 
		\quotientmed{A\,}{\,f(B^{\geq 1})\. A} 
	\mathrlap.
	\]
	Given a group $W \leq \Aut A$ of algebra automorphisms,
	we let $\defm{A^W} \leq A$ denote the subalgebra of $W$-fixed elements.
	We write $\defm{\H_\G} \ceq \H(B\G;\Q)$ for $\G$ a topological group.
\end{notation}

\bthm[{\cite[p.~211]{onishchik}}]\labell{thm:formal}
If $(G,H)$ is a pair of compact, connected groups, then $G/H$
is formal if and only if
\[
\H(G/H) \,\iso\, \qquotientmed{\H_{H}\,}\HG \,\ox\, \Lambda \wP,
\]
where $\Lambda \wP \iso \im\!\big(\mn\H(G/H) \to \H G\big)$
is an exterior algebra on $\rk G - \rk H$ generators of odd degree 
and $\H_{H} \quot \HG$ is a complete intersection ring.
\ethm

\nd Thus, by \Cref{thm:CF}, to assess isotropy-formality of $(G,H)$, 
it will help to understand 
$\H_{H} \quot \HG$.

\begin{proposition}\labell{thm:Heven}
	For a maximal regular pair $(G,H)$ 
	with maximal tori $(T,S)$,
	we have isomorphisms
\[
		\Heven(G/H) 
		\,\iso\, 
		\H_{H} \quot \HG
		\,\iso\,
		(\HS)^{W_{H}} \quot \im(\HG \to \HS).
\]
\end{proposition}
\begin{proof}
\revision{
    First note that there is a finite covering
    $T/S \longepi L/H$ 
    where $L = \Kl$ and $H = [L,L]\.S$ are as in \Cref{def:LH}.
    This map arises because,
    as noted in the proof of \Cref{thm:HS}.\ref{thm:WHS},
    we have $L = [L,L]\.T = H T$,
    and hence
    $T \inc L \epi L/H$ is surjective
    with kernel $T \inter H$.
    This kernel is compact and properly contained in $T$,
    so that its rank equals that of $S$,
    and hence by compactness the index $\defm k$ of $S$ is finite,
    so that 
    $T/S \lt L/H$ is a degree-$k$ map of circles.
    }

Consider the Serre spectral sequence $({E_r},{d_r})$ 
associated to the principal circle bundle
\[
L/H \lt G/H \lt G/L\mathrlap. 
\]
Since this spectral sequence's $E_2$ page 
has only its $0\th$ and $1^{\mr{st}}$ rows nonzero,
it collapses at $E_3$. 
\revision{Moreover since $L_S$ is of full rank in $G$,
the cohomology groups of $G/L_S$ are concentrated in even degree,
so the spectral sequence}
can be equivalently understood as the exact (Gysin) sequence
of graded vector spaces
\begin{equation}\labell{eq:Gysin}
0 \to \Hodd(G/H){[-1]}\lt \H(G/L) \os{\. e}\lt H(G/L) \lt \Heven(G/H) \to 0\mathrlap,
\end{equation}
where the class $\defm e$ is $d_2\big(1 \ox\, [L/H]\big)$.
Particularly, 
the ring $\Heven(G/H)$ is isomorphic to the quotient of $\H(G/L)$
by the ideal $(e)$.

The covering spaces $W_L \to BT \to BN_L(T)$
and $W_H \to BS \to BN_H(S)$
induce isomorphisms
$\H_{L} \isoto (\HT)^{W_{L}}$ 
and $H_{H} \isoto (\HS)^{W_H}$~\cite[Prop.~3G.1]{hatcher}.
    We claim the class $e$ is the Euler class 
	of the bundle $G/H \to G/L$,
    \revision{and also, up to a scalar multiple,}
	the image under $(\HT)^{W_L} \simto \H_{L} \to \H(G/L)$ 
	of the Euler class $\defm{\widetilde e}$
	of the circle bundle $BS \lt BT$.
	This follows from the commutative diagram
	\[
	\xymatrix@C=.5em@R=.75em{
		&	&BS\ar[dr]\ar[dd]|\hole\ar[rr]	&&	E(T/S) \ar[dd]|\hole\ar[dr]
		\\
		&G/H\ar[rr]\ar[dd]	&	
        &BH\ar[dd] \ar[rr]&& E(L/H)\ar[dd]	
		\\
		&	&BT\ar[dr]\ar[rr]|(.425)\hole &&B(T/S)\ar[dr]		\\
		&G/L\ar[rr]	&	&BL \ar[rr]&& B(L/H),
	}
	\]
	of principal bundles.
Here the vertical arrows are the bundle projections
\revision{of $L/H$-bundles in front, and of $T/S$-bundles in back},
the front horizontal maps \revision{in the left square}
can be seen as
	$gH \lmt xgH$ and $gL \lmt xgL$ for a fixed $x \in EL$,
	the maps of classifying spaces are up to homotopy 
	those functorially induced from 
    \revision{inclusions and quotient projections of groups},
	the bottom edge of the back square is also
	the classifying map $BT \lt B(T/S)$ of the 
	principal bundle $T/S \to BS \to BT$,
    and similarly for the edge $BL \lt B(L/H)$ of the front square.
    \revision{The maps $BS \lt BH$ and 
    $E(T/S) \lt E(L/H)$ induce maps
    of principal bundles,
    equivariant with respect to the covering map $T/S \lt L/H$.}

    The map of spectral sequences induced by the front \revision{left} square 
	shows that under $G/L \lt BL$,
    the class $\defm e = d_2\big(1 \otimes {[L/H]}\big)$ \revision{(name reused)}
	indeed does pull back to $e = d_2\big(1 \otimes {[L/H]}\big)$.
\revision{The spectral sequence map induced by the front right square 
	identifies $e$ 
	as the transgression of the class $1 \otimes [L/H]$
    in the $E_2$ page $\H_L \ox \H(L/H)$.}
	The map of spectral sequences induced by the back square 
	identifies $\widetilde e$ 
	as the transgression of the class $1 \otimes [T/S]$
	in the $E_2$ page $\HT \ox \H(T/S)$.
\revision{Finally, the map of spectral sequences induced by the right face
    sends $1 \otimes [L/H]$ to $k \otimes [T/S]$,
    and hence the transgression of the former is sent to
    $k$ times the transgression of the latter.
    Since the bundle square commutes,
    this means $\H_L \lt \H_T$ sends $e$ to $k\widetilde e$,
    so the identification $\H_L \isoto (\HT)^{W_L}$
    shows $\widetilde e$ lies in $(\HT)^{W_L}$.
    }
	
    The ideal $(\revision{\widetilde e})$ is the kernel of the surjection $\HT \lt \HS$,
	which can be seen either directly by observing 
	it generates the kernel of
	$H_T^2 \lt H_S^2$
	,
	or else from the \SSS of $T/S \to BS \to BT$.\footnote{\
		If we make the identifications $H^2(BT;\R) \iso \ft\dual$ 
		and $H^2(BS;\R) \iso \fs\dual$, 
		then $\widetilde e$ can be identified with $\weight$, 
		which is another way of showing it is $W_L$-invariant.
		}
As the inclusion $H \longinc L$
induces a restriction isomorphism $W_{L} \isoto W_H$ 
by \Cref{thm:HS}.\ref{thm:WHS},
the identifications yield $(\HT)^{W_L}/(\widetilde e) \iso (\HS)^{W_H}$,
for any $x + (\te) \in \HS = \HT/(\te)$ 
invariant under the action of $W_L = W_H$ 
has $x \equiv w\.x \pmod \te$ for all $w \in W_L$,
and hence $\frac 1{|W_L|} \sum w\.x \in (\HT)^{W_L}$ represents $x + (\te)$.
As~$L$ is of full rank in $G$,
the \SSS of $G/L \to BL \to BG$ collapses 
(the~$E_2$ page is concentrated in even degree),
yielding the standard isomorphisms
$\H(G/L) \iso \H_L \quot \HG = (\HT)^{W_L} \quot (\HT)^{W_G}$.
Thus, as hoped,
\eqn{
	\Heven(G/H) 
	\ \iso\  
	\frac{
		(\HT)^{W_L} 
		}
		{\big(\te, (H_T^{\geq 1})^{W_G}\big)}
	&\ \iso\  
	\qquotientmed{(\HT)^{W_L}/(\te)\,}{\,\im\!
	\big(\mnn(\HT)^{W_G} \inc \HT \epi \HS\big)}
	\\
	&\ \iso \ 
	\qquotientmed{(\HS)^{W_H}\,}{\,\im(\HG \to \HS)}
	\mathrlap.
\qedhere
}
\end{proof}

Now we can prove the target result of this subsection.


\begin{proof}[Proof of \Cref{thm:spheres}]
The final statement follows immediately from \Cref{thm:Wv}.\ref{thm:N}.
If $G/K$ is a rational cohomology sphere of odd dimension,
then $\dim_\Q \H(G/K) = 2$;
in that case, we see from \Cref{thm:dim-criterion}
that $|N|/|W_K| = 1$ and $(G,K)$ is \isotf.
On the other hand,
if $G/K$ has the rational cohomology of a product of spheres, 
one even- and one odd-dimensional,
then evidently $\dim_\Q \H(G/K) = 4$.
Thus in this case, by \Cref{thm:dim-criterion},
the pair $(G,K)$ is \isotf if and only if $[N:W_K] = 2$,
and in particular $N$ is not isomorphic to $W_K$.

For the converse implication \revision{of clause \ref{item:converse}},
write \revision{$H = \Hl$} and
assume $(G,H)$ is isotropy-formal;
we must show $G/H$ has the cohomology
of an odd-dimensional sphere or the product 
of an odd- and an even-dimensional sphere.
By \Cref{thm:CF},
it is formal,  
so by \Cref{thm:formal},
its cohomology is of the form
$(\H_{H} \quot \HG) \ox \ext\wP$
for $\wP$ a one-dimensional vector space graded in odd degree
(the cohomology of an odd-dimensional sphere)
and $\H_{H} \quot \HG$ a complete intersection ring.
\revision{
Thus $\dim_\Q \H(G/H) = 2\. \dim_\Q (\H_H \quot \HG)$.
By \Cref{thm:dim-criterion}, 
we also know $\dim_\Q \H(G/H) = 2[N:W_H]$,
so $\dim_\Q (\H_H \quot \HG) = [N:W_H]$,
which by \Cref{thm:HS}.\ref{thm:N} is either $1$ or $2$.
Thus $\H_H \quot \HG$ has the cohomology either of a point or
of an even-dimensional sphere.
}
\end{proof}

\begin{remark}
   \revision{ We have the anonymous referee to thank for pointing us 
    to the current proof of the converse direction,
    which at one point we had found independently but failed to implement.
    An older version runs as follows.}
Note $\H_{H} \quot \HG$ is $\Heven(G/H)$,
which by \Cref{thm:Heven} is isomorphic to
$(\HS)^{\WH} \quot \im(\HG \to \HS)$.
From the assumption $(G,H)$ is \isotf and \Cref{thm:CF} again,
the image of $\HG \lt \HS$ is $(\HS)^N$ and $N$ acts on $\fs$
as a reflection group,
so the Chevalley--Shepherd--Todd theorem~\cite[Thm.~7.1.4]{neuselsmith}
implies $(\HS)^N$ is a polynomial ring as well.
By \Cref{thm:formal},
then, $(\HS)^\WH\quot (\HS)^N ={\H_{H} \quot \HG}$
is a complete intersection ring, meaning 
$(\HS)^\WH$ too is free over $(\HS)^N$.
Since by a result of Chevalley~\cite[Thm.~7.2.1]{neuselsmith}
$\HS$ is of rank $|N|$ over $(\HS)^N$ 
		and of rank $|W_H|$ over $(\HS)^{W_H}$,
it follows that 
\[
\dim_\Q \Heven(G/H) = \rk_{(\HS)^N} (\HS)^{W_H} = |N|/|W_H|
\mathrlap.\]
But by \Cref{thm:Wv}.\ref{thm:N}, this is $1$ or $2$.
\end{remark}

This theorem distinguishes an especially interesting class
of transitive actions. 
\begin{definition}
	We call a Lie group pair $(G,H)$
	a \defd{\rspp}
	if $G/H$ has the rational cohomology of
	a product
	$S^{\mr{even}} \x S^{\mr{odd}}$ 
	and a \defd{\rsp}
	if $G/H$ has the rational cohomology
	of $S^{\mr{odd}}$. 
	In both cases, we call \GH a \defd{\rsppp}.
\end{definition}
\begin{remark}\labell{rmk:rational-homotopy-1}
	\revision{%
	It is well known that 
	when $G/H$ is simply-connected, 
	any homogeneous space 
	having the rational cohomology ring of a product of spheres 
	(and more generally, any for which the pure Sullivan model
	$\H(BH) \ox \H(G)$ is formal)
	is rationally homotopy equivalent to one.
	}
	Because the vocabulary differs
	when $G/H$ is not simply-connected, we have opted not to speak in terms
	of rational homotopy type in the general case 
	(but see \Cref{rmk:rational-homotopy-2}).
\end{remark}

In our new terminology, \Cref{thm:spheres}
says that to understand \isotfity in the corank-one case,
we need only examine {\rsppp}s which are also {\mrp}s.
The two notions are closely related.

\blem\labell{thm:dichotomy}
Let \GH be a \rspp.
If there exists a closed, connected, regular subgroup $H^- \leq H$
sharing the maximal torus $S$ with $H$,
and $H/H^-$ is an even-dimensional rational homotopy sphere,
then either $(G,H_S)$ is a \rsppp or $H_S = H^-$.
\elem
\bpf
Consider the bundle $H/H^- \to G/H^- \to G/H$,
which admits a bundle map to $H/H^- \to BH^- \to BH$.
As the \SSS of the latter bundle is concentrated
in even degree, it collapses at $E_2$,
so the map of spectral sequences implies 
that of the former bundle does as well;
we conclude $\dim_{\Q}\H(G/H^-) = \dim_\Q(H/H^-)\.\dim_\Q(G/H) = 8$.
Considering instead the bundle 
$H_S/H^- \to G/H^- \to G/H_S$,
we see that if $H_S > H^-$, 
then $\dim_\Q \H(G/H_S)$ is $2$ or $4$.
\revision{%
Since $\dim G/T$ and $\dim H_S/S$ are even
and $\dim T/S = 1$, it follows $\dim G/H_S$ is odd,
and hence by Poincar\'{e} duality the fundamental class 
$[G/H_S]$ is the product of an odd-dimensional class
and its even-dimensional Poincar\'{e} dual.
It follows $(G,H_S)$ is a \rsppp.
}
\epf

\blem\labell{thm:isotropyOddSphere}
Any \rsp \GH is \isotf.
The associated \mrp $(G,H_S)$ is a \rsppp,
and $H = H_S$ if and only if $H$ is regular.
\elem
\bpf
Isotropy-formality of an \rsp is well-known and already shown
in the proof of \Cref{thm:spheres}.
In that case \Cref{thm:spheres} implies $(G,H_S)$ is a \rsppp.
If $H$ is irregular, then by definition $H \neq H_S$.
On the other hand, if $H$ is regular, it is contained in $H_S$,
and we have a fiber bundle $H_S/H \to G/H \to G/H_S$
which by the argument of \Cref{thm:dichotomy} gives
\[2 = \dim_\Q \H(G/H) = \dim_\Q \H(H_S/H) \. \dim_\Q \H(G/H_S)\mathrlap. \]
As 
$\dim_\Q \H(G/H_S) \geq 2$,
we must have $\dim_\Q \H(H_S/H) = 1$.
But then as $H_S$ and $H$ are of equal rank,
we have $\chi(H_S/H) = |W_{H_S}|/|W_H| = 1$, 
implying $H_S = H$.
\epf

We can use this to show that if the 
isotropy group of a corank-one pair
properly contains that of a \rspp, it is \isotf.

\blem\labell{thm:extend-H}
Suppose \GH is a \rspp.
If there exists a closed, connected subgroup $K > H$ of $G$
sharing the same maximal torus $S$,
then $K/H$ is an even-dimensional sphere,
\GK is a \rsp, 
and \GH is isotropy-formal.
\elem
\bpf
Consider the bundle $K/H \to G/H \to G/K$;
as in the proof of \Cref{thm:isotropyOddSphere},
we have $4 = \dim_\Q \H(G/H) =  \dim_\Q \H(K/H)\.\dim_\Q \H(G/K)$.
Since $K > H$ is of equal rank, we have $\dim_\Q \H(K/H) \geq 2$,
and since $\rk G - \rk K = 1$,
we have $\dim_\Q \H(G/K) \geq 2$,
so we must actually have equality in both cases.
Thus $K/H$ has the rational cohomology an even-dimensional sphere, and hence by a result of Borel~\cite[Thm.~IV]{borel1949remarks} is an even-dimensional 
sphere, and $(G,K)$ is a \rsp.
Hence by \Cref{thm:isotropyOddSphere},
$(G,K)$ and \GH are \isotf.
\epf 

In fact, we do not need the quotient rational homotopy sphere $G/H$
to be odd-dimensional.

\blem\labell{thm:isotf-sphere}
If $G/H$ is a rational cohomology sphere, then $(G,H)$ is \isotf.
\elem
\begin{proof}
	We have addressed the odd-dimensional case in \Cref{thm:isotropyOddSphere}.
	The even-dimensional case is well-known and follows from 
	\Cref{thm:dim-criterion}
	since $\chi(G/H) = 2$ implies $\rk G = \rk H$,
	so that $T = S$ and $N = W_G$,
	and then $2 = \chi(G/H) = |W_G|/|W_H| = |N|/|W_H|$.
\end{proof}

We will need one more lemma of this type 
for our proof of \Cref{thm:isotfclassification}.

\bcor\labell{thm:-Id}
Let \GH be a regular \rspp such that
the longest word $w_0^G$ of $W_G$ acts as $-{\id}$ on~$\ft$.
Then \GH fails to be \isotf 
if and only if
the longest word $w_0^{\smash{H}}$ of $H$ acts as $-{\id}$ on~$\fs$
and
$H = H_S$ in the notation of \Cref{def:LH}.
\ecor
\bpf
If $H_S > H$, then \GH is \isotf by \Cref{thm:extend-H}.
Thus we assume $H_S = H$.
Since $w_0^G$ acts on $\ft$ as $-{\id}$, in particular $w_0^G \. v=-v$
and by \Cref{thm:w0}, we have 
$N = \ang{W_{v}|_\fs,w_0^G|_{\fs}} = \ang{W_{v}|_\fs,-{\id}_{\fs}}$. 
If $-{\id}_{\fs} \notin W_H$, then we must have $N > W_H$,
so \GH is \isotf by \Cref{thm:spheres}.
On the other hand,
by \Cref{thm:HS}.\ref{thm:WHS},
we have $W_H = W_v|_\fs$,
so if $-{\id_{\fs}} \in W_H$, then $N = W_v|_{\fs} = W_{H_S}$
and $(G,\revision{H_S})$ is by \Cref{thm:spheres} not \isotf,
so \GH is not either.
But by the argument of \Cref{thm:w0},
applied to an interior point $v'$ of $\ol C$,
the only element of $W_H$ that might 
act as $-{\id_{\fs}}$ is $w_0^{H}$.
\epf

\subsection{Covers}\labell{sec:cover}
It will be convenient to be able to 
make replacements of pairs ensuring $\pi_1(G/H)$ is free abelian.
\revision{Recall that the fundamental group of a Lie group is abelian,
so that $\pi_1 G$ is abelian and if $H$ is connected, then so is $\pi_1(G/H)$.
In this section we will frequently reason in terms of the rational 
vector spaces $\defm{\pi_*(-;\Q)} \ceq \pi_*(-) \ox \Q$.}

%


\bdefn\labell{def:cover}
We call a pair $(\tG,\tH)$ of compact, connected Lie groups 
a 
\defd{cover} of 
another such pair $(G,H)$
if there exists a finite covering map $q\: \tG \longepi G$ such that 
$\tH$ is the identity component $q\-(H)^0$ 
of the preimage $q\-(H)$.
We call a cover \defd{equal-sheeted} if $\tH = q\-(H)$. 
\edefn

This definition is rigged to 
induce
a covering $\tG/\tH \lt G/H$, which in the equal-sheeted case is a diffeomorphism.
Coverings preserve and reflect all properties we are interested in.

\bprop%
\labell{thm:cover-isotf}
A compact, connected Lie group pair 
$(\tG,\tH)$ covering a compact, connected pair $(G,H)$
is \isotf if and only if $(G,H)$ is,
is (maximal) regular if and only if $(G,H)$ is,
and 
is a \rsppp (respectively) if and only if \GH is.
The projection $\tG/\tH \lt G/H$ 
induces isomorphisms of all rational homotopy groups.
\eprop
\begin{proof}
	The statement about \isotfity
	appears in earlier work of one of the authors%
	~\cite[Thm.~1.2]{carlson2018eqftorus}.
For the statement about {\rsppp}s,
another lemma in that work~\cite[Prop.~3.12]{carlson2018eqftorus} 
observes that  one has $\H(\tG/\tH) \iso \H(G/H)$.
The statement about regularity follows because it can be expressed
in terms of the common Lie algebras $\fg,\f h,\f t$
of $G,H,T$ and $\tG,\tH,q\-(T)^0$,
and that about maximality because closed, connected subgroups are
in bijection with Lie subalgebras.
The isomorphism of rational homotopy groups follows
from the long exact homotopy sequences of 
the bundles $\tH \to H$,\,\, $\tG \to G$,\,\, $\tG \to \tG/\tH$, and $G \to G/H$.
\end{proof}

We will therefore only consider pairs up to covers.
\revision{For this, it will be useful to 
recall some fundamental results in the structure theory 
of compact Lie groups we will use throughout the rest of the work,
so we now make a brief digression to rehearse 
a sufficiency of this theory.
}

\begin{definition}\labell{def:virtual-product}
	We say a compact Lie group is \defd{simple} if it contains 
	no nontrivial \emph{connected}, 
	proper, normal subgroups.
	If $H$ embeds as a normal subgroup of each of two groups $K_1$ and $K_2$,
	we write $\defm{K_1 \ox_H K_2}$ for the \emph{balanced product},
	the quotient 
	of $K_1 \x K_2$ by the equivalence relation setting
	$(k_1 h,k_2) \sim (k_1,hk_2)$ for $(k_1,k_2) \in K_1 \x K_2$ and $h \in H$.
	A \defd{virtual direct product} of two Lie groups $K_1$ and $K_2$,
	written $G = \defm{K_1 \. K_2}$, is a Lie group $G$ 
	of the form $K_1 \ox_F K_2$ 
	for some finite subgroup $F$ central in $K_1$ and $K_2$.\footnote{\ 
		In this case the traditional terminology is that
		$G$ is \emph{locally isomorphic to} $K_1 \x K_2$,
		a sensible wording since this means the Lie algebras are isomorphic.
	}
	More generally a \defd{virtual direct product} 
	of two topological spaces $X_1$ and $X_2$
	is any space finitely and centrally 
	covered by the direct product $X_1 \x X_2$.
	The definition generalizes to virtual direct products 
	of any finite number of factors.
\end{definition}

\begin{theorem}\labell{thm:Lie-structure}
	Every compact, connected Lie group ${G}$ is 
	a virtual direct product of the central torus~$Z(G)^0$ 
	and the commutator subgroup $[G,G]$,
	which is in turn the virtual direct product of
	finitely many simple groups.
	If ${H} \ideal {G}$ is a connected normal subgroup, then $G$ is the virtual
	direct product of~$H$ and a \defd{complement}~$P$, 
	meaning there is a compact, connected 
	subgroup $P \ideal {G}$
	such that $\dsp G = HP \iso \smash{H \ox_{H \inter P} P}$.
	Particularly, $P$ centralizes $H$ and $H \inter P$ is a finite central subgroup
	of both $H$ and $P$.
\end{theorem}

In the case $\pi_1$ is infinite, 
\revision{we have the following.}

\begin{theorem}\labell{thm:isotf-circle-factor}
	Let a \ccpair \GH of Lie groups be such that $\pi_1(G/H)$ is infinite
	and $\H(G/H) \iso \H(S^1 \x S^\ell)$ for some $\ell \geq 2$.
	Then \GH finitely covers
	a pair of the form $(G',H') \x (S^1,\trivialgroup)$
	where $G'/H'$ is a rational cohomology $S^\ell$,
	and hence is \isotf.
\end{theorem}
\begin{proof}
\revision{
    Note from the cohomology that we must have $\pi_1(G;\Q) \iso \Q$ one-dimensional,
    and that the maps $Z(G)^0 \longinc G$
    and $G \longepi G\ab \ceq G/[G,G]$ induce
    isomorphisms on $\pi_1(-;\Q)$.
    Hence $\dim_\Q \pi_1(G;\Q)$ is equal to the dimension 
    of the torus $G\ab$, and 
    the (toral) image $\defm{\ol S}$ of $H \to G \to G\ab$ is of dimension one less.
    Select a circle $\defm C \leq Z(G)^0$
    whose image $\defm{\ol C}$ under 
    $Z(G)^0 \inc G \epi G\ab$ is complementary to $\ol S$,
    so that $\ol C \. \ol S = G\ab$ and $\ol C \inter \ol S$ is finite
    and $\ol C \inc G\ab \epi G\ab/\ol S$ is a finite covering.
    The map $C \longepi \ol C$ is also a finite covering since
    $Z(G)^0 \lt G\ab$ is.
    Write $\defm{\wh S}$ for the identity component
    of the preimage of $\ol S$ under $Z(G)^0 \inc G \epi G\ab$
    and set $\defm{G'} \ceq \wh S \. [G,G]$.
    Then the image of $G' \to G \to G\ab$ is $\ol S$,
    so the image of $G'\.C \to G \to G\ab$ is all of $\ol S\.\ol C = G\ab$,
    and hence $G'\. C = G$.
    The intersection $F = G' \inter C$ of the two 
    is also the kernel of the surjection
    $C \inc Z(G)^0 \epi G\ab \epi G\ab/\ol S \simto \ol C / (\ol C \inter \ol S)$ 
    and hence is finite since $\ol C \inter \ol S$
    and the kernel of $Z(G)^0 \lt G\ab$ are.
    Of course, since $H = Z(H)^0 \. [H,H] \leq \wh S \. [G,G] = G'$,
    the intersection $\defm{F_H} \ceq H \inter C \leq F$ is also finite.}

    \revision{
    Since $G$ is isomorphic to $G' \ox_F C$,
    } 
	we see \GH covers $(G'/F \x C/F,\, H/F_H \x \trivialgroup)$.
	The latter has homogeneous quotient $G'/FH \x C/F$.
	Evidently $C/F$ is a circle, 
	and since covers are rational cohomology isomorphisms,
	we see $G'/FH$ is a rational cohomology $S^\ell$.
	The old pair is \isotf if and only if the new pair is
	by \Cref{thm:cover-isotf}.
	For the new pair, the homotopy orbit space
	$\smash{\big(\frac{G'/F}{H/F_H} \x \frac{C/F}{\trivialgroup}\big)_{H/F_H \x \trivialgroup}}$ of the isotropy action
	is $\smash{\big(\frac{G'/F}{H/F_H}\big)_{H/F_H} \x {C/F}}$,
	so the new pair is \isotf if and only if 
	$(G'/F,H/F_H)$ is,
	but this is always the case
	by \Cref{thm:isotf-sphere}.
\end{proof}
\begin{remark}\labell{rmk:rational-homotopy-2}
	This implies the converse to Bletz-Siebert's Corollary 6.3.2~\cite{bletzsiebert2002},
	which states that if \GH is a \ccpair of Lie groups with
	$\pi_*(G/H;\Q) \iso \pi_*(S^1 \x S^\ell;\Q)$,
	then $\H(G/H) \iso \H(S^1 \x S^\ell)$.
	With \Cref{rmk:rational-homotopy-1},
	this shows that for such a pair,
	$G/H$ has the rational cohomology of the product of (one or) two spheres
	if and only if it has the rational homotopy of such a product. 
\end{remark}

If $\pi_1$ is finite, we may assume for our purposes it is zero by applying the following:

\bprop \labell{thm:cover-constr}
Let $(G,H)$ be a pair of compact, connected Lie groups.
Then all connected finite covers of $G/H$ are 
realized as $\tG/\tH$ for some compact, connected pair $(\tG,\tH)$ such that $\tH \iso H$.
\eprop
\bpf
Considering the long exact sequence of the fibration $H \os{\defm i}\to G \to G/H$,
since $H$ is path-connected, 
we see $\pi_1(G/H) \iso \coker\pi_1(i)$.
Since $G$ is a topological group, $\pi_1(G/H)$ is then abelian,
so connected finite covers of $G/H$ correspond bijectively
to finite-index subgroups $\G \leq \pi_1(G)$ containing $\im \pi_1(i)$,
which also correspond to connected finite covers $\tG \lt G$
along which $i$ lifts.
If we write $\tH \iso H$ for a lift, then $(\tG,\tH)$ covers $(G,H)$
and the image of $\pi_1(\tG/\tH) \lt \pi_1(G/H)$ is isomorphic to $\G/\mn\im \pi_1(i)$.
\epf

\nd If $\pi_1(G/H_S)$ is finite, then taking $\G = \im \pi_1(i)$,
we get a \mrp \tGtH such that 
$\pi_1(\tG/\tH) = 0$ 
which is \isotf (respectively a \rsppp) if and only if \GH is.


\subsection{\revision{Reduction to} 
        the effective \revision{quotient}}\labell{sec:effective}
It is in a naive sense impossible to fully write out {\rsppp}s 
$(G,H)$,
because if $G$ and another group $K$ are closed normal subgroups
in some larger group and $G \inter K = \trivialgroup$,
we may take the pair $(GK, HK)$ and obtain the same quotient.
While no one is defending this as a sound use of time,
it does obstruct a full classification.
To get rid of boring components of an action,
note that the non-identity elements of the image of the group homomorphism
$\rho\: G \lt \Homeo(G/H)$
act nontrivially by definition,
and the kernel of $\rho$, 
the subgroup of elements fixing each $gH$,
is a normal subgroup of $G$,
namely $\Inter_{g \in G} gHg\-$, 
the largest normal subgroup of $G$ contained in
(and hence also normal in) $H$.

\bdefn\labell{def:effective}
We call a pair $(G,H)$ of Lie groups 
\defd{(virtually) effective} if the 
action of $G$ on $G/H$ is,
which is to say the kernel of $\rho\: G \lt \Homeo(G/H)$
is trivial (respectively, finite).\footnote{\ 
	\emph{Almost effective} is standard, 
	but we prefer to make the analogy with 
	geometric group theory, 
	where $G$ is said to be \emph{virtually} of some class $C$
	if it admits a finite-index subgroup of class $C$.
}
We call a pair $(\ol G,\ol H)$ equivalent to $\big(\rho(G),\rho(H)\mn\big)$
an \defd{effective \revision{quotient}} of $(G,H)$.
\edefn	

Evidently the effective \revision{quotient} of a \rsppp \GH  
is again a \rsppp.
Slightly less obviously, effective \revision{quotient}
also reflects {\mrp}s and \isotfity.

\bprop\labell{thm:inheritance-regularity}
	Let \GH be a pair of compact, connected Lie groups,
	$T \leq G$ a maximal torus,
	$K \leq H$ a closed subgroup normal in $G$,
	and $S \leq H \inter T$ a maximal torus 
	of $H$ such that 
	the identity component $(S \inter K)^0$ 
	is a maximal torus of $K$.
	Then \GH is a \mrp with respect to $T$ 
	if and only if $(G/K,H/K)$ is a \mrp with respect to $TK/K$.
\eprop
\bpf
One checks easily
that the bijection from (closed) subgroups $H' \geq K$
of $G$ to (closed) subgroups $H'/K \leq G/K$
restricts to a bijection between groups $H' \geq K$
satisfying $T \leq N_G(H')$
and groups $H'/K$ with $TK/K \leq N_{G/K}(H'/K)$,
so maximality is preserved and reflected.

Evidently, 
if $S$ is maximal in $H$,
then $SK/K$ is maximal in $H/K$.
On the other hand, suppose tori
$S \leq S' \leq H$ are given such that
$S \inter K$ contains a maximal torus $T_K$ of $K$ and
$SK/K$ is a maximal torus of $H/K$.
As $SK/K \leq H/K$ is a maximal torus, we have
$S'K/K \leq SK/K$
and hence may write any $s' \in S'$ as $sk$ for some $s \in S$
and $k \in K$.
Then $k = s\-s'$ lies in $S' \inter K$,
so $S' \leq S(S' \inter K)$.
As $S' \inter K$ is a compact abelian
group with identity component $T_K$,
this means $S'$ is contained in the union of finitely many
translates of $ST_K$, which is only possible
if $\dim S' \leq \dim ST_K$.
But $T_K \leq S \inter K \leq S \leq S'$,
so  $S' = S$.
\epf

\bprop[Equivariant formality and effective \revision{quotient}s]%
\labell{thm:inheritance-eff}
Let a compact, connected Lie group $H$ act on a compact space $X$ 
in such a way that the restricted action of a closed normal subgroup $K$ is trivial.
Then the $H$-action on $X$ is equivariantly formal 
if and only if the induced action of $\ol H = H/K$ is.
\eprop

\begin{proof}
    \revision{%
        Let $S$ be a maximal torus of $H$
        and $\ol S$ its image in $\ol H$.
        By Borel's well-known
        characterization~\cite[IV.5.5, XII.3.4]{borel1960seminar} of \eqfity of 
        the action of a torus $T$ on a compact space $X$ 
        (also leading to \Cref{thm:dim-criterion}),
        we have $\dim_\Q \H(X) \geq \dim_\Q \H(X^S)$,
        with equality if and only if the $S$-action on $X$
        (or equivalently the $H$-action)
        is \eqf, and similarly for the $\ol S$- and $\ol H$-actions.
        Of course, since we have assumed $K$ acts trivially,
        the $S$-action on $X$ factors through an $\ol S$-action
        and the fixed point sets $X^S$ and $X^{\ol S}$ 
        are thus equal.}\footnote{\ 
            \revision{At least one direction survives 
            even when $H$ is only be a topological group
            acting continuously on a topological space $X$,
		and in fact holds with arbitrary coefficients.
	    }
            As $EH \x E\ol H$ is contractible and a principal
            $H$-bundle under the diagonal action,
            we can use it in the Borel construction
            \[
            X_H = \frac{EH \x E\ol H \x X}{H}
                    = \frac{(EH \x E\ol H \x X)/K}
                                    {H/K}
                    =  \frac{BK \x E\ol H \x X}
                    {\ol H}
                    \mathrlap.
            \]
            With this identification,
            the map $X_H \lt X_{\ol H}$ induced by $H \longepi \ol H$
            projects out the $BK$-coordinate,
            \revision{%
            and $X \lt X_{\ol H}$ factors through $X \lt X_H$,
            so if the former induces a surjection in cohomology,
            then so does the latter.
            }
        }
%
\end{proof}

Now per \Cref{thm:spheres}, 
it remains only to examine 
effective {\mrp}s,
determine which are {\rsppp}s,
and examine the action of $w_0$.
We will see in \Cref{sec:cases}
there are up to equivalence only finitely many classes of cases 
to be dealt with.

\subsection{Reduction to the irreducible}%
\labell{sec:irred}%
\labell{sec:enlarge}

Even an effective action can sometimes be simplified.

\bdefn\labell{def:irreducible}
We call a pair $(G,H)$ \defd{\revision{(normally)} irreducible}
if the action of $G$ on $G/H$ is \defd{irreducible}
\revision{in the sense that} the restricted action
of each proper normal subgroup of $G$ is intransitive.
Given a pair $(\wh G,\wh H)$
and closed subgroups $G \ideal \widehat{G}$ and $H = G \inter \widehat{H}$ 
such that the induced map $G/H \lt \widehat{G}/\widehat{H}$
is a homeomorphism and $(G,H)$ is irreducible,
we call the latter a \defd{(normally) \isp}
of $(\widehat{G},\widehat{H})$.
We also call the embedding of pairs 
$(G,H) \lt (\widehat{G},\widehat{H})$
an \defd{enlargement} in this case, and
$(\widehat{G},\widehat{H})$ an \defd{enlargement of} $(G,H)$.

\edefn

\brmk\labell{rmk:irreducible}
Every pair contains an \isp,
and an \ip is always virtually effective, 
but not vice versa~\cite[p.~76]{onishchik}.
For example, 
the standard action of $\U(n)$ on $S^{2n-1} \subn \C^n$
is effective but the action of its normal subgroup $\SU(n)$
is also transitive.
\ermk

We mean to find irreducible subpairs of the effective \revision{quotient}.
In \Cref{def:irreducible}, there is no guarantee that if 
$\wh G$, $\wh H$, and $G$ are connected, then $H$ will be,
but if $\pi_1(G/H) = 0$,
as we arranged in \Cref{sec:cover},
then this is true by examination 
of the long exact homotopy sequence.\footnote{\ 
On the other hand, if $\widehat{G}$, $G$, and $H$ are connected,
$\wh H$ is connected too. 
%
}
Isotropy-formality is inherited by irreducible subpairs.

\bprop[Inheritance of equivariant formality]\labell{thm:inheritance}
Let a topological group $\widehat{H}$ act on space $X$ and let $H \leq \widehat{H}$ be a subgroup.
Then if the $\widehat{H}$-action on $X$ is equivariantly formal, 
so is the restricted $H$-action.
\eprop
This is very well known, 
but the proof is too simple to be worth excluding.
\bpf
Taking $EH = E\widehat{H}$, the fiber inclusion $i\: X \lt X_{\widehat{H}}$
of the Borel construction factors as
$X \to X_{H} \to X_{\widehat{H}}$.
Evidently if $i$ induces a surjection in cohomology,
so must $X \lt X_{H}$.
\epf

Irreducibility is also preseved by finite covers.

\bprop[Irreducibility and covers]\label{thm:cover-irred}
\revision{%
A \ccpair $(G,H)$ with $G/H$ positive-dimensional 
is irreducible if
and only if
any one of its finite connected covers $(\widetilde G, \widetilde H)$
is irreducible.}
\eprop
\bpf
\revision{%
If $\wt K \ideal \wt G$ acts transitively on $\wt G/\wt H$,
then the image $K$ of $\wt K \inc \wt G \epi G$ acts transitively on $G/H$
and is again proper since $\wt K$ cannot be contained in $\ker(\wt G \to G)$
lest one have cardinalities $|\wt G/\wt H| = |\wt K \. 1\wt H| \leq |\wt K|$ finite.
On the other hand $K \ideal G$ acts transitively on $G/H$
if and only if $K\. 1H = G/H$, in which case the identity component
$\wt K \normal \wt G$ of the preimage of $K$ 
is such that $\wt K\. 1\wt H \sub \wt G/\wt H$ is connected and covers $G/H$,
so that it is all of $\wt G/\wt H$.
}
\epf



%
%


Evidently we can create ineffective enlargements 
$(G,H) \longinc (G \x P, H \x P)$
of compact, connected Lie group pairs
with wild abandon,
and if $G$ is normal in $\widehat{G}$,
then $\widehat{G}$ must virtually be such a product,
but effectiveness means the action of the new factor must commute
with the given $G$-action, severely restricting our options.
\revision{The interaction between normalizers and centralizers is
thus crucially important.}

\begin{lemma}\labell{thm:NZH}
Let \GH be a \ccpair. 
Then $N_G(H)^0$ is $Z_G(H)^0 \. H$.
\end{lemma}
\begin{proof}
\revision{
As $H$ is normal in $N_G(H)^0$, 
it has a connected complement $Q$ in $N_G(H)^0$, 
by \Cref{thm:Lie-structure}.
Since $Q$ centralizes $H$,
we have
$N_G(H)^0 = Q\. H \leq Z_G(H)^0\. H \leq N_G(H)^0\.H = N_G(H)^0$.
	}
\end{proof}

\Cref{thm:NZH} gives a virtual direct product 
decomposition if $H$ is semisimple, 
but otherwise complements to $H$ in $N_G(H)^0$
may only be determined up to some arbitrary choice.



\begin{lemma}\labell{thm:Q}
Let \GH be a pair of compact, connected Lie groups.
Then a subgroup $\defm Q \leq Z_G(H)^0$ can be found
which is a complement to $H$ in $N_G(H)^0$  
and such that every closed, connected subgroup $\wh P \leq N_G(H)^0/H$
is of the form $PH/H$ for a unique closed, connected subgroup $P \leq Q$.
\end{lemma}
\begin{proof}
	Recall that $\fg$ carries an $\Ad(G)$-invariant inner product $-B$.
	The orthogonal complement $\f q = \fh^\perp$ in 
	the normalizer $\f n = \f n_{\fg}(\fh)$ 
	is automatically an ideal of $\f n$,
	since as $\f h$ is an ideal, for $(h,n,q) \in \f h \x \f n \x \f q$, 
	we have $B\big(h,[n,q]\big) = B\big([h,n],q\big) = 0$.
	Let $Q = \exp \f q$.
	Writing $\pi\: N_G(H)^0 \lt N_G(H)^0/H$ for the projection
	and given a connected $\smash{\wh P} \leq N_G(H)^0/H$, 
	we take $P = \big(\pi\-(\smash{\wh P}) \inter Q\big){}^0$.
\end{proof}

\bthm[{\emph{Cf.} Kramer~\cite[\S3.5, p.~21--22]{kramer2002homogeneous}, 
	Onishchik~\cite[pp.~73--76]{onishchik}}]\!\!\!\footnote{\ 
	A result of this kind for $G/H = S^{\mr{odd}}$ is already proved by Montgomery--Samelson%
	~\cite[\S\S4--7]{montgomerysamelson1943}.
}
\labell{thm:enlargement}
Let $(G,H)$ be an effective pair of compact, connected Lie groups.
Then for
any effective enlargement $(\widehat{G},\widehat{H})$  
there exists
a closed, connected subgroup $P$ of $Z_{G}(H)^0$, 
meeting $H$ in a finite subgroup,
such that $\big(G \x P,(H \x 1) \. \D P\big)$ is a cover of $(\wh G,\wh H)$.
This 
is again an enlargement of $(G,H)$
and is \isotf if and only if $(\widehat{G},\widehat{H})$ is.
 
The action of $G \x P$ on $G/H$
is given by $(g,p)\.xH \ceq gxp\-H$.
\ethm

\begin{proof}
	Let $\defm{\widehat P} \ideal \widehat{G}$ be a complement of $G$.
	By effectiveness, we may identify $\widehat{G}$ with a subgroup of $\Homeo(G/H)$
	and $\wh P$ with a subgroup of the centralizer $\defm C$ of $G$ 
	in $\Homeo(G/H)$.
	Any map $\phi \in C$
	satisfies $\phi(1H) \eqc \ell H$ for some $\ell \in G$,
	and since $hH = H$ for each $h \in H$,
	one has $h\ell H = h\phi(H) = \phi(hH) = \phi(H) = \ell H$,
	so $\ell$ lies in the normalizer $N_{G}(H)$.
	Now, $\phi(gH) = g\ell H = gH\ell H$ for $g \in G$,
	so that $\phi$ is the right translation by $\ell H$ on $G/H$.
	Thus $C$ is isomorphic to $N_{G}(H)/H$,
	acting on $G/H$ by the standard effective right action.

	Since $\widehat{G} = G \widehat P$ acts on the left on $G/H$,
	to proceed we convert the action of $\widehat P \leq N_G(H)/H$ to a left action.
	Using \Cref{thm:Q},
	we fix a centralizing complement $Q$ to $H$ in $N_G(H)^0$
	and find a unique closed, connected subgroup $\defm{P} \leq Q$ 
	such that $\wh P = PH/H$.
	Then $p H\.xH = x p\- H$ 
	for $pH \in \wh P = PH/H$ and $xH \in G/H$,
	so the action of $\widehat{G}$ on $G/H$
	descends from the ineffective action
	$(g,hp)\.xH \ceq gxp\-H$ of $G \x HP$.
	The stabilizer of $1H \in G/H$
	under this action
	is the set $(H \x \trivialgroup)\.\D(HP)$
	of pairs $(h'hp,hp) \in G \x HP$ for 
	$h' \in {H}$ and $hp \in HP$,
	so we have an ineffective enlargement 
	$\big(G \x HP,(H \x \trivialgroup)\.\D(HP)\mnn\big)$.

	Now, $P$ contains a full set of coset representatives of 
	$\wh P = PH/H$,
	so the composite
	$
	G \x P \inc G \x HP \epi G \x \widehat P
	$
	is a surjection with kernel $\trivialgroup \x (H \inter P)$.
	The stabilizer $\smash{(H \x \trivialgroup)} \.\D(HP)$ meets $G \x P$
	in $\smash{(H \x \trivialgroup)}\.\D P$, 
	which we will abbreviate $\defm{H \D P}$.
	Thus $(G \x P,H \D P)$
	is a virtually effective pair 
	covering 
	$(\widehat{G},\widehat{H})$;
	by \Cref{thm:inheritance-eff}, 
	it is \isotf if and only if $(\widehat{G},\widehat{H})$ is.
\end{proof}


As $\wh P$ in \Cref{thm:enlargement} is a subgroup of $N_G(H)^0/H$,
which equals $Z_G(H)^0 H/H$ by \Cref{thm:NZH} 
(and $QH/H$ by \Cref{thm:Q}),
we have the following.

\bprop \labell{thm:enlargement-conditions}
Let $(G,H)$ be an effective pair of compact, connected Lie groups.
Then the existence of a proper virtually effective enlargement $(\widehat{G},\widehat{H})$ of $(G,H)$ 
such that $G$ is normal in $\widehat{G}$ 
is equivalent to each of the following conditions:
\bitem
\item 
$Z_{G}(H)^0$ is not contained in $H$;
\item 
$H$ is strictly contained in $N_{G}(H)^0 = Z_G(H)^0 H$;
\item
$\rk N_{G}(H)^0 = \rk Z_G(H)^0 H > \rk H$;
\item
$\rk Z_{G}(H)^0 H / H = \rk Q > 0$.
\eitem 
\eprop
%
%
%

In the case of a corank-one pair, the third condition of
\Cref{thm:enlargement-conditions}
implies that $N_{G}(H)^0$ is of full rank in ${G}$,
or in other words that $H$ is regular.

\bprop\labell{thm:enlargement-corank1}\labell{thm:regularity}
Let $(G,H)$ be an effective corank-one pair of compact, connected Lie groups.
Then the existence of a proper virtually effective enlargement $(\widehat{G},\widehat{H})$ of $(G,H)$ 
such that $G$ is normal in $\widehat{G}$ 
is equivalent to each of the following conditions:
\bitem
\item $H$ is regular in $G$;
\item 
$\rk Z_{G}(H)^0 H/H = \rk Q = 1$; \emph{i.e.}, $Q$ is isomorphic 
	to one of $\U(1)$, $\Sp(1)$, or $\SO(3)$.
\eitem 
In this case, letting $P \leq Q$ be a closed rank-one subgroup
with maximal torus $T_P$ as in \Cref{thm:enlargement},
$ST_P \x T_P$ is a common maximal torus of $G \x P$ and of $HP \x P$,
and $S\D T_P$ is a maximal torus of $H\D P$.
\eprop

\begin{corollary}\labell{thm:sphere-enlargement}
If \GH is an effective \rsp,
there is a proper, virtually effective enlargement
$(\widehat{G},\widehat{H})$ of $(G,H)$ 
such that $G$ is normal in $\widehat{G}$ 
if and only if $H = H_S$.
\end{corollary}
\begin{proof}
	By \Cref{thm:isotropyOddSphere},
	in this case $H$ is regular with respect to $ST_P$
	if and only if $H = H_S$.
\end{proof}


%

When there does exist an extension, it is essentially unique:

\blem \labell{thm:enlargement-equiv}
Two enlargements
$(G \x  P_1,  H \D  P_1)$ and  $(G \x  P_2,  H \D  P_2)$ 
of an irreducible corank-one pair \GH
as constructed in \Cref{thm:enlargement}
are equivalent if and only if $P_1$ and $P_2$ are isomorphic.
\elem
\bpf
Necessity is immediate.
For sufficiency, 
recall that in \Cref{thm:enlargement},
the $P_j$ can be taken 
as subgroups of a fixed rank-one $Q \leq Z_G(H)^0$.
If $P_1 \iso P_2 \leq Q$ are of type~$A_1$,
then all must be equal. 
Otherwise, $P_1 \iso P_2 \iso S^1$ are maximal tori of $Q$,
so $P_2 = q P_1 q\-$ for some $q \in Q$.
Then $(g,p) \lmt (qgq\-,qpq\-)$ induces 
an equivalence of pairs $(G \x P_1, H \D P_1) \lt (G \x P_2, H \D P_2)$.
\epf

An effective \revision{quotient} of a \mrp gives another \mrp,
and a virtually effective covering of a \mrp
is again a \mrp,
so our reduction only requires us to handle the case 
when $H \D P$ is regular.

\bcor \labell{thm:regularityOfEnlargement}
In the situation of \Cref{thm:enlargement-corank1}, 
the stabilizer $H \D P$ is a regular subgroup of $G \x P$ 
with respect to $ST_{P} \x T_{P}$ if and only if $P \iso S^1$.
\ecor
\bpf
Since $H$ centralizes $\D P$, 
we see 
$ST_{P} \x T_{P}$ normalizes $H \D P$ 
if and only if 
$T_{P} \x T_{P}$ normalizes $\D P$.
This happens if $P = T_{P}$ and not otherwise.
\epf


Now we finally can reduce \isotfity to the irreducible case.

\begin{theorem}\labell{thm:enlargement-isotf}
If $(G, H)$ is a regular, irreducible \rspp,
then there exists an isotropy-formal enlargement 
$(G \x S^1,  H \D S^1)$
if and only if $H_S > H$. 
In this case $(G,H)$ is also \isotf.
\end{theorem}


\bpf
For the second clause,
if any enlargement $(\widehat G,\widehat H)$ is isotropy-formal,
then $(G,H)$ inherits \isotfity by \Cref{thm:inheritance}.

For the equivalence, 
first suppose $(G,H)$ is a \rspp and $H_S > H$; 
then \Cref{thm:extend-H} says $(G,H_S)$ is a \rsp.
As $H_S$ is regular, \Cref{thm:enlargement-corank1} implies 
$\rk Z_G(H_S)/H_S = 1$
so there is a circle $C \leq Z_G(H_S)^0$ meeting~$H_S$ finitely
and \emph{a~fortiori} also meeting~$H$ finitely.
The enlargement $(G \x C,H_S \D C)$ is again an \rsp,
hence \isotf by \Cref{thm:isotropyOddSphere},
and hence so is the enlargement $(G \x C,H \D C)$
as $H \D C$ and $H_S \D C$ share a maximal torus. 
By \Cref{thm:enlargement-equiv}, 
any other enlargement $(G \x S^1,  H \D S^1)$ is also \isotf.

On the other hand, suppose 
an 
\revision{\isotf}
enlargement $(G \x C,  H \D C)$ 
exists for some circle $C \leq G$.
By \Cref{thm:enlargement-corank1}, 
$(G,H)$ is regular, so $H \leq H_S$
for some maximal torus $S$ of $H$.
\revision{By \Cref{thm:enlargement-equiv},
$(G \x S^\perp, H\D\St)$ is \isotf as well.}
We will prove that if $H_S = H$, 
\revision{this leads to a contradiction,
so that in fact $H_S > H$.}
Now~$S^\perp$ centralizes $H = H_S$ by \Cref{thm:HS},
and by definition $T = S S^\perp$ is a maximal torus of $G$,
so $(\defm{\widehat T}, \defm{\widehat S}) = (T \x S^\perp,S\D S^\perp)$ 
are maximal tori in the enlargement $(G \x S^\perp, H \D S^\perp)$,
and an orthogonal complement 
is given by $\defm{\widehat{S}^\perp} \ceq \big\{
(z,z\-) \in G \x S^\perp: z \in S^\perp\big\}$.
Since $N_{G \x S^\perp}(\widehat T) = N_G(T) \x S^\perp$
and $N_{H\D S^\perp}(S \D S^\perp) = N_H(S)\D S^\perp$,
the inclusion of \GH
induces isomorphisms of Weyl groups $W_G \isoto W_{G \x S^\perp}$
and $W_H \isoto W_{H\D S^\perp}$.
In particular, if $n_0 \in N_G(T)$
represents the longest word~$w_0^G$ of $W_G$,
then $(n_0,1) \in N_{G}(T)  \x S^\perp$
represents $w_0^{\smash{G \x S^\perp}} \in  W_{G \x S^\perp}$.
Moreover, since an element $(g,s^\perp) \in G \x S^\perp$ 
normalizing $S \D \St$ also normalizes $G \inter S\D\St = S$,
and then all elements of $\{g\} \x \St$ induce the same automorphism of $S\D\St$,
the projection $G \x \St \longepi G$ induces an injection
$\defm{\widehat N} \ceq \pi_0 N_{G \x \St}(S\D\St) \longmono \pi_0 N_G(S) = N$. 
\revision{By the second clause in this theorem, $(G,H)$ is \isotf,}
so applying \Cref{thm:spheres} and \Cref{thm:Wv}, 
the longest word $w_0^G$ satisfies
$ w^G_0 \. v = -v$ and 
$ w^G_0|_{ \fs} \notin W_H$.
But the adjoint action $\smash{w_0^{\smash{G \x S^\perp}}}$ 
of $(n_0,1)$ sends $(v,v) \in \revision{\D\fs^\perp <{}} \fs + \D\fs^\perp$
to $(-v,v) \notin \fs + \D\fs^\perp$,
so the restriction $\smash{w_0^{\smash{G \x S^\perp}}|_{\fs\, + \D \fs^\perp}}$ 
is not an element of $\widehat N$
and hence by \Cref{thm:Wv} 
and \Cref{thm:spheres} again,
the enlarged pair 
is not \isotf.
\epf

\begin{corollary}\labell{thm:exclude-reducible}
	A maximal regular, effective but reducible \rspp is not \isotf.
\end{corollary}
\begin{proof}
  By \Cref{thm:enlargement}, \Cref{thm:enlargement-corank1} and \Cref{thm:regularityOfEnlargement},
  such a pair
  $(G,H)$ is finitely covered
  by $(G' \x S^1, H' \D S^1)$, where $(G',H') < (G,H)$ is irreducible
  with $G'/H' \lt G/H$ a homeomorphism and $S^1 \revision{{}\leq{}} Z_{G}(H')$ 
  meets $H'$ in a finite set.
  This cover is again maximal regular by \Cref{thm:cover-isotf} 
  and virtually effective.
  If $(G,H)$ were \isotf, then by \Cref{thm:enlargement-isotf} applied to $(G',H')$,
  we would have $H_{S'} > H'$ for $S'$ a maximal torus of $H'$,
  but then we would have $H_{S'} \D S^1 > H' \D S^1$,
  so that $(G \x S^1, H_{S'} \D S^1)$ would be a regular pair properly containing
  $(G' \x S^1, H' \D S^1)$, which we asssumed was maximal regular.
\end{proof}

\begin{corollary}
	A maximal regular, effective \rspp $(G,H)$ with $Z_G(H)^0 H/H \iso A_1$ is not \isotf.
\end{corollary}
\begin{proof}
\revision{    
Due to maximal regularity, so $H = H_S = [L_S,L_S]S$,
which by \Cref{thm:HS}.\ref{thm:HS-2} centralizes~$S^\perp$.
Since $S^\perp$ does not lie in $H$ 
(lest $H$ contain $T = SS^\perp$),
it descends to a circle subgroup of $Z_G(H)^0 H / H$. 
Notice that the nontrivial element $\bar w$ of the Weyl group
of $Z_G(H)^0 H / H = A_1$
reflects the tangent line to this circle.
The quotient $A_1$ corresponds to a normal complement~$Q$ to $H$
in $N_G(H)^0$ which lies in $Z_G(H)^0$, 
and has $S^\perp$ as a maximal torus, 
by \Cref{thm:Q},
and $\bar w$ lifts to the nontrivial element of $W_Q$,
represented by some $q \in Q$.
Since $Q$ centralizes $H \leq S$ and normalizes~$S^\perp$,
it normalizes $T = SS^\perp$,
and so induces an element $w \in W_G$
reflecting the tangent line $\R v$ to~$S^\perp$.
Thus, in the notation of \Cref{def:tN} and \Cref{thm:w0},
we have $\tilde N = W_{\{\pm v\}} = \ang{W_v, w} \neq W_v$.
But $q$ centralizes $S$, 
so $w$ acts trivially on $\mathfrak s$, 
and hence $N = \ang{W_v |_{\mathfrak s}, w|_\fs} = W_v|_\fs$. 
By \Cref{thm:spheres}.\ref{item:product}, then,
\GH cannot be isotropy-formal.
}
\end{proof}

\subsection{Recapitulation}\labell{sec:procedure}

We have now assembled a long list of reductions
from the generic case to a highly constrained one,
which we summarize  as a proof of the algorithm.

\begin{proof}[Proof of \Cref{thm:alg}]
\ 
	
\textbf{1}.
By \Cref{thm:CF}, 
	\isotfity of \GK is equivalent to that of $(G,H_S)$.

\textbf{2}.
The evenness of $d$ follows since exactness of \eqref{eq:Gysin}
implies the dimensions of $H^{\mr{even}}(G/H_S)$ and $H^{\mr{odd}}(G/H_S)$
are equal.
By \Cref{thm:isotf-circle-factor}, if $\pi_1(G/H_S)$ is infinite,
then \GHS is \isotf.

\textbf{3}.
The case analysis is \Cref{thm:spheres}.
By \Cref{thm:inheritance-eff}, 
\isotfity of $(G,H_S)$ is equivalent to that of the effective
\revision{quotient} $(\ol G,\ol H)$.
By \Cref{thm:inheritance-regularity}, $(\ol G,\ol H)$ is again a \mrp,
and by definition, $(\ol G,\ol H)$ is a \rspp if and only if $(G,H_S)$ is.

\textbf{4}.
This is a \rspp.
\revision{By \Cref{thm:cover-constr},
we can find a pair $(\widetilde G, \widetilde H)$ covering $(\ol G,\ol H)$
with $\pi_1(\widetilde G/\widetilde H) = 0$ and $\widetilde H = \ol H$.
This is again a \mrp and a \rspp by \Cref{thm:cover-isotf}.}
In \Cref{sec:sphere-product-class},
we construct the list of irreducible {\rspp}s $(\widetilde G, \widetilde H)$
with simply-connected $\widetilde G/\widetilde H$
in \Cref{table:Kramer+Wolfrom}.
This list includes the irreducible {\mrp}s which are {\rspp}s.
If $(\wt G,\wt H)$ is \revision{not in the table,
then it is reducible, and by \Cref{thm:exclude-reducible},
it is not \isotf.}
If $(\wt G,\wt H)$ is in the table, it is irreducible; 
in \Cref{sec:isotf} we determine which entries of the table are \isotf
through a case analysis hinging on \Cref{thm:spheres} and \Cref{thm:Wv}.\ref{thm:N}.
\end{proof}

\section{Homogeneous products of rational spheres}\labell{sec:cases}

The previous section should 
make it clear that we need to classify irreducible {\rspp}s,
which we will accomplish in this section.

\begin{discussion}[Classes of cases]\labell{rmk:subdivision}
Up to a bounded error 
(a handful of missing and spurious cases, 
which we will repair and exclude in detail in 
\Cref{rmk:comparison-Kramer}),
this classification is already present in the work of Kramer 
and his students
in the cases of 
\begin{itemize}
\item 
	$S^n \x S^m$ for $n$ odd, $m$ even and $n > m \geq 4$ (Kramer~\cite{kramer2002homogeneous}),
\item 
	$S^n \x S^{n'}$ for odd $n,n' \geq 3$~(Kramer~\cite{kramer2002homogeneous}),
\item 
	$S^n \x S^2$ for odd $n > 2$ (Wolfrom~\cite{wolfrom2002}), and
\item 
	$S^1 \x S^\ell$ for any $\ell$ (Bletz-Siebert~\cite{bletzsiebert2002}).
\end{itemize}
Beyond examining this work, it thus remains only for us to analyze the cases of 
\begin{itemize}
\item
		$S^n \x S^m$ for $n$ odd, $m$ even and $m > n \geq 3$, which we need, and
\item 
	$S^m \x S^{m'}$ for $m,m'$ both even,
	which we include for completeness.
\end{itemize}
We will prove in \Cref{section:unexamined}
that the pairs in the two previously unexamined cases are all 
covered by products of two homogeneous rational cohomology spheres,
so that we do not find any new classes of example.
\end{discussion}

Our analysis and repair of the classification
of {\rspp}s
is greatly aided by the doctoral disseration of Kamerich~\cite{kamerich},
which classifies the irreducible pairs $(G,H)$ 
such that $G/H$ is
\emph{homeomorphic} to the product of two spheres, but
which the other authors named above seem not to use.\footnote{\ 
	Kramer also classifies the irreducible pairs such that
	$G/H$ has the \emph{integral} cohomology of a product of spheres~\cite[pp.~x--xi]{kramer2002homogeneous}---in fact,
	this is his main goal of the first part of his monograph---%
	although he is missing some cases implied by 
	Kamerich's classification, which was not accessible to 
	him at the time [personal communication].
}
Although Kamerich's goals differ from ours,
the rational classification can be extracted 
with some additional analysis
from his preliminary results assembling a list of candidates,
which includes a number of cases not considered by Kramer;
see \Cref{rmk:comparison-Kamerich}.

In the first subsection to follow,
we state what is known in the $n = 1$ case.
In the next two subsections we will extend 
the classification to cover the new cases, such as they are,
and in the subsequent, much longer, subsection, 
we will recapitulate, revise, and verify the existing classification.

\subsection{The case of $S^1 \x S^\ell$}\labell{sec:circle-cases}

By \Cref{thm:isotf-circle-factor},
we do not need a full classification of {\rspp}s
with $n = 1$, but 
for completeness we recount what is known.
As we have seen in \Cref{thm:isotf-circle-factor},
in this case any such pair covers a pair of the form 
$(G',H') \x (S^1,\trivialgroup)$
with $G'/H'$ a rational cohomology sphere,
so the classification of such (effective, resp. irreducible) pairs up to covers
reduces to the classification of (effective, resp. irreducible) pairs $(G',H')$
with $G'/H'$ a rational cohomology sphere up to covers.
These are completely known, and tabulated in \Cref{table:EvenRationalSphere,table:OddRationalSphere,table:OddRationalSphereEffective}.

Bletz-Siebert has also studied pairs \GH with the rational homotopy 
of $S^1 \x S^\ell$ and $H$ not necessarily connected,
producing something a bit finer-grained.
We quote three such results to give the flavor.

\begin{theorem}[Bletz-Siebert~{\cite[Thm.~2.5.11]{bletzsiebert2002}}]
	Let \GH be a pair of compact Lie groups, $G$ connected,
	such that $\pi_1(G/H) \iso \Z$
	and $\pi_*(G/H;\Q) \iso \pi_*(S^1 \x S^\ell;\Q)$ with $\ell \geq 2$.
	Then component group $\pi_0(H)$ is cyclic
	and there is a circle subgroup $C \leq Z(G)^0$
	complementary to the commutator subgroup $[G,G]$
	such that the restricted action of $G' = [G,G]C$ on $G/H$ is transitive
	with stabilizer $H' = [G,G] \inter H = (G' \inter H)^0$,
	and $G/H$ is finitely covered by $S^1 \x [G,G]/H'$.
\end{theorem}

\begin{theorem}[Bletz-Siebert~{\cite[Lem.~6.3.6]{bletzsiebert2002}}]
	Let \GH be a pair of compact, connected Lie groups such that 
	$\pi_*(G/H;\Q) \iso \pi_*(S^1 \x S^\ell;\Q)$ with $\ell \geq 3$. 
	Then there are a circle subgroup $C \leq Z(G)^0$
	and a simple subgroup $G' \ideal G$
	such that the virtual direct product $G' \. C \ideal G$
	acts transitively on $G/H$,
	with stabilizer $(G' \inter H)^0$ either simple or trivial.
\end{theorem}

\begin{theorem}[Bletz-Siebert~{\cite[Thm.~6.3.7]{bletzsiebert2002}}]
	Let \GH be an irreducible pair of compact Lie groups 
	such that $G$ is connected,
	$\pi_1(G/H)$ is torsion-free, and 
	$\pi_*(G/H;\Q) \iso \pi_*(S^1 \x S^\ell;\Q)$ with $\ell \geq 3$. 
	Then $\pi_0(H)$ is cyclic,
	$[G,G]$ is simple, $Z(G)^0$ is a circle, 
	$H^0$ is $H \inter [G,G]$,
	and $[G,G]/H^0$ is a simply-connected rational cohomology $S^\ell$.
\end{theorem}

\subsection{Reduction to the case $\rk \pi_3 \leq 2$}\labell{sec:reductions}
For the general discussion of {\rspp}s $(G,H)$,
we require a result comparing the ranks and degrees of $G$ and $H$.

\bdefn
By the \defd{degrees} of a connected Lie group $G$
we mean the multiset $\defm{\deg \, G}$ of degrees $|z|$
of generators of the exterior algebra $\H(G;\Q) \iso \ext[\vec z]$. 
These can be enumerated as
the degrees of a homogeneous basis of the graded vector space
$\defm {P G}$ of primitives of the Hopf algebra $\H(G;\Q)$.
\edefn

\bprop[{\cite[\S9]{kamerich}}, 
		\emph{cf.} Kramer 
		{\cite[Thm.~3.11]{kramer2002homogeneous}}%
		]%
		\labell{thm:deg}
Let $(G,H)$ be a {\rspp}.
Then $\rk G = \rk H + 1$ and one of the following two 
possibilities holds:
\begin{enumerate}[label=\alph*)]
\item $n\neq m-1$ and $\deg G \,\less \deg H = \{n,2m-1\}$, 
  while $\deg H\, \less \deg G = \{m-1\}$;
\item $n = m-1 \in \deg G \inter \deg H$ and $\deg G = \deg H \amalg \{2m-1\}$.
\eenum
\eprop

We can briefly reproduce the proof in Kamerich's 
dissertation.


\begin{proof}
%
%
It is a result of Onishchik~\cite[Rmk., p.206]{onishchik}\cite[Thm.~1]{onishchik1963transitive}
following from uniqueness of the minimal model
that if homogeneous spaces $G/H$ and $G'/H'$ 
with $G$, $G'$, $H$, $H'$ compact, connected
have the same real homotopy type,
the quotients $p(G)/p(H)$ and $p(G')/p(H') \in \Q(t)$
of Poincar\'e polynomials are equal.
Taking $G' = \SO(n+1) \x \SO(m+1)$ and $H' = \SO(n) \x \SO(m)$,
we have 
\[
\frac{\prod_{j \in \deg G}(1+t^j)}{\prod_{k \in \deg H}(1+t^k)}
	= 
\frac{p(G)}{p(H)} 
	= 
\frac{p(G')}{p(H')} 
	= 
(1+t^n)\frac{1+t^{2m-1}}{1+t^{m-1}}
\mathrlap,
\]
or simply $1+t^{2m-1}$ if $n = m-1$, giving the result.
\end{proof}


\brmk\labell{rmk:correction}
While \Cref{thm:deg} rules out
potential {\rspp}s very efficiently,
it is not a sufficient condition.
For example, Kramer's lists~\cite[5.12. p.~60]{kramer2002homogeneous} 
include the pair $(G,H) = \big(\SO(8),\SO(3) \x \SO(5)\mnn\big)$,
which has 
\eqn{
	\deg G 			&= \{3,7,7,11\},\\
	\deg H 			&= \{3,3,7\},\\
	\deg G \,\mnn\less \deg H &= \{7,11\},\\
	\deg H \,\mnn\less \deg G &= \{3\}.
}
This numerical data satisfies the conclusion of \Cref{thm:deg} 
with $m  = 4$ and $n = 11$,
but the cohomology ring of the 
oriented real Grassmannian $\widetilde G_3(\R^8) = G/H$
is actually $\quotientmed{\Q[p_1]}{\mn(p_1)^3} \,\ox\, \ext[\eta]$
with $|p_1| = 4$ and $|\eta| = 7$;
the relevant arithmetic of degrees is that $11 = 3m - 1$, 
the differential $d_{12}$ of the degree-$11$ generator of $\H \SO(8)$
cancelling $p_1^3 \in H^{12}(BH)$ 
in the \SSS of $G \to G/H \to BH$,
while the fact $n = 7 = 2m - 1$ is a red herring.

We are thus forced to compute $\H(G/H)$ for all pairs \GH
in the existing classification
to verify that they truly belong.
Fortunately, we will find we only need to exclude a few cases.
\ermk

\Cref{thm:deg} enables us to make a useful reduction on $\rk G$.

\bprop
\labell{prop:Kramer-reduction}
Suppose $(G,H)$ is a {\rspp}
with $m,n\geq 2$.
Then there is a closed, connected, semisimple normal subgroup $\ol G \ideal G$ 
acting transitively on \revision{$G/H$}
and $\rk \pi_3 \ol G \leq 2$.
This $\ol G$ can be taken to act irreducibly.
\eprop
\begin{proof}
	This is Kramer's Proposition 3.14~\cite{kramer2002homogeneous},
	minus his hypothesis that $n>m$.
	It turns out this hypothesis is unnecessary:
	examining his proof, 
	he uses it only 
	in his observation that $\dim \ker(PG \to PH) = 2$,
	a tacit reference to his Theorem 3.11;
	but this is subsumed in Kamerich's \Cref{thm:deg} \cite{kamerich},
	which does not require the dimension hypothesis.
\end{proof}

\brmk 
Since $\ol G$ is semisimple, we have $\pi_1 \ol G$ finite and can replace $\ol G$ with its universal cover;
$\rk \pi_3 \ol G$ is just the number of simple factors of this cover. 
If $\rk \pi_3 \ol G = 2$, 
then $\ol G$ will be a virtual direct product of two simple groups.
\ermk
\subsection[Unexamined cases in the classification]{The case with $m>n\geq 3$ and the case with both $m,m'$ even}\labell{section:unexamined}
In this subsection, we show that when a homogeneous space has the rational homotopy of $S^n \x S^m$ with
$m$ even, $n$ odd,
and $m>n\geq 3$, or the rational homotopy of $S^m \x S^{m'}$ with both $m,m'$ even,
it is a virtual 
product of homogeneous rational cohomology spheres.

We continue to write $\defm{G^0}$ for the identity component of a topological group $G$.

\bthm \labell{thm:m>n}
Let $(G,H)$ be an irreducible {\rspp} 
with $m > n \geq 3$. 
Then $G = K_1 \. K_2$ is a virtual direct product of 
two simple subgroups
and
$H = (H\inter K_1)^0 \cdot (H \inter K_2)^0$ is also a virtual direct product. Hence $G/H$ is a virtual direct product of 
homogeneous rational cohomology spheres of dimensions $n$ and $m$.   
\ethm

This says, in essence, that this newly considered class of cases
contains nothing morally new. 
The proof will be given in several steps.

\blem\labell{lem:pi3-rank}
Under the assumptions of \Cref{thm:m>n}, both $G$ and $H$ are semisimple and $\rk \pi_3 H \leq \rk \pi_3 G \leq 2$.
\elem
\begin{proof}
By \Cref{prop:Kramer-reduction}, 
we know $G$ contains a connected, semisimple normal subgroup $\ol G \ideal G$ with $\rk \pi_3 \ol G \leq 2$ and acting transitively on $G/H$. 
Since we have assumed $G$ itself acts irreducibly, 
we must have $G = \ol G$;
hence $G$ is semisimple with $\rk \pi_3 G \leq 2$.

By assumption, $\pi_1(G/H) \ox \Q$ is trivial, 
so $\pi_1(G/H)$ is finite.
Let $\widetilde{G/H}$ be its universal cover,
$\widetilde G$ the universal cover of $G$,
and~$\widetilde H < \widetilde G$ 
the stabilizer of a point of $\smash{\widetilde{G/H}}$ lying over $1H \in G/H$.
Then~$\widetilde H$ is a finite-sheeted cover of $H$ as well.
The long exact homotopy sequence of the fibration
$\smash{\widetilde H \to \widetilde G \to \widetilde{G/H}}$
contains the subsequence
$
\pi_2(\widetilde{G/H})
	\to
\pi_1 \widetilde H
	 \to
\pi_1 \widetilde G = 0
,
$
implying $\pi_2(\widetilde{G/H})
\lt
\pi_1 \widetilde H$ is surjective.
As we have assumed $m > n \geq 3 > 2$,
we have $\pi_2(\widetilde{G/H}) \ox \Q \iso \pi_2(G/H) \ox \Q = 0$,
so $\pi_2(\widetilde{G/H})$ is finite.
But then by surjectivity, 
$\pi_1 \widetilde H$ 
and hence $\pi_1 H$ are finite as well, 
so that $H$ too is semisimple.

We also have the exact fragment
\quation{\labell{eq:LES}
	\overbrace{\!\pi_4 G}^{\mr{finite}} \!\lt 
	\pi_4 (G/H) \lt 
	\pi_3 H \lt 
	\pi_3 G \lt 
	\pi_3(G/H) \lt
	\!\!\overbrace{\pi_2 H}^0\!\!.
}
Now we show by contradiction that $\rk \pi_3 H \leq \rk \pi_3 G$.
Assume instead that $\rk \pi_3 H > \rk \pi_3 G$.
Then by exactness of \eqref{eq:LES} 
we have $\rk \pi_4(G/H) \geq 1$ and hence $m = 4$.
Since we have assumed $m > n \geq 3$, 
this forces $n = 3$ and hence $\rk \pi_3(G/H) = 1 = \rk \pi_4(G/H)$. 
But then by exactness of \eqref{eq:LES} again,
we have $\rk \pi_3 H = \rk \pi_3 G$,
contradicting our assumption.
\end{proof}

\begin{lemma}\labell{lem:noSimplePair}
Under the assumptions of \Cref{thm:m>n}, we have $\rk \pi_3 G = 2$.
\end{lemma}
\begin{proof}
By \Cref{lem:pi3-rank}, we have $\rk \pi_3 G \leq 2$.
If we had $\rk \pi_3 G = 1$, or in other words if $G$ were simple,
then $H$ would also be simple by \Cref{lem:pi3-rank}.
But we are able to rule out all examples of pairs $(G,H)$
of simple groups such that $\H(G/H)$ is of the requisite form
$\H(S^n \x S^m)$
using the necessary condition of \Cref{thm:deg}
and working through the Killing--Cartan classification.
We have no particular insight as to why such examples cannot occur,
so such a proof is just a verification, 
either manual or (better) automated, and not particularly rewarding.
We list some simplifications that make the verification process finite:
\begin{itemize}
\item For each simple $G$, 
there are only finitely many simple $H$ with $\defm \ell \ceq \rk H = \rk G - 1$.

\item For each $\ell \geq 9$ there are only $3$ types of $G$ and $H$ to consider.

\item The gaps in degrees for Killing--Cartan types $BC_\ell$ and $D_\ell$ 
are generically $4$,
while the gaps for type $A_\ell$ are $2$.
Since \Cref{thm:deg} implies the cardinality of 
$(\deg G \less \deg H) \union (\deg H \less \deg G)$ is at most $3$,
examples with $G = A_{\ell+1}$ and $H \in \{BC_\ell,D_\ell\}$
or $G \in \{BC_{\ell+1},D_{\ell+1}\}$ and $H = A_\ell$ 
are impossible for $\ell \geq 4$.

\item The degrees for the cases when $G$ and $H$ both lie in the same one of 
the infinite families 
$A$, $BC$, and $D$ correspond to rational cohomology spheres.

\item For $G = D_\ell$ and $H = BC_{\ell-1}$,
one has $\deg G = \deg H \amalg \{2\ell-1\}$,
corresponding to a rational homotopy sphere for $H = B_{\ell-1}$
and not actually possible for $H = C_{\ell-1}$.

\item For $G = BC_{\ell}$ and $H = D_{\ell-1}$,
one has $\deg G \less \deg H = \{4\ell-1,4\ell-5\}$ 
and $\deg H \less \deg G = \{2\ell-1\}$,
unless some of these numbers are equal.
If that is the case, then 
$2m-1 = 4\ell-1$ and $n = m-1 = 4\ell-5 = 2\ell-1$, so that $\ell = 2$,
but the rational cohomology calculations for $\Sp(2)/\SO(2)$ and $\SO(5)/\SO(2)$
yield the cohomology of $S^2 \x S^7$ after all, contradicting
$m > n$.
If the numbers are distinct,
then $m = 2\ell$, so $n = 4\ell-5 > m$ in all but finitely many cases.

\item The finitely many cases where $\deg G = \deg H \dis \{2m-1\}$
and $G$ and $H$ do not lie in the same infinite family
all 
correspond to cases where $G/H$ 
is a rational cohomology $S^{2m-1}$.
\qedhere
\end{itemize}
\end{proof}
\begin{proof}[Proof of \Cref{thm:m>n}]
By \Cref{lem:pi3-rank,lem:noSimplePair}, we see $G$ and $H$ are both semisimple and that $\rk \pi_3 H \leq \rk \pi_3 G = 2$. 
Write $G = K_1 \. K_2$ with $K_j$ simple and normal in $G$,
and let $H_j$ be the identity component of $H \inter K_j$.
Each $H_j$ is normal in $H$, 
so there is a connected complement $H_0$ to $H_1 H_2$. 
As $\rk \pi_3 H \leq 2$, 
this implies at least one of the three $H_j$ must be trivial.

We claim the trivial factor must be $H_0$. Otherwise,
the trivial factor is one of $H_1$ or $H_2$, say the latter. 
The compositions of Lie algebra maps 
$\f h_0 \inc \fg \epi \fk_j$ 
must be injective for $j \in \{1,2\}$ lest $H_0$ lay in $K_1$ or $K_2$. 
Writing the corresponding lowercase Roman letters for ranks of Lie groups, 
we get $k_1 \geq h_0 + h_1$ and $k_2 \geq h_0$. 
As 
\[
k_1 + k_2 
=
g
=
h+1
= 
h_0 + h_1 + 1
\mathrlap,
\]
arithmetic then shows $k_2 = h_0 = 1$,
meaning $\fk_2 \iso \fh_0$ and $K_2 \iso H_0$ owing to semisimplicity.
But this would mean $K_1$ acted transitively on 
$
G/H 
= 
{K_1 H_0}\,/\,{H_1 H_0} 
,
$
violating our assumption that the $G$-action on $G/H$ was irreducible.
Thus $H_0$ is trivial. 

Since $H$ is assumed to be connected, 
we thus have $H = H_1 \. H_2$.
Writing $G \iso (K_1 \x K_2)/F$ for $F$ a finite central
subgroup, 
we see $H \iso (H_1 \x H_2)F/F$,
so $K_1/H_1 \x K_2/H_2 \lt G/H$
is a finite-sheeted cover with abelian fiber 
$F(H_1 \x H_2)/(H_1 \x H_2)$.
Since $F$ is central in $K_1 \x K_2$\revision,
the covering transformations are given by 
$(x_1,x_2)(H_1 \x H_2) \cdot f(H_1 \x H_2) 
= f\.(x_1,x_2)(H_1 \x H_2)$,
the left action of $f \in K_1 \x K_2$,
and hence 
are homotopic to the identity map since
$K_1 \x K_2$ is path-connected.
Thus the covering induces an isomorphism
\[\textstyle
\H(G/H) 
	\isoto 
\H\Big(\frac{K_1}{H_1} \x \frac{K_2}{H_2}\Big)\!{}^{\frac{F(H_1\x H_2)}{H_1 \x H_2}}
	= 
\H\Big(\frac{K_1}{H_1} \x \frac{K_2}{H_2}\Big)
\mathrlap.
\]
By the K\"unneth formula, $K_1/H_1$ and $K_2/H_2$
are rational cohomology spheres.
\end{proof}

\brmk
	Kamerich~\cite[\S13]{kamerich} conducts a similar analysis
	of ranks but with the stronger hypotheses 
	that $G$ is a product of two simple groups and $G/H$
	is homeomorphic to a product of two spheres,
	and without the hypothesis $m > n$.\footnote{\ 
	He coins the charming term \emph{auletic}
	for a pair leading to a space of the form $(K_1/H_1) \x (K_2/H_2)$,
	from the ancient Greek {\textgreek{αὐλός}}	(\emph{aulos}),
	a wind instrument with two pipes,
	one keyed by each hand.
	An online search for Kamerich finds in later life he was a teacher
	of mathematics 
	as well as a longtime bassoonist for the Arnhem Symphony Orchestra
	and a member of a local Renaissance music ensemble.
	A memorial for him in 2017 was titled ``His life was music.''
	}
	His analysis in the subsequent sections finds many 
	examples, all of the form $S^2 \x S^n$ for $n$ odd
	or $S^n \x S^{n'}$ for both $n$ and $n'$ odd.
\ermk

Next we consider the case of $m$, $m'$ both even.

\bthm\labell{thm:mm'}
Let $(G,H)$ be a pair of compact, connected Lie groups
such that $G$ acts effectively on $G/H$ and 
$G/H$ has the rational homotopy type 
of $S^m \x S^{m'}$
with $m,m'$ both even. 
Then $G/H$ is homeomorphic to a direct product of 
two homogeneous spheres of dimensions $m$ and $m'$.
\ethm
\begin{proof}
The Euler characteristic $\chi(G/H)$ is $4$,
so by a theorem of Wang~\cite[Thm.~I, p.~927]{wang1949}
on homogeneous spaces of positive Euler characteristic,
$G/H$ is 
homeomorphic to the direct product of homogeneous spaces
of compact simple groups.
There can \emph{a priori} be arbitrarily many such factors
with $\chi = 1$
and must also be either exactly two factors with $\chi = 2$
or one with $\chi = 4$.
We will rule out the former and latter possibilities.

Wang's proof descends from $G$ to its effective \revision{quotient}
$\ol G$ in $\Homeo G/H$,
which is a direct product of centerless simple groups $K_j$,
notes $\ol H$ is connected because $H$ is, and
argues $\ol H$ is the direct product of subgroups $H_j < K_j$,
which then are connected as well.
We cannot have any factor $K_j/H_j$ 
with $\chi(K_j/H_j) = 1$,
because $1 = \chi(K_j/H_j) = |W_{K_j}|/|W_{H_j}|$
would imply $K_j = H_j$.

 Thus if $K_j/H_j$ is a factor with $\chi = 4$,
then $K_j/H_j$ is actually all of $\ol G/\ol H$,
so $\ol G$ is simple.
We will show this too is impossible by
examining $\H(\ol G/\ol H) \iso \H(G/H)$
and showing it cannot be isomorphic to $\H(S^m \x S^{m'})$. 

Assume first $H$ is a maximal closed, connected subgroup of $K$.
Then by Borel and de Siebenthal's classification~\cite[p.~219]{boreldesiebenthal},
since $|W_{\ol G}|/|W_{\ol H}| = \chi(\ol G/\ol H) = 4$, 
the type of $(\bar\fg,\bar\fh)$ 
must be one of
$(A_3,A_2\x S^1)$,\, 
$(C_4,C_3 \x C_1)$, and
$(C_2,A_1\x S^1)$,
corresponding respectively to the homogeneous spaces 
$\CP^3$,\, 
$\HP^3$, and
$\widetilde{G}_2(\R^5)$ or $\Sp(2)/\big(\Sp(1)\+\U(1)\mnn\big)$.
But in each of these cases the cohomology ring is a truncated polynomial ring.

Now assume instead there is an intermediate closed, connected 
subgroup $\ol K$ between $\ol G$ and $\ol H$.
Then one must have 
$|W_{\ol G}|/|W_{\ol K}| = 2 = |W_{\ol K}|/|W_{\ol H}|$.
By Borel--de Siebenthal again, as Borel notes~\cite[p.~586]{borel1949remarks},
the pairs $(\ol G,\ol K)$ and $(\ol K,\ol H)$ must be of the types
$(B_r,D_r)$ or $\big(G_2,A_2\big)$, but $\ol K$ cannot be simultaneously
of types $D_2$ and $G_2$, nor simultaneously of types $A_2$ and $B_2$.

So we may conclude that $\ol G/\ol H$ is the product of 
homogeneous spaces $K_1/H_1$
and $K_2/H_2$ of simple groups, 
both factors of Euler characteristic two.
But Borel observes that pairs of this type give rise to honest spheres.
\end{proof}

\subsection{The classifications of irreducible actions}
Having shown that the only irreducible {\rspp}s
with even $m > n$ odd are covered by products of homogeneous
rational cohomology spheres,
it remains to evaluate and update the existing classifications
for $n > m$.

\subsubsection{Sphere pairs}
We first collect simply-connected irreducible rational sphere pairs in 
\Cref{table:EvenRationalSphere,table:OddRationalSphere};
notation is explained in \Cref{rmk:sphere-notation}.
The computations 
of $Z_G(H)^0$ and $H_S$ (as defined in \Cref{def:LH})
are not part of the standard table.
We tabulate them here and will use them in \Cref{sec:isotf}.
As we are only concerned with rational homotopy type, 
when $F$ is a finite subgroup of $H$ acting trivially on $G/H$,
we do not distinguish between $(G,H)$ and $(G/F,H/F)$
in our tabulation. 

\begin{table}
	\caption{Irreducible pairs \GH 
		with $\pi_1(G/H) = 0$ and $\H(G/H) \iso \H(S^{\mr{even}})$}%
	\labell{table:EvenRationalSphere}
	\begin{center}
		\begin{tabular}{ >{$}l<{$}|>{$}l<{$}|>{$}l<{$}|>{$}l<{$} }
			G & H & \dim G/H & G/H  \\
			\hline
			\SO(2k+1), \quad k\geq 1& \SO(2k) & 2k & S^{2k} \\
			\hline
			G_2 & \SU(3)& 6 & S^{6} \\
			\hline
			\Sp(2)& \Sp(1) \x \Sp(1)& 4&S^4 = \Spin(5)/\Spin(4) \\
		\end{tabular}
	\end{center}
\end{table}

\begin{table}
	\caption{Odd-dimensional 
		irreducible rational sphere pairs~%
		\cite[Table~2, p.~457]{onishchik1963transitive}%
		\cite[Table~10, p.~265]{onishchik}%
		\cite[Thm.~F.7.50-7.54, p.~195-196]{besse1978}%
		\cite[Ex.~B.7.13, p.~179]{Be}%
		\cite[p.~64-66]{kramer2002homogeneous}%
		\cite[Table~III, p.~154]{kapovitchziller2004}
		}
		\labell{table:OddRationalSphere}
		
	\begin{center}
		\resizebox{\columnwidth}{!}{
		\begin{tabular}{ >{$}l<{$} >{$}r<{$}|>{$}l<{$}|>{$}l<{$}|>{$}l<{$}|>{$}l<{$}|>{$}l<{$} }
			G && H & \Cen_G(H)^0 & H_S & \dim G/H & G/H  \\
			\hline
			\SU(k+1), & k\geq 2 & \SU(k)& \U(1) & H
				& 2k+1 & S^{2k+1}\\
			\SU(4) && \Sp(2) & \trivialgroup & \SU(2)\+\SU(2) & 5 & S^5 =\Spin(6)/\Spin(5)\\
			\SU(3) && \SO(3)_4& \trivialgroup & \SU(2) & 5 & \mbox{Wu manifold} \\
			\hline
			\SO(2k+1),&  k\geq 3 & \SO(2k-1) & \SO(2) & H
				& 4k-1 &V_2(\R^{2k+1})\\
			\Spin(9) && \Spin(7)& \trivialgroup & \SU(4) & 15 &S^{15}\\
			\Spin(7) && G_2 & \trivialgroup & \SU(3) & 7 &S^7\\
			\Spin(5)  \mathrlap{\,\,\mn \iso \Sp(2)}
			&& \SU(2) \iso \Sp(1) & \SU(2)^{\perp} & H
	& 7 &S^7 = \Sp(2)/\Sp(1)\\
\comment{\SO(5) && \SO(3)_2 & \SO(2) & H
			& 7 & V_2(\R^5) \iso \Sp(2)/\D\Sp(1)\\
			}
			\hline
			\Sp(k),&  k\geq 1 & \Sp(k-1)& \Sp(1) & H
				& 4k-1 &S^{4k-1} \\
			\Sp(2) \mathrlap{\,\,\mn \iso \Spin(5)}&& \SU(2)_{10}:{}^{\HH}\rho_{3\lambda_1} & \trivialgroup & i_{3,1} \U(1)  & 7 & \mbox{Berger $7$-space}\\
			\Sp(2)&&\big(\D \Sp(1)\mnn\big)_2 \,(\text{or } \SU(2)_2) & \SO(2)& H
				& 7 & V_2(\R^5)\iso \SO(5)/\SO(3)\\
			\hline
			\SO(2k),& k\geq 3 & \SO(2k-1) & \trivialgroup & \SO(2k-2) & 2k-1 &S^{2k-1}\\
			\hline
			G_2 && \SU(2)_1	\phantom{{}_{8}}:
			{}^{\R}\rho_{\lambda_1}& \SU(2)_{3} & H
				& 11 & V_2(\R^{7})\\
			G_2 && \SU(2)_{3}	\phantom{{}_{8}}:
			{}^{\R}\rho_{\lambda_1} + {}^{\R}\rho_{2\lambda_1}& \SU(2)_1& H
				& 11 & \\
			G_2 && \SO(3)_{4}	\phantom{{}_{8}}:
			2\cdot{}^{\R}\rho_{2\lambda_1}& \trivialgroup & i_{1,1} \U(1)  & 11 & \\
			G_2 && \SO(3)_{28}:{}^{\R}\rho_{6\lambda_1}& \trivialgroup & i_{5,1} \U(1)  & 11 & 
		\end{tabular}}
	\end{center}
\end{table}

\begin{table}
	\caption{Odd-dimensional 
		virtually effective but reducible rational sphere pairs~%
	}
	\labell{table:OddRationalSphereEffective}
	
	\begin{center}
			\begin{tabular}{ >{$}l<{$} >{$}r<{$}|>{$}l<{$}|>{$}l<{$}|>{$}l<{$}}
				G && H & \dim G/H & G/H  \\
				\hline
				\U(k+1), & k\geq 2 & \U(k)
				& 2k+1 & S^{2k+1}\\
				\hline
				\hline
				\SO(2k+1)\x \SO(2)&  k\geq 3 & \SO(2k-1)\D\SO(2)
				& 4k-1 &V_2(\R^{2k+1})\\
			\comment{\SO(5) \x \SO(2) && \SO(3)_2\D\SO(2)
				& 7 & V_2(\R^{5}) \iso \Sp(2)/\D\Sp(1)\\
				}
				\hline
				\hline
				\Sp(k) \x \Sp(1),&  k\geq 1 & \Sp(k-1)\D\Sp(1)
				& 4k-1 &S^{4k-1} \\
				\Sp(k) \x \U(1),&  k\geq 1 & \Sp(k-1)\D\U(1)
				& 4k-1 &S^{4k-1} \\
				\hline
				\Sp(2)\x \SO(2)&&\SU(2)_2 \D\SO(2) 
				& 7 & V_2(\R^5)\iso \SO(5)/\SO(3)\\
				\hline
				\hline
				G_2 \x \SU(2)&& \SU(2)_1 \D\SU(2)
				& 11 & V_2(\R^{7})\\
				G_2 \x \U(1)&& \SU(2)_1	\D\U(1)
				& 11 & V_2(\R^{7})\\
				\hline
				G_2 \x \SU(2) && \SU(2)_{3} \D\SU(2)
				& 11 & \\
				G_2 \x \U(1) && \SU(2)_{3} \D\U(1)
				& 11 & \\
		\end{tabular}
	\end{center}
\end{table}

\brmk[Notation for \Cref{table:OddRationalSphere}]\labell{rmk:table-explanation0}%
\labell{rmk:sphere-notation}

\ \\

\vspace{-1.5em}

\begin{itemize}
\item The block-diagonal product of two matrix groups is denoted with a ``$\defm\+$,''
	according with the standard notation $h_1 \+ h_2\: V_1 \+ V_2 \lt V_1 \+ V_2$, 
	%
and the trivial $1 \x 1$ block-diagonal subgroup 
	of a matrix group is denoted $\defm{[1]}$.
	Thus, for example, $\Sp(1) \+ [1]$ is the subgroup of $\Sp(2)$
	with elements $\mr{diag}(q,1)$ for $q \in \Sp(1)$.
	
\item For rational cohomology spheres in the list which are not spheres,
the symbols after the colon denote representations
named under the conventions of Chapter 4 
of Kramer's \emph{Habilitationsschrift}~\cite[Ch.~4]{kramer2002homogeneous},
identifying specific subgroups $H$ of $G$ isomorphic to the group preceding the colon. 
We are able to mostly avoid engaging directly with these representations,
but list them for completeness. 
The subscript attached to a subgroup $H$ denotes the \defd{Dynkin index} 
of $H$ in $G$, which can be identified with the order of the finite group 
$\pi_3(G/H)$~\cite[p.~257]{onishchik}. 
The Dynkin indices in the above table are taken from 
Onishchik~\cite[Table~2, p.~457]{onishchik1963transitive} 
and date back to Dynkin's work~\cite{dynkin1952}. 
Those $H$ without subscripts 
(except for $\Sp(0) < \Sp(1)$, which has no Dynkin index) 
have Dynkin index $1$ in the ambient group. 
The subscript ``$2$'' of $G_2$, though, is just its rank;
its Dynkin index in $\Spin(7)$ is~$1$.

\item
For the pair $\big(\Spin(5),\SU(2)\mnn\big)$,
the group denoted $\defm{\SU(2)^\perp}$
is the copy of $\SU(2)$ such that $\SU(2)\.\SU(2)^\perp = \Spin(4) < \Spin(5)$,
or equivalently $\Sp(1)\.\Sp(1)^\perp = \Sp(1)\+\Sp(1) < \Sp(2)$.

\item
The map $\defm{i_{p,q}}\: \U(1) \lt T^2$,
for $p$, $q \in \Z$,
denotes a parameterization
by $t \lmt \sigma_1(t^p)\sigma_2(t^q)$
of a circle in a maximal torus 
$T^2 = \sigma_1 \U(1) \. \sigma_2 \U(1)$ of $G$.
We may as well assume $p$ and $q$ are coprime.
In~$\SU(3)$, we take $T^2 = \mathrm{S}\big(\U(1)^3\big)$
the diagonal matrices of determinant~$1$,
and in~$\Sp(2)$, we take $T^2=\U(1)\+\U(1) <\Sp(1)\+\Sp(1)<\Sp(2)$,
with the standard parameterizations.
For~$G_2$, we will explain our choice of $\sigma_1$ and $\sigma_2$
in \Cref{rmk:sphere-discussion};
we will have $\s_1\U(1) \inter \s_2\U(2) \iso \{\pm 1\}$,
so that $i_{p,q}$ is not injective, 
but this is fine as we are only after 
a systematic way of naming the image circles $i_{p,q} \U(1)$.

Note that depending on the Weyl group of $G$, different pairs $(p,q)$
can correspond to conjugate subgroups
and hence result in equivalent pairs $(G,H)$. 
For example, two circular groups $i_{p,q}\U(1)$ and $i_{p',q'}\U(1)$ 
in $\Sp(2)$ are conjugate 
if and only if their ordered pairs of weights $(p,q)$ and $(p',q')$ 
lie in the same orbit of the Weyl group action on the weight lattice,
\emph{i.e.},
if the sets $\{p',q'\}$ and $\{\pm p,\pm q\}$ are equal.
\end{itemize}
\ermk

\begin{discussion}[Discussion of examples]\labell{rmk:sphere-discussion}
	
\ \\

\vspace{-1.5em}

\begin{enumerate}[label=(\alph*)]
\item\labell{rmk:SO(3)4}
The standard embedding $\SO(3) \leq \SU(3)$ has Dynkin index $4$ because 
the Dynkin index of the composite $\SO(3) \inc \SU(3) \inc G_2$ 
is the product of the Dynkin indices of the two intermediate embeddings~\cite[Prop.~9, p.~58]{onishchik},
and the index of $\SU(3)$ in $G_2$ is $1$ since $\pi_3(S^6)  = 0$,
while the index in $G_2$ of $\SO(3) < \SU(3)$
(which is not a maximal proper, closed, connected subgroup of $G_2$) 
must be $4$ because the $\SO(3)$ subgroups of $G_2$
all have index $4$ or $28$,
and those of index $28$ are maximal in $G_2$ (see \ref{rmk:G2}).\footnote{\ 
A proof not involving knowledge of $G_2$
is possible by computing that the cohomology of the Wu manifold
$W = \SU(3)/\SO(3)$ is the sum of $H^0(W) \iso \Z \iso H^5(W)$
and $H^3(W) \iso \Z/2 \iso H^2(W;\Z/2)$
then taking the homotopy fiber $F$ of a representing map $W \lt K(\Z/2,2)$
to kill $\pi_2(W) = \Z/2$
and finding the order of $H^4(F) \iso H_3(F) \iso \pi_3(F) \iso \pi_3(W)$ 
through the Serre spectral sequence of $K(\Z/2,1) \to F \to W$. 
}

\item\labell{rmk:S15}
The embedding $\Spin(7) \longmono \Spin(9)$
is not the standard one covering $\SO(7) \+ [1]^{\oplus 2} \longinc \SO(9)$,
but the composition of a standard embedding $\Spin(8) \longinc \Spin(9)$
with a lift $\defm\iota\: \Spin(7) \lt \Spin(8)$
of one of the (faithful) half-spin representations of $\Spin(7)$ on $\R^8$.
We will encounter~$\iota$ again in \Cref{rmk:example-discussion}\ref{rmk:S6S7}.

\item\labell{rmk:diagonalSp(1)}
The standard embedding $\defm{\SO(3)_2} = \SO(3) \+ [1]^{\oplus 2}$ 
of $\SO(3)$ in $\SO(5)$ has Dynkin index $2$ 
because the real Stiefel manifold $V = V_2(\R^5) = \SO(5)/\SO(3)$
has $\pi_3 \iso \Z/2$~\cite[25.6, p.~132]{steenrod1951}. 
This manifold is the total space of a bundle
$\SO(4)/\SO(3) \to \SO(5)/\SO(3) \to \SO(5)/\SO(4)$,
which Kapovitch--Ziller~\cite[p.~155]{kapovitchziller2004}
observe is diffeomorphic to 
$\Sp(1) \to \Sp(2)/\D\Sp(1) \to \HP^1$,
the unit tangent bundle of the quaternionic projective line. 
But $\D\Sp(1)$ is conjugate in $\Sp(2)$ 
to the standard subgroup $\defm{\SU(2)_2}$~\cite[2.14, p.~23]{mimuratoda}, 
so $\Sp(2)/\D\Sp(1) \homeo \Sp(2)/\SU(2)$ 
and both subgroups have Dynkin index $2$ in $\Sp(2)$ just as $\SO(3)$ does in $\SO(5)$.

\item\labell{rmk:SU(2)10}
The irreducible representation $\defm{{}^\Quat\rho_{3\lambda_1}}$ 
of $\SU(2) \iso \Sp(1)$
(in Kramer's notation~\cite[Ch.~4]{kramer2002homogeneous})
is of complex dimension $4$, 
hence of quaternionic dimension $2$, 
and induces a nonstandard embedding 
$\defm{\SU(2)_{10}}$ of $\SU(2)$ in $\Sp(2)$. 
Berger~\cite[p.~237--239]{berger1961} 
studied the topological and geometric properties of the quotient 
$\Sp(2)/\,\mn\SU(2)_{10}
$,
now called the \emph{Berger $7$-space},
which also appears in Onishchik's work~\cite[p.~457]{onishchik1963transitive} as $\SO(5)/\,\mn\defm{\SO(3)_{10}}$,
where the denominator subgroup is the nonstandard $\SO(3)$ subgroup of 
Dynkin index $10$ in $\SO(5)$ 
arising from the representation of $\SO(3)$ 
on the $5$-dimensional space $\R\{xy,yz,zx,x^2-y^2,y^2-z^2\}$
of homogeneous quadratic polynomials on $\R^3 = \R\{x,y,z\}$.
This representation is $\defm{{}^{\R}\rho_{4\lambda_1}}$ in Kramer's notation. 	
According to Berger~\cite[p.~237]{berger1961},
one maximal torus of the embedded subgroup $\SU(2)_{10}$ in $\Sp(2)$
is the diagonal circle $i_{3,1}\U(1) < \U(1) \+ \U(1) 
< \Sp(2)$.

\item\labell{rmk:G2}
 According to Dynkin~\cite[p.~411]{dynkin1952}, 
there are four conjugacy  classes of $A_1$ subgroups of $G_2$,
to wit, $\defm{\SU(2)_1}$, $\defm{\SU(2)_3}$,
$\defm{\SO(3)_4}$, and $\defm{\SO(3)_{28}}$,
of which only the first two are regular.
Let $\alpha$ and $\beta$ 
respectively be the short and the long simple roots of $G_2$
with respect to a fixed maximal torus,
so that $||\alpha||^2=1$, $||\beta||^2=3$, and $-B(\alpha,\beta)=-3/2$.
Then $\g = 3\a+2\b$ is another long root
and is orthogonal to $\a$,
so that $h_\a$ and $h_\g$ give an orthogonal basis of $\ft^2$.
According to Mayanskiy~\cite[p.~4]{mayanskiy2016}, 
any coroot corresponding to a long
root (such as $h_\g$) is tangent to the maximal torus
of an $\SU(2)_1$ subgroup of $G_2$ 
while any corresponding to a short root (such as $h_\a$)
is tangent to that of an $\SU(2)_3$ subgroup.
Mayanskiy~\cite[p.~17]{mayanskiy2016}
also notes that $\SO(3)_4$ and $\SO(3)_{28}$ subgroups of $G_2$
can be found with Cartan algebras respectively spanned by $2h_{3\alpha+\beta}$
and $14 h_{9\alpha+5\beta}$.
By pairing with $2\beta+3\alpha$ and $\alpha$,
we see these vectors are respectively
$h_{2\beta+3\alpha}+h_{\alpha}$
and
$5h_{2\beta+3\alpha}+h_{\alpha}$.
Dynkin shows these representative $\SU(2)_1$ and $\SU(2)_3$
subgroups centralize one another,
yielding a virtual direct product $\SO(4)$ 
which is a maximal proper closed subgroup of $G_2$
and can be chosen to have $\alpha$ and
$2\beta+3\alpha$ as its simple roots.
If we take $T^2$ to be the product $\U(1)\. \U(1) < \SU(2)_1 \. \SU(2)_3 = \SO(4)$,
then the maximal tori of $\SU(2)_1$, $\SU(2)_3$, $\SO(3)_4$, $\SO(3)_{28}$ are respectively realized as $i_{1,0}  \U(1)$, $i_{0,1}  \U(1)$, $i_{1,1} \U(1)$, $i_{5,1}\U(1)$.
\end{enumerate}
\end{discussion}

\brmk[History]
The list of pairs resulting in actual spheres is a classical
result due to Montgomery--Samelson~\cite{montgomerysamelson1943} 
and Borel~\cite[Thm.~3, p.~486]{borel1949remarks}\cite[Thm.~3]{borelCR1950Cplan}.
Borel also showed in 1953 that a simply connected integral homology sphere is a sphere~\cite[(4.61), p.~455]{borel1953bouts},
and Bredon~\cite{bredon1961homogeneous} showed in 1961 that the only integral homology sphere that is not a sphere
is the Poincar\'e homology sphere $\SO(3)/(\mbox{icosahedral group})$.
Matsushima~\cite{matsushima1951type} showed that 
the infinite series in \Cref{table:OddRationalSphere}
were the only odd-dimensional rational homology spheres 
$G/H$ with (perforce simple) $G$ acting irreducibly and virtually effectively,
up to possible low-dimensional exceptions,
and that $G_2$ was the only possible exceptional $G$ 
for $H$ of type $A$ or~$C$; 
Montgomery and Samelson had already done this 
for~$B$ and~$D$.
\ermk

\begin{proof}[Proof of centralizers and maximal regular subgroups for \Cref{table:OddRationalSphere}]

In most of these examples, $H$ is regular,
which by \Cref{thm:isotropyOddSphere} implies it is the maximal
closed, connected subgroup $H_S$ containing $S$,
as defined in \Cref{def:LH}.
We can usually verify this by checking $\rk Z_G(H)H/H = 1$
using \Cref{thm:enlargement-corank1}.
Then $Z_G(H)^0 H/H$ is either isomorphic to $S^1$ or a group of type $A_1$.
When it is $S^1$,
it is often simple to verify this directly,
but it is simpler to note that were it of type $A_1$,
then $G$ would admit a maximal-rank subgroup $H \x A_1$,
which can be ruled out by consultation with Borel--de Siebenthal~\cite[p.~219]{boreldesiebenthal},
and we do not spell out this verification.
When the centralizer really is of type $A_1$,
our approach is \emph{ad hoc}.

\bitem
\item
For $\big(\SU(k+1),\SU(k)\+[1]\big)$, 
the circle $\big\{\!\diag(z,\ldots,z,z^{-k})\big\}$ centralizes $H$,
so $H = H_S$.
\item
  For $\big(\SO(2k+1),\SO(2k-1) \+ [1]^{\oplus 2}\big)$,
the circle $[1]^{\oplus 2k-1} \+ \SO(2)$ centralizes $H$, so $H = H_S$.
\item
  For $\big(\Sp(k),\Sp(k-1)\+[1]\big)$,
  the $A_1$ subgroup $[1]^{\oplus k} \+ \Sp(1)$ centralizes $H$, so $H = H_S$.
\item
	For $\big(\Sp(2),\SU(2)\big)$, the diagonal $\D\U(1)$ centralizes $H$,
	so $H = H_S$.
\item
  For $\big(G_2,\SU(2)\big)$, the $\SU(2)$ can be of index $1$ or $3$,
  and each centralizes the other.

\item
For $\big(G_2,\SO(3)\big)$,
where $H = \SO(3)$ is either $\SO(3)_4$ or $\SO(3)_{28}$,
it is known~\cite[Prop.~4.1]{mayanskiy2016}
that the complement to $\f{so}(3)_\C$ 
in the adjoint representation of $\SO(3)$ on $(\f{g}_2)_\C$
is a sum of nontrivial irreducible representations,
so $Z_G(H)^0 = 1$,
and hence $H$ is not regular.
If $S$ is a maximal torus of $H$, then
it is centralized by a maximal torus $T < G_2$ containing it,
and so regular, and if $H_S > S$,
then $H_S$ is an $A_1$ subgroup;
but we have seen in \Cref{rmk:sphere-discussion}\ref{rmk:G2}
that up to conjugacy there are only the four such subgroups,
that only the $\SU(2)$ subgroups are regular,
and that these $A_1$ groups have distinct maximal tori,
so $H_S = S$.

\item
	For $\big(\SU(4),\Sp(2)\big)$,
	one sees manually the centralizer is $\D\{\pm 1\}$,
	so $H$ is not regular.
	A maximal torus of $\Sp(2)$ is given by 
	$\big\{\!\diag(z,z\-,w,w\-):w,z \in \U(1)\big\}$,
	which is also maximal in $H' = \SU(2) \+ \SU(2)$,
	and $H'$ is centralized by the circle
	$\big\{\!\diag(z,z\-,z,z\-):z \in \U(1)\big\}$,
	so it is regular.
	In fact, $(G,H')$ is a \rspp.
	To see $H'$ is maximal regular,
	note that by \Cref{thm:extend-H},
	if we had $H_S > H'$,
	then $(G,H_S)$ would be a \rsp,
	but consulting \Cref{table:OddRationalSphere},
	there are no candidates.
\item
	For $\big(\SU(3),\SO(3)\big)$,
	one computes directly that the centralizer is trivial.
	The subgroup $S = \SO(2) \+ [1]$
	is a maximal torus, which is contained in $H' = \SU(2)\+[1]$.
	The circle $\big\{\!\diag(z,z,z^{-2})\big\}$
	centralizes $H'$, which thus is regular,
	and $(G,H')$ is a \rsp,
	so by \Cref{thm:isotropyOddSphere},
	$H' = H_S$.
\item
	In $\big(\Spin(9),\Spin(7)\big)$,
	as we have discussed in \Cref{rmk:sphere-discussion}\ref{rmk:S15},
	the embedding of $\Spin(7)$ in $\Spin(8) < \Spin(9)$
	is nonstandard,
	and hence is not centralized by the
	expected $\Spin(2)$.
	Consulting the (yet to come) \Cref{rmk:example-discussion}\ref{rmk:S6S7},
	we see $\Spin(7)$ contains an $\SU(4)$
	subgroup~$H'$ sharing its maximal torus
	and such that $(G,H')$ is a \rspp.
	The double-covering $\Spin(8) \lt \SO(8)$ takes
	$H'$ isomorphically to the standard $\SU(4) < \U(4) < \SO(8)$, which is centralized by $\D_4\U(1)$,
	and this circle is double-covered by a circle in 
	$\Spin(8)$ centralizing $H'$, which thus is regular.
	\comment{If it were centralized by an $A_1$ subgroup of $G$,
	then $\Spin(9)$ would contain a $B_3 \x A_1$
	subgroup, which it does not by 
	Borel--de Siebenthal.
	}
	If we had $H_S > H'$,
	then by \Cref{thm:extend-H},
	$(G,H_S)$ would be a \rsp,
	but consulting \Cref{table:OddRationalSphere},
	there are no candidates.

\item
	For $\big(\Spin(7),G_2\big)$,
	the centralizer is $Z(G)$ by Borel--de Siebenthal.
	But (see \Cref{rmk:example-discussion}\ref{rmk:S6S7}),
	$G_2$ shares a maximal torus with a 
	standardly embedded $\SU(3)$ subgroup $H'$
	of the copy of $\SU(4)$ in $\Spin(7)$,
	and $\SU(3)$ is centralized
	within $\SU(4)$ by $\big\{\!\diag(z,z,z,z^{-3})\big\}$,
	and hence is regular.
	To see $H' = H_S$,
	again note that $(G,H')$ is a \rspp
	but $(G,K)$ is not a \rsp for any $K > H'$.

\item
	For $\big(\Sp(2),\SU(2)_{10}\big)$,
	note that the maximal torus $S = i_{3,1}\U(1)$ of $\SU(2)_{10}$
	is centralized by the diagonal maximal torus,
	hence is regular. Since $H/S \homeo S^2$,
	the Serre spectral sequence of
	$H/S\rightarrow G/S \rightarrow G/H$ collapses,
	so $(G,S)$ is a \rspp.
	By \Cref{thm:isotropyOddSphere}, 
	either $H_S = S$
	or $(G,H_S)$ is a \rsp,
	but there are no possibilities for $H_S > S$
	in \Cref{table:OddRationalSphere}. Hence, we have $H_S = S$.

\item
	For $\big(\SO(2k),\SO(2k-1)\big)$,
	the centralizer is $\pm[1]^{\oplus 2k}$,
	so $H$ is not regular, but $H$
	shares the maximal torus $S = \SO(2)^{\oplus k-1}$
	with $H' = \SO(2k-2)$
	and $(G,H')$ is a \rspp
	with $H'$ centralized by $[1]^{\oplus 2k-2} \+ \SO(2)$.
	Again $H' = H_S$ by lack of 
	candidates $K > H'$ with $(G,K)$ a \rsp.
	\qedhere
\eitem
\end{proof}

\subsubsection{Sphere product pairs}\labell{sec:sphere-product-class}
Our \Cref{thm:m>n} 
shows the existing
classifications of Kramer~\cite[Thm.~3.15, Chs.~5--6]{kramer2002homogeneous} for $n>m\geq 4$
and 
Wolfrom~\cite[Thm.~2.1]{wolfrom2002} for $n>m=2$ 
should contain all simply-connected irreducible cases.\footnote{\ 
	Note that we must have $n, m\geq 2$ 
	due to the simple-connectivity hypothesis.}
As mentioned earlier, this is nearly but not entirely true, 
since we must re-analyze the candidates 
from these lists to make
sure they have the correct cohomology, 
as explained in \Cref{rmk:correction},
and since there are a few omitted cases.
Kamerich, in classifying homogeneous products of two spheres, 
first lists all the possible pairs of Lie algebras satisfying
the condition of~\Cref{thm:deg}
as candidates~\cite[Table~1, p.~55--60]{kamerich}.
His list includes all the examples in Kramer and Wolfrom's lists,
and after extensive checking, we are confident his list is complete.
In this subsection we justify this conclusion,
which leads to 
\Cref{table:Kramer+Wolfrom}.

After establishing notation, we discuss how our table 
and the work of previous authors 
relate to each other in \Cref{rmk:comparison-Kamerich}
and \Cref{rmk:comparison-Kramer},
elaborate on several cases in \Cref{rmk:example-discussion},
and then finally prove that the candidates do have the correct cohomology
in several representative cases,
meant to convey the flavor of this somewhat lengthy verification process.

Although $Z_G(H)^0$ and $H_S$ appear in \Cref{table:main},
we will only need them later in \Cref{sec:isotf} to
determine \isotfity,
so we will defer their calculation until then.

\begin{notation}[for \Cref{table:Kramer+Wolfrom}]\labell{rmk:table-explanation}
This discussion will strictly cover notation \emph{per se};
\Cref{rmk:example-discussion} will describe the less obvious embeddings of $H$ in $G$.

\begin{itemize}
\item
 In the first four lines of the table, 
 if $Z_G\big(i_{p,q}\U(1)\mnn\big){}^0$ contains an $A_1$ subgroup,
then its complexified Lie algebra $\fa_1 \otimes \C = \C\{e_\b,e_{-\b},h_\b\}$ 
is centralized by $\f h = \f s = \R u$,
so that~$\b(u) = 0$ and 
the hence the tangent line $\R h_\b/i < \ft^2$ 
to its maximal torus is $B$-orthogonal to $\f s$.
This maximal torus will be $i_{p^*,q^*}\U(1)$,
for a pair $(p^*,q^*) \in \Z^2$ 
we will determine in the proof of \Cref{thm:isotfclassification}.
We then write $\defm{(A_1)_{p^*,q^*}}$
for this $A_1$ subgroup. 
In the case, $p\.q=1,-2$, 
we have written $\SU(2)_{p+2q,-q-2p} = \SU(2)_{\pm 3,\mp 3} = \SU(2)_{\pm 1,\mp 1}$
and $\SU(2)_{\pm 3,0} = \SU(2)_{\pm 1,0}$
and $\SU(2)_{0,\pm 3} = \SU(2)_{0,\pm 1}$. 
\textbf{In all these entries, this classification
	is only unique up to the assumption
	$\fs\less\{0\}$ meets the fundamental closed Weyl chamber.}
	The conditions on $p$ and $q$ necessary
	to ensure this are left to the reader.
\item
	In the fourth line of the table, 
	the subpair $(\defm{K_1},\defm{H_1})$ 
	can be any pair from \Cref{table:OddRationalSphere} 
	such that $\Cen_{K_1}(H_1)^0\neq \trivialgroup$
and the embedding $i_{p,q}\:\U(1)\longrightarrow T^2$ is as defined in \Cref{rmk:table-explanation0}
for $T^2 \iso \U(1) \x \U(1)$ the product of a maximal torus of 
$\Cen_{K_1}(H_1)$ and a maximal torus of $\Sp(1)$.

\item In the fifth line of the table, 
the pairs $(\defm{K_j},\defm{H_j})$ 
for $j \in \{1,2\}$ 
are respectively drawn from 
\Cref{table:OddRationalSphere} and 
\Cref{table:EvenRationalSphere}.

	\item The notation $\defm{\widetilde G_k(\C^n)}$
	reflects that the space
	can be seen as the total space of a bundle over
	the complex Grassmannian
	$G_k(\C^n) = \SU(n)/ \mr{S}\big(\U(k) \+ \U(n-k)\mnn\big)$
	with fiber the circle
	$\mr{S}\big(\U(k) \+ \U(n-k)\mnn\big) / \big(\SU(k)\+\SU(n-k)\mnn\big)$.
	We might think of this as an \emph{oriented complex Grassmannian}
	by analogy with the oriented real Grassmannians.

\end{itemize}
\end{notation}

\begin{discussion}[Comparison with Kamerich]\labell{rmk:comparison-Kamerich}
	Recall that Kamerich 
	wants to determine pairs $(G,H)$
	with $G/H$ homeomorphic to a product of two spheres,
	considering all parities of dimensions for the two potential spheres.
	He rules out 
	simple groups $G$ acting irreducibly on homogeneous spaces $G/H$ 
	homeomorphic
	to $S^1 \x S^n$ for all $n$,
	to $S^2 \x S^m$ for $m \geq 2$ even~\cite[\S10]{kamerich},
	through an analysis that shows there are no possibilities rationally,
	and to $S^2 \x S^n$ for $n \geq 3$ odd,
	through an analysis that leaves 
	$G \in \big\{\SU(3),\Sp(2),G_2\big\}$
	as the only possibilities rationally.\footnote{\ 
		This argument, as written, would incorrectly rule out $\SU(3)$,
		on considering 
		the long exact homotopy sequence of $H \to G \to G/H$,
		but this is due to the claim that $\pi_4(S^2) \iso \Z$
		(probably a typo),
		which does not affect his argument since he only 	
		needs that this group is nonzero;
		of course $\pi_4(S^2) \iso \Z/2$.
	}
	
	He applies \Cref{thm:deg}
	to determine corresponding Lie algebra pairs
	$(\fg,\fh)$ 
	with $\fg$ simple and $\fh$ semisimple
	satisfying the degree criterion,
	excluding cases where the Dynkin index is not $1$
	and (1) $\fh$ is simple
	and $m,n \geq 5$ or
	(2) $\fg$ is exceptional,
	since $\pi_3(G/H) \neq 0$
	is not consistent with a product of spheres
	unless one of the factors is $S^2$ or $S^3$.
	We briefly run through the cases from his resulting
	Table 1~\cite[p.~55]{kamerich}
	that we exclude and our reasoning.
	
	Case 11 is a typo, which should be $(B_3,A_1)$.
	The cases 13--16, all of type 
	 $(B_3,A_1\+ A_1)$,
	 respectively correspond to the cases
	$\big(\SO(7),\SO(4)\mnn\big)$
	(which we have absorbed 
	as $(B_3,D_2)$ 
	in our analysis),
	$\big(\SO(7),\SO(3)\oplus \SU(2)$,
	$\big(\SO(7),\SO(3)^{\oplus 2} \oplus [1]\big)$,
	and $\smash{\big(\Spin(7),\SO(4)\mnn\big)}$,
	which is rationally equivalent 
	to $\smash{\widetilde G_3(\R^8)}$, 
	and which we noted in 
	\Cref{rmk:correction}
	has the wrong cohomology. 
	We write $(B_4,A_3)$, case 17,
	as $\big(\SO(2k+1),\SO(2k-2)\mnn\big)$,
	whereas case 18, also $(B_4,A_3)$, is our $\big(\Spin(9),\SU(4)\mnn\big)$.
	The cases of type $(C_3,A_1 \+ A_1)$,
	comprising one unlabeled case and Kamerich's cases 28--30,
	are respectively
	$\big(\Sp(3),\SO(3)\.\D_3\Sp(1)\big)$,
	$\big(\Sp(3),\Sp(1)^{\oplus 2} \oplus [1]\big)$,
	$\big(\Sp(3),\Sp(1)\.\SU(2)_{10}\big)$, and
	$\big(\Sp(3),\Sp(1)\+\D_2\Sp(1)\big)$.
	Cases 36 and 37, both $(D_4,A_3)$,
	we have written as $\big(\SO(8),\SO(6)\mnn\big)$
	and $\big(\Spin(8),\SU(4)\mnn\big)$
	respectively.
	We do not include $(B_n,B_{n-1})$, case 19,
	since the Stiefel manifolds $V_2(\R^{2n+1})$	
	are rationally $(4n-1)$-spheres;
	$(D_4,A_1 \+ B_2)$, case 43,
	both embeddings of which represent $\widetilde G_3(\R^8)$;
	or $(G_2,A_1)$, case 48, 
	all four subcases of which are rational spheres.
	Since all yield spheres,
	we omit the cases 20, 25, 31, 33, 39, and 40,
	respectively 
	$(B_4,B_3) = \big(\Spin(9),\Spin(7)\mnn\big)$,
	$(B_3,G_2) = \big(\Spin(7),G_2\big)$,
	$(C_3,B_2) = \big(\Sp(3),\Sp(2)\mnn\big)$,
	$(C_n,C_{n-1}) = \big(\Sp(n),\Sp(n-1)\mnn\big)$, 
	$(D_n,B_{n-2}) = \big(\SO(4n),\SO(4n-1)\mnn\big)$,
	and $(D_4,B_3) = \big(\Spin(8),\Spin(7)\mnn\big)$
	either of the conjugacy classes of inclusions
	not covering $\big(\SO(8),\SO(7)\mnn\big)$---%
	these two $(D_4,B_3)$ pairs are equivalent by triality.
	
	Kamerich next deals with the possibility $G$ is
	a simply-connected connected group times a torus (possibly trivial),
	but not simple.
	He is interested in virtually effective, irreducible actions,
	and finds the cases $(K_1,H_1) \x (K_2,H_2)$
	and another case which he subjects to an analysis similar to that of our \Cref{thm:m>n},
	but without the restriction on dimensions of spheres,
	to arrive at his Table 2 of candidates satisfying the Onishchik degree criterion.
	When we exclude cases with two odd dimensions,
	there remain only the examples $\big(K_1 \x \Sp(1), H_1 \. i_{p,q} \U(1)\mnn\big)$
	for $K_1 \in \big\{ \SU(3), \Sp(2), G_2 \big\}$
	and $H_1 \iso \SU(2)$ in cases 1--4,
	and for $(K_1,H_1)$ equal to 
	$\big(\SU(k+1),\SU(k)\mnn\big)$, 
	$\big(\Spin(2k+1),\Spin(2k-1)\mnn\big)$, and 
	$\big(\Sp(k+1),\Sp(k)\mnn\big)$ respectively in cases 5--7.
	We do not find any irreducible {\rspp}s 
	not found in Kamerich's preliminary lists. 

	As far as his real goal goes,
	pairs giving rise to genuine homogeneous spaces 
	$G/H = S^m \x S^n$,
	for $m$ even and $n$ odd,
	he finds only product pairs, 
	$\Spin(7)/\SU(3) = \SO(8)/\SO(6) \iso S^6 \x S^7$,
	and three classes of examples with $G$ semisimple:
	the class $\Sp(2) / i_{p,q}\U(1) \iso S^2 \x S^3$ with $\gcd(p,q) = 1$,
	the class $\big(\SU(2n+1) \x \Sp(1)\mnn\big) / \big(\SU(2n)\.i_{p,q}\U(1)\mnn\big)
	\iso S^{4n+1} \x S^2$
	with $\gcd(p,q) = 1$ and $q|n$,
	and the class $\big(\Sp(n) \x \Sp(1)\mnn\big) / \big(\Sp(n-1)\.i_{p,1}\U(1)\mnn\big)
	\iso S^{4n-1} \x S^2$.
	Kamerich notes that when $G/H$ is homeomorphic
	to a product of spheres, it is in fact diffeomorphic to the standard product,
	but also finds an example which is merely homotopy equivalent without being homeomorphic.

\end{discussion}
\begin{discussion}[Comparison with later work]\labell{rmk:comparison-Kramer}
	Kramer and his students
	want to find pairs $(G,H)$ such that $\H(G/H;\Z)$
	is isomorphic to $\H(S^m \x S^n;\Z)$ with $m$ and $n$ both odd
	or $n > m$ with $m$ even and $n$ odd,
	but need to find candidates with the desired rational 
	cohomology along the way.
	The union of their classifications
	contains most of what we need,
	with a few exceptions, as we have already noted.
	Briefly, these are as follows:
	the pair $\big(\SU(4),\SO(4)\mnn\big)$ is missing,
	as are two of the four $(C_3,A_1\+ A_1)$ pairs,
	namely
	$\big(\Sp(3),\Sp(1)\+\D_2\Sp(1)\big)$ 
	and 
	$\big(\Sp(3),\SO(3)\.\D_3\Sp(1)\big)$.
	The case $\big(G_2 \x \Sp(2),  \Sp(1)\.\D\Sp(1)\.\Sp(1)\mnn\big)$
	is mentioned in passing (p. 73) and shown to not have
	the correct integral cohomology, but not tabulated
	as having the desired rational cohomology.
	Also, 
	Kramer's pairs $\big(\Spin(7),\SO(4)\big)$ and $\big(\SO(8),\SO(3) \x \SO(5)\mnn\big)$,
	as noted in \Cref{rmk:correction},
	must be excluded as they do not give the proper cohomology.
	Apart from this, our results agree.
	
	When Kramer goes on to consider cohomology over $\Z$,
	he finds the following examples with the integral cohomology of 
	a product $S^m \x S^n$ with $m$ even less than $n$ odd:
	$\big(\Spin(9),\SU(4)\mnn\big)$
	and $\big(\Spin(10),\SU(5)\mnn\big)$,
	giving rise to the same quotient,
	two of the three cases $\big(\Sp(3),\Sp(1)\.\Sp(1)\mnn\big)$,\footnote{\ 
	He does not consider the remaining case,
	which one can however show has $H^8 \iso \Z/3$.
	In fact, the pair $\big(\Sp(3),\Sp(1)\+\SU(2)_{10}\big)$ 
	also appears to have torsion, namely $H^8 \iso \Z/91$.
	}
	the pairs $\big(\SO(2k),\SO(2k-2)\mnn\big)$
	giving rise to Stiefel manifolds $V_2(\R^{2k})$, 
	and $\big(\SU(5),\SU(2)\+\SU(3)\mnn\big)$.
	
	The majority of our calculations of $Z_{G}(H)^0$ 
	can also be found in Kramer's work. 
	However, 
	we have found that $Z_{G}(H)^0$ for the pairs $(G,H) = \big(\SO(7),\SO(3)\+\SU(2)\mnn\big)$,
	$\big(\Sp(4),\SU(4)\mnn\big)$,
	and $\big(\Sp(3),\SU(3)\mnn\big)$ 
	is in fact $\U(1)$, rather than the trivial group,
	our computations of $Z_G(H)^0$
	for the first four lines
	of \Cref{table:Kramer+Wolfrom} seem to be original,
	and the excluded cases seem not to have been completed.
\end{discussion}

\begin{discussion}[General patterns]\labell{rmk:general-discussion}
  \ \begin{itemize} 
\item
	From \Cref{table:Kramer+Wolfrom}, 
	we see that whenever $Z_G(H)^0 \iso S^1$,
	we also have $H = H_S$. 
	We can see this more directly by observing that if 
	we could have
	$Z_G(H)^0 \iso S^1$ and $H < H_S$,
	then $(G,H_S)$ would be a rational sphere pair
	with $Z_G(H_S)^0 \leq Z_G(H)^0 = S^1$ a circle,
	and this does not occur in the classification \Cref{table:OddRationalSphere},
	but it would be nice to see a conceptual reason for this that does not pass
	through some classification.
\item 
	The converse also seems to nearly hold: if $H \ncong S^1$
	and $H = H_S$, then in all but one case, we have $Z_G(H)^0 H/H \iso S^1$.
	Again, we do not have a direct argument why this should be.
\end{itemize}
\end{discussion}

\begin{discussion}[Discussion of selected cases]\labell{rmk:example-discussion}
\ 
\begin{enumerate}[label=(\alph*)]
	\item The quotient $\SU(4)/\SO(4)$ is not included by 
	Kramer~\cite{kramer2002homogeneous}, 
	but we find it in Kamerich's 
	classification~\cite[Table~1, p.~55]{kamerich} 
	as the unlabeled case before case 5.
	To compare $\SU(4)/\SO(4)$ with $\SU(4) / \big( \SU(2) \+ \SU(2)\mnn\big)$, 
	we first note that the quotients are both simply-connected, 
	but have nonisomorphic $\pi_2$
	since
	$\pi_1 \SO(4)$ and $\pi_1\big(\SU(2) \+ \SU(2)\mnn\big)$
	are nonisomorphic,
	and different Dynkin indices, as Kamerich notes.
	In terms of more familiar spaces, we have
	\begin{align*}
		\frac{\SU(4)}{\SO(4)} 
			&\homeo 
		\frac{\SU(4)/\Z_2}{\SO(4)/\Z_2} 
			\homeo 
		\frac{\SO(6)}{\SO(3) \+ \SO(3)}
			=
		\widetilde{G}_3(\R^6)
	\mathrlap,\\
	\\
		\smash{\frac{\SU(4)}{\SU(2) \+ \SU(2)}} 
			&\homeo 
		\smash{\frac{\SU(4)/\Z_2}{\big(\SU(2) \+ \SU(2)\mnn\big)/\Z_2} 
			\homeo 
		\frac{\SO(6)}{\SO(4)}=V_2(\R^6)}
	\mathrlap.
	\end{align*} 
	For explicit matrix expressions for the isomorphisms,
	see Yokota~\cite{yokota1979}.

\item The quotients 
$\Sp(3)/\big(\Sp(1)\+\D_2\Sp(1)\mn\big)
\homeo
\Sp(3)/\big(\Sp(1)\+\SU(2)\mnn\big)$ 
and $\Sp(3)/\big(\SO(3)\.\D_3\Sp(1)\mn\big)$ 
are not included by Kramer~\cite{kramer2002homogeneous}, 
but are by Kamerich~\cite[Table~1, p.~58]{kamerich}. 
Here $\D_2\Sp(1)$ and $\D_3\Sp(1)$ 
denote the diagonally embedded copies of $\Sp(1)$ in 
$\Sp(2)$ and $\Sp(3)$ respectively 
and $\SO(3)$ is the subgroup of real matrices in $\Sp(3)$. 
As noted in \Cref{rmk:sphere-discussion}\ref{rmk:diagonalSp(1)}, 
$\D_2\Sp(1)$ is conjugate to $\SU(2)$ in $\Sp(2)$, 
explaining the above diffeomorphism. 
We will check both pairs have the
cohomology of $S^4 \x S^{11}$ in the proof to follow.

\item\labell{rmk:S6S7}
The quotients $\Spin(7)/\SU(3)$ and $\SO(8)/\SO(6)$
	can be seen to both be $V_2(\R^8) \homeo S^6 \x S^7$.
	For the latter,
	identifying $\R^8$ with the octonions $\defm\Oct$,
	we see $\SO(8)/\SO(6) = V_2(\R^8)$ can be identified
	with the unit tangent bundle to the sphere $S^7$ of unit octonions.
	Although~$S^7$ is not a group, it is a smooth H-space,
	and so left translation to $1 \in \R < \Oct$
	shows this unit tangent bundle is parallelizable,
	hence diffeomorphic to $S^6 \x S^7$
	(see also Kamerich~\cite[pp.~70,\,77]{kamerich}).
	The quotient $V_2(\R^4) = \SO(4)/\SO(2)$ is the unit tangent
	bundle to $S^3 = \Sp(1)$, 
	hence diffeomorphic to $S^3 \x S^2$ by the same reasoning.

	\hspace{20pt}
	For the former, write $\Oct = \Quat \+ \Quat\e$
	with $\e$ a new pure imaginary
	and fix the $\R$-basis 
	$\{1,j,\e,j\e,i,ij,i\e,ij\e\}$
	for $\Oct$ 
	to endow it with a linear $\SO(8)$-action.
	It is well known that $\Spin(8)$ 
	embeds in $\SO(8)^3$ 
	as the set of triples $(A,B,C)$
	such that for all $o,o' \in \Oct$
	we have $A(o)B(o') = C(oo')$,
	and it is a manifestation of triality that
	$A$ determines $\pm(B,C)$ and symmetrically, 
	so that each of the three coordinate projections
	is a double-covering
	$\Spin(8) \lt \SO(8)$. 
	The conditions $A = B = C$ cut out the subgroup $\defm{G_2}$
	of $\R$-algebra automorphisms of~$\Oct$.
	The equivalent conditions $A(1) = 1$ and $B = C$
	cut out a group $\defm{\Spin(7)^+}$ isomorphic to $\Spin(7)$,
	double-covering $[1] \+ \SO(7) < \SO(8)$
	under $(A,B,B) \lmt A$.
	The further demand $A(i) = i$
	gives $i B(o) = B(io)$, 
	and so cuts out a subgroup $\defm{\SU(4)^+}$
	taken injectively to $\SU(4) < \SO(8)$
	by $(A,B,B) \lmt B$.
	The subgroup $\defm{\Spin(7)^-} < \Spin(8)$ cut out by
	$B(1) = 1$ and hence comprising triples $(A,B,A)$,
	meets $\SU(4)^+$ in the group of triples $(A,A,A)$ with $A(i) = i$,
	namely the stabilizer of
	$i \in S^6 \subn \Im\, \Oct$
	under the action of $G_2 = \Aut_\R \Oct$.
	Hence this copy of $\SU(3)$ is
	the stabilizer of $(i,1)$ under the transitive
	action of $\Spin(7)^+$ on $S^6 \x S^7 \subn \Oct \x \Oct$
	by $(A,B,B)\.(o,o') \ceq \big(A(o),B(o')\mnn\big)$,
	yielding a diffeomorphism $\Spin(7)^+/\SU(3) \homeo S^6 \x S^7$.

\hspace{20pt}
	This explanation is mostly drawn from Kerr~\cite[\S6]{kerr1998}.
	
\item\labell{rmk:S6S7'}
	In the vocabulary of \ref{rmk:S6S7},
	the group $\defm{\SU(4)^-}$ of triples $(A,B,C) \in \SO(8)$
	with $B(1) = 1$ and $B(i) = i$ meets $\Spin(7)^+$ in 
	$\SU(3) = \Stab_{G_2} i$
	and they together generate $\Spin(8)$,
	so the inclusion $\Spin(7)^+ \longinc \Spin(8)$
	induces a diffeomorphism $\Spin(7)^+\mn/\SU(3) \lt \Spin(8)/\SU(4)^-$.

\item\labell{rmk:S6S15}
	The homogeneous space $\Spin(9)/\SU(4)$ 
	is the total space of a bundle
	\[
	S^6 \x S^7 = 
	\Spin(8)/\SU(4)^- \lt  \Spin(9)/\SU(4)^-\lt \Spin(9)/\Spin(8) = S^8
	\mathrlap,
	\]
	where $\Spin(8)/\SU(4)^+$ is as in \ref{rmk:S6S7} and \ref{rmk:S6S7'}
	and the inclusion $
	\Spin(8) \lt \Spin(9)$ is the standard
	one covering $\SO(8) \+ [1] \longinc \SO(9)$.
	As we have $\SU(4)^- < \Spin(7)^-$, there is a map from this bundle to
	\[
	S^7 = \Spin(8)/\Spin(7)^- \to \Spin(9)/\Spin(7)^- \to \Spin(9)/\Spin(8) = S^8\mathrlap.
	\]
	The fiber of $\Spin(8)/\SU(4)^-\longepi \Spin(8)/\Spin(7)^-$
	is $\Spin(7)^-/\SU(4)^-\homeo S^6$,
	as noted in~\ref{rmk:S6S7},
	so we may view this map as the coordinate projection
	$S^6 \x S^7 \lt S^7$,
	inducing an isomorphism on $H^7$.
	The projection $(A,B,A) \lmt A$ from \ref{rmk:S6S7}
	is a double-covering from $\Spin(7)^-$ to an $\SO(7)$ subgroup of $\SO(8)$
	giving one of the half-spin representations,
	and it follows $\Spin(7)^- \inc \Spin(8) \to \Spin(9)$
	is the map $\iota$ of \Cref{rmk:sphere-discussion}\ref{rmk:S15}.
	Thus the total space $\Spin(9)/\Spin(7)^-$ 
	of the second displayed bundle is $S^{15}$.
	Accordingly, in the \SSS,
	the fundamental class of the fiber $S^7$
	transgresses to that of the base $S^8$,
	so the induced map of spectral sequences
	shows the generator of degree $7$ in $\H(S^6 \x S^7)$
	transgresses to the fundamental class of $S^8$ as well.
	It follows by Poincar\'e duality 
	that $\H\big(\Spin(9)/\SU(4);\Z\big)$ 
	is isomorphic to $\H(S^6 \x S^{15};\Z)$.
	
\item\labell{rmk:S6S15'}
The embedding of $\SU(5)$ in $\Spin(10)$
	is obtained by lifting the standard embedding of $\SU(5)$ in $\SO(10)$.
	The subgroup $\SU(4) \+ [1]$ maps to $\SO(8) \+ [1]^{\oplus 2} < \SO(9) \+ [1]$,
	so the composite embedding $\SU(4) \to \SU(5) \to \Spin(10)$
	factors through the map $\SU(4)^-\lt \Spin(9)$ of \ref{rmk:S6S15}.
	To see the diffeomorphism $\Spin(9)/\SU(4) \lt \Spin(10)/\SU(5)$
	it remains to show $\Spin(9)$ and $\SU(5)$ meet in $\SU(4)$
	and together generate $\Spin(10)$.
	For the former, it is clear $\SU(4)$ lies in the intersection $H$,
	and the converse follows 
	by a dimension count and the exact sequence 
	$\pi_1\big(\Spin(10)/\SU(5)\mnn\big) \to \pi_0 H \to \pi_0 \Spin(9)$.
	For the latter, note
	the action $(A,B)\. C = ABC\-$ of $\SU(5) \x \Spin(9)$ on $\Spin(10)$
	is transitive since the induced action of 
	$\SU(5)$ on $\Spin(10)/\Spin(9) = \SO(10)/\SO(9) = S^9$,
	which is the standard action, is transitive.

	\item It is interesting that the two quotients $\Spin(7)/\SO(4)$ and 
	$\SO(8)/\big(\SO(3)\x \SO(5)\mnn\big) = \widetilde{G}_3(\R^8)$ 
	cited by Kramer~\cite[pp.~59--60]{kramer2002homogeneous} 
	are diffeomorphic.
	Kerr~\cite[\S6]{kerr1996some}
	notes the projection $(A,B,B) \lmt B$ 
	is an embedding $\Spin(7)^- \longmono \SO(8)$
	and that the action of the image on~$\R^8$
	induces a transitive action on $\widetilde{G}_3(\R^8)$,
	and it can be shown that any element stabilizing the $3$-plane $\R\{i,j,k\}$
	also fixes $1$, 
	so that the stabilizer is a subgroup of $G_2$ stabilizing $\Quat$.
	A dimension count shows that it is this entire subgroup, 
	which is known to be isomorphic to $\SO(4)$.
	We have seen in \Cref{rmk:correction} 
	that this Grassmannian does not have the 
	cohomology of $S^{4} \x S^{11}$,
	so these two quotients are not included in \Cref{table:Kramer+Wolfrom}.

\item
	One of the descriptions of $F_4$ is as the group of $\R$-algebra
	automorphisms of the $27$-dimensional Jordan algebra 
	$\fh_3(\Oct)$ and it follows indirectly from this description
	that $F_4$ acts transitively on the octonionic projective plane
	$\OP^2$ with stabilizer $\Spin(9)$.
	In fact, this realizes $\OP^2$ as a rank-one symmetric space,
	and the adjoint representation of 
	$\Spin(9)$ on $\f f_4$ decomposes as the $+1$-eigenspace $\spin(9)$
	of the involutive automorphism of $F_4$
	plus the $(-1)$-eigenspace, 
	of dimension $52-36 = 16$,
	the isotropy representation of 
	$\Spin(9)$ on the tangent space to $\OP^2$ at $1F_4$.
	Because $\OP^2$ is the irreducible symmetric space of type $FII$
	and has positive Euler characteristic,
	this $16$-dimensional representation is irreducible~\cite[p.~302, 8.13.4]{wolf2011book},
	and hence it can only be the spin representation, so $\Spin(9)$ acts transitively on the fiber
	$S^{15}$ of the unit tangent bundle
	with stabilizer $\Spin(7)^+$ as discussed in \ref{rmk:S6S15}.
	Thus the unit tangent bundle can be written as
	$S^{15} \to F_4/\Spin(7)^+ \to \OP^2$.
	Since the Euler characteristic $\chi(\OP^2)$ is 
	$\dim (H^0 \+ H^8 \+ H^{16}) = 3$,
	the fundamental class $[S^{15}]$
	transgresses to $3[\OP^2]$	in the \SSS, 
	so $\Ei \iso \H\big( F_4/\Spin(7);\Q \big)$
	is spanned by the four remaining classes, of degrees
	$0$, $8$, $23$, and $31$ 
	and hence the rational cohomology ring is isomorphic to $\H(S^8 \x S^{23})$
	by Poincar\'e duality.
	
\item 
	For the pair $\big(F_4,\Sp(3)\mnn\big)$,
	recall that $\Sp(3)$ is contained in a rank-$4$ subgroup
	$K = \Sp(1) \. \Sp(3)$ of $G = F_4$
	such that $G/K$ 
	is the symmetric space ${F\mn I}$ of $\HP^2$ subspaces of $\OP^2$.
	Ishitoya and Toda~\cite[Cor.~4.6]{ishitoyatoda}
	determine that the torsion-free quotient 
	of $\H(G/H;\Z)$ is isomorphic to $\H(S^8 \x S^{23})$
	on the way to computing the integral cohomology of $G/K$.
\item 
	For $G = \Sp(k) \x \Sp(2)$
	and $H = \Sp(k-1) \x \D\Sp(1) \x \Sp(1)$,
	if we take $K = \Sp(k) \x \Sp(1) \x \Sp(1)$,
	then $K/H \to G/H \to G/K$
	can be identified with 
	the unit sphere bundle
	of the Whitney sum of $k$ copies of the 
	tautological quaternionic line bundle over $\HP^1 = G_1(\HH^2)$.
	
\item 
For $G = G_2 \x \Sp(2)$
and $H \iso \Sp(1) \. \D\Sp(1) \. \Sp(1) = H_1\.H_2\.H_3$, 
there are two possibilities.
In both, we take $H_3$
to be $[1] \+ \Sp(1) < \Sp(2)$.
As for $H_1$ and $H_2$,
recall 
from \Cref{rmk:sphere-discussion}\ref{rmk:G2}
that $G_2$ contains a
virtual product subgroup $\SO(4) = \SU(2)_1\.\SU(2)_3$.
We may then take $H_1 = \SU(2)_1$
and $H_2$ to be the diagonal $\Sp(1)$ in 
$\SU(2)_3 \x \big(\Sp(1) \+ [1]\big) < \SO(4) \x \Sp(2)$,
or we may take 
$H_1 = \SU(2)_3$
and 
$H_2$ the diagonal in $\SU(2)_1 \x \big(\Sp(1) \+ [1]\big)$.
\end{enumerate}
\end{discussion}

\begin{proof}[Sketch proofs for inclusions of $(G,H)$ in \Cref{table:Kramer+Wolfrom}]
	We noted that \Cref{thm:deg} gives a necessary condition for $(G,H)$ to be a \rspp,
	but saw in \Cref{rmk:correction} 
	that it is possible a pair satisfying
	that criterion not have the correct cohomology, 
	so we need to compute the cohomology of all candidates.
	
	The tool of choice in most cases is the
	\defd{Cartan model} $\H(BH) \ox \H(G)$,
	a \CDGA whose cohomology is $\H(G/H)$
	and for which the spectral sequence
	associated to the filtration by degree in $\H(BH)$
	is the \SSS of the Borel fibration $EH \ox_H G \to BH$.
	The differential of the Cartan model
	is a derivation vanishing on $\H(BH)$;
	on homogeneous exterior generators
	of $\H(G)$, it is the composite
	$H^{*-1}(G) \os\tau\to \H(BG) \to \H(BH)$  
	of a lift of the universal transgression 
	and the functorial map $\H B(H \inc G)$.
    \revision{See Onishchik~\cite[\S\S\,8\,\&\,12]{onishchik} 
    for setup (with real coefficients)
    and Greub, Halperin, and Vanstone~\cite[Ch.~XI.4]{GHVIII} 
    for examples of computations in this model}.
	
	In most of the cases in the table, 
	computation of $\H(G/H)$
	via the Cartan model is not particularly challenging,
	so we will content ourselves 
	to illustrate a few representative examples
	not already covered in \Cref{rmk:example-discussion}.

\begin{itemize}
\item
	For $(G,H) = \big(G_2,i_{p,q}\U(1)\mnn\big)$,
	one has $\H (G) = \ext[z_3,z_{11}]$ and $\H(BH) = \Q[c_1]$,
	and $H^4(BG) \lt H^4(BH) = \Q c_1^2$ 
	is nonzero whenever $G$ is semisimple,
	independent of $p$ and $q$,
	so the ideal $(dz_3)$ of $\H(BH)$ is $(c_1^2)$.
	Irrespective of what $dz_{11} \in \Q c_1^6$ is, 
	it will lie in this ideal,
	so $d_{12}z_{11} = 0$ in the \SSS,
	showing $\H(G/H) \iso \Q[t]/(t^2) \ox \ext[z_{11}]$.
	The same argument applies in all cases
	with $H \iso \U(1)$.
\item 
	The pairs
$(G,H) = \big(\Sp(3),(\Sp(1)\+\D_2\Sp(1)\mnn\big)$ 
and $\big(\Sp(3),\SO(3)\.\D_3\Sp(1)\mnn\big)$ 
both have the rational cohomology of $S^4 \x S^{11}$.
For the former, 
we can factor the inclusion $\Sp(1)\+\D_2\Sp(1) \longinc \Sp(3)$
through $\Sp(1)^{\oplus 3}$.
The factor map \[
\H B\Sp(3) \lt \H B \Sp(1)^{\oplus 3}
\iso \big(\mn\H B\Sp(1)\mnn\big){}^{\otimes 3}
\iso \Q[q,q',q'']
\]
embeds the symplectic Pontrjagin classes $q_1$, $q_2$, $q_3$
as the elementary symmetric polynomials on the generators 
$q,q',q''$.
Under $(B\D_2)^*\: \H B\Sp(1)^{\oplus 2} \lt \H B\Sp(1)$,
both generators $q'$, $q''$ are taken to the polynomial
generator $r \in H^4 B\Sp(1)$,
so the map $\H B\Sp(3) \lt \H B\big(\Sp(1) \+ \D_2 \Sp(1)\mnn\big)$
is given by
\[
q_1 \lmt q + 2r,\qquad
q_2 \lmt 2qr + r^2,\qquad
q_3 \lmt qr^2\mathrlap.
\]
Computing the cohomology of the Cartan model 
$\Q[q,r] \ox \ext[z_3,z_7,z_{11}]$,
one finds $\H(G/K) \iso \Q[r]/(r^2) \ox \ext[\bar z_{11}]$
for an element $\bar z_{11} \in z_{11} + (\im d)$.
As noted earlier, integrally one has $H^8 \iso \Z/3$,
as one has $q \sim -2r$ in $H^4$ and $0 \sim 2qr + r^2 \sim -4r^2+r^2$ in $H^8$
using the \SSS.

\hspace{20pt}
As for the inclusion $\SO(3)\.\D_3\Sp(1) \longinc \Sp(3)$,
the maximal torus $\SO(2) \+ [1]$ of $\SO(3)$
can be conjugated to the diagonal subgroup 
with elements $\diag(\w,\w\-,1)$ for $\w \in \U(1)$
without moving the maximal torus $\big\{{\diag}(\z,\z,\z)\big\}$ 
of $\D_3\Sp(1)$.
The two-dimensional maximal torus $S$ generated by these two subtori
has generic element $(\w\z,\w\-\z,\z)$.
Taking $s \in H^2 BS$ to correspond to $\w$ and $t$ to $\z$,
the restriction from 
$\Q[t_1,t_2,t_3] \iso \H B\U(1)^{\oplus 3}$
is 
\[
t_1 \lmt t+s,\qquad\quad
t_2 \lmt t-s,\qquad\quad
t_3 \lmt t\mathrlap.
\]
Restricting the $q_j$, which are the elementary symmetric
polynomials in $t_1^2$, $t_2^2$, $t_3^2$, one finds
the map to $\H(BH) \iso \Q[s^2,t^2]$ is given by
\[
q_1 \lmt 2s^2 + 3t^2,\qquad
q_2 \lmt s^4 + 3t^4,\qquad
q_3 \lmt s^4 t^2 - 2s^2 t^4 + t^6\mathrlap.
\]
One then can compute the cohomology of the 
model $\Q[s^2,t^2] \otimes \ext[z_3,z_7,z_{11}]$
and find it is isomorphic to $\Q[t^2]/(t^4) \otimes \ext[\bar z_{11}]$
for an element $\bar z_{11} \in z_{11} + ({\im d})$.

\end{itemize}
	The cases involving classical groups are mostly resolved 
	by choosing a maximal torus $T_H$ of $H$ contained in a maximal torus $T_G$ of $G$
	and identifying $\H(BG) \lt \H(BH)$ with the map of invariants
	$\H(BT_G)^{W_G} \lt \H(BT_H)^{W_H}$,
	using the fact that these invariants are 
	generated by the symmetric polynomials $c_\ell$
	in the generators $t_j$ of $\H(BT) = \Q[\vec t]$ for groups of type $A_k$,
	symmetric polynomials $p_\ell$ in the $t_j^2$ for groups of type $BC_k$ and $D_k$,
	and additionally the product $e = t_1 \cdots t_k = \sqrt{p_k}$ for groups of type $D_k$.
	We do two examples of this type to communicate the basic idea.
	
\begin{itemize}
\item
	For $(G,H) = \big(\SU(5),\SU(2) \+ \SU(3)\mnn\big)$,
	the differential $\ext[z_3,z_5,z_7,z_9] \to \Q[c_2, c_3, c_4, c_5] = \H(BG) 
				\to 
			\H(BH) = \Q[c'_2] \ox \Q[c''_2,c''_3]$ sends
	\eqn{
	z_3 \lmt	c_2&\lmt c'_2 + c''_2, 
	\\
	z_5 \lmt	c_3&\lmt c''_3, 
	\\
	z_7 \lmt	c_4&\lmt c'_2 c''_2, 
	\\
	z_9 \lmt	c_5&\lmt c'_2 c''_3. 
		}
	In the \SSS of $G \to G/H \to BH$, then,
	we see  $c'_2 \equiv -c''_2$ after $E_4$
	and 	$c''_3 \equiv 0$ after $E_6$,
	so $d_8 z_7 \equiv -(c'_2)^2$
	and $d_{10}z_9 \equiv 0$.
	Thus $\H(G/H) \iso \Q[c'_2]/(c'_2)^2 \ox \ext[z_9]$.
	
\item
	For $G = G_2 \x \Sp(2)$ and $H = \Sp(1) \. \D\Sp(1)\.\Sp(1)$,
	the maximal torus 
	$T \iso \U(1)^4$ of $\Sp(1)\.\Sp(1) \x \Sp(1) \+ \Sp(1) < G$
	meets $H$ in $T' = \U(1) \x \D\U(1) \x \U(1)$,\footnote{\ 
		The two copies of $\U(1) < \Sp(1)$ in $G_2$
		meet in $\{\pm 1\}$,
		but the diagonal $\D\U(1)$ with one
		coordinate in $\Sp(1) \+ [1]$
		does not meet the $\U(1)$,
		so this really is a direct product.
		}
	so the induced map $\Q[t_1,u_1,u_2,t_2] = \H(BT) \lt \H(BT') = \Q[s_1,v,s_2]$
	is given by $t_j \lmt s_j$ and $u_j \lmt v$.
	Writing $\H(BG_2) \ox \H\big(B\Sp(2)\mnn\big) = \Q[r_4,q_{12}] \ox \Q[p_1,p_2]$
	and $\H(BH) = \Q[p',p'',p''']$, we get the splittings of characteristic classes
	\eqn{
		r	&\lmt	p' + p''\mathrlap,\\
		p_1 &\lmt	p'' + p'''\mathrlap,\\
		p_2 &\lmt	p''p'''\mathrlap.
		}
	For our purposes, 
	we do not need to work out what $q_{12}$ goes to, 
	but observe only that it will be some only polynomial 
	in $p'$ and $p''$.
	In the spectral sequence of $G \to G/H \to BH$, we have $p' \equiv - p'' \equiv p'''$ 
	from $E_6$ on
	since $p'+p''$ and $p'' + p''$ are images of $d_5$,
	and $(p')^2 \equiv 0$ from $E_{10}$ on since
	$p'' p''' \equiv (-p')p'$ is in the image of $d_9$.
	Since the image of $q_{12}$ is congruent 
	modulo the previous differentials
	to a polynomial in $p'$ alone and $(p')^2 \equiv 0$,
	it follows $d_{12}z_{11} = 0$,
	so we have
	$\H(G/H) \iso \Q[p']/(p')^2 \ox \ext[z_{11}]$.\qedhere
\end{itemize}
\end{proof}

\section{Isotropy-formal pairs}\labell{sec:isotf}
In \Cref{sec:spheres} we reduced the study of \isotfity of corank-one pairs $(G,K)$ 
to the computation of $\H(G/H_S)$
and examination of the corresponding irreducible pair in 
\Cref{table:Kramer+Wolfrom}.
In this section we determine which of those pairs is \isotf.
The largest and simplest class of cases is that of products $(K_1 \x K_2,H_1 \x H_2)$.

\blem\labell{thm:formalityOfProduct}
Suppose $G=K_1\cdot K_2$ is a virtual direct product of two compact, connected Lie groups, 
and $H = H_1\cdot H_2 \leq G$ is also a virtual direct product, with $H_j \leq  K_j$. 
The pair $(G,H)$ is \isotf if and only if both $(K_1,H_1)$ and $(K_2,H_2)$ are.
\elem
\begin{proof}
Noting that finite coverings do not impact the question of isotropy-formality
by \Cref{thm:cover-isotf},
we may assume $(G,H)$ is an honest product of $(K_1,H_1)$ and $(K_2,H_2)$.
Then the lemma follows from \Cref{thm:CF} or \Cref{thm:dim-criterion}.
\end{proof}

%
%
%
%
%
We now check isotropy-formality for 
the remaining pairs in \Cref{table:Kramer+Wolfrom}. 
For all pairs, we will simultaneously determine $\Cen_G(H)^0$ and $H_S$, 
which play crucial roles in \Cref{thm:enlargement-isotf} 
and \Cref{thm:enlargement-corank1} respectively,
determining \isotfity in the virtually effective but reducible case.
The calculation of $Z_G(H)^0$ appears in Kramer~\cite{kramer2002homogeneous}
and requires only mild 
expansion and correction, but
the calculation of $H_S$ 
in the cases we need does not seem to be in the literature.

\begin{discussion}\labell{def:HS-arguments}
There are a few basic principles underlying these computations.
For $G$ and $H$ presented as matrix groups or products thereof,
it is elementary and usually uncomplicated to compute $Z_G(H)^0$
directly.
When $H$ is a product, one can compute the centralizer of one factor
and then within that subgroup the centralizer of the other factor.
For the remaining pairs, in which $G$ is exceptional, we take recourse to the literature.

As for the maximal regular subgroup $H_S$ sharing the maximal torus $S$ of $H$
(\Cref{def:LH}),
it is usually not fastest to compute
it in terms of the root space composition of \Cref{thm:HS}.
In many cases, $H$ is regular,
which we
usually verify by checking $\rk Z_G(H)H/H = 1$
using \Cref{thm:enlargement-corank1},
and then $Z_G(H)^0H/H$ is either isomorphic to $S^1$ or a group of type $A_1$.
It is almost invariably $S^1$.
This is usually simple to verify directly,
but it is simpler still to note that if it is of type $A_1$,
then $G$ admits a maximal rank subgroup $H \x A_1$,
which can be ruled out by consultation with Borel--de Siebenthal~\cite{boreldesiebenthal}.
Then to check if $H$ agrees with $H_S$, we note that if $H_S > H$,
then by \Cref{thm:extend-H}, 
the pair $(G,H_S)$ appears in the table \ref{table:OddRationalSphere}
of rational sphere pairs, so if no pair $(G,K)$ with $K > H$ appears in this list,
then $H = H_S$. 
If, on the other hand, there are any such $(G,K)$ on the list,
then it remains to check if any of the $K$ are regular,
which again can be accomplished by determining the centralizer.
For brevity we call this the \defd{\csa}. 
\end{discussion}
\begin{proof}[Proof of \Cref{thm:isotfclassification}]
	\Cref{thm:enlargement} and 
	\Cref{thm:enlargement-equiv}
	extend the classification of irreducible {\rspp}s
	in \Cref{table:Kramer+Wolfrom}
	to the stated classification of \ve\ {\rspp}s
	in terms of subgroups $P < Z_G(H)^0$ 
	descending to rank-one subgroups of $N_G(H)^0/H
	= Z_G(H)^0 H/H$.
	\Cref{thm:enlargement-isotf} 
	shows such extensions are \isotf if and only if $H_S > H$.
	Thus it remains only to complete
	\Cref{table:Kramer+Wolfrom}
	by determining which irreducible {\rspp}s
	are \isotf and finding $Z_G(H)^0$ and $H_S$
	in all cases.
	
	By Lemmas \ref{thm:isotf-sphere} and \ref{thm:formalityOfProduct}, 
	all the pairs $(G,H)$ whose resulting homogeneous spaces are rational cohomology spheres 
	or virtual direct products of rational cohomology spheres are \isotf.
	This particularly covers the case $S^1 \x S^m$ studied by Bletz-Siebert.
	It remains to check Kamerich, Kramer, and Wolfrom's non-product cases 
	from \Cref{table:Kramer+Wolfrom}.
		

	When there is obviously some larger closed, connected $K$ sharing the
	maximal torus $S$ of $H$,
	we apply \Cref{thm:extend-H} to show the following pairs 
	\textbf{\GH are \isotf}
	in the following cases.
	We interleave a brief argument 
	computing $Z_G(H)^0$ and $H_S$ in each case.

\begin{itemize}
	\item $\big(\SU(4), \SU(2)\+\SU(2)\mnn\big)$ and $K = \Sp(2) > \Sp(1) \+ \Sp(1)$

	  Here $Z_G(H)^0 = \big\{\!\diag(z,z,z\-,z\-)\big\}$
	  and $H = H_S$ by a \csa.

	\item $\big(\SO(2k+1), \SO(2k-2)\mnn\big)$ and $K = \SO(2k-1)$

	  Here $Z_G(H)^0 = [1]^{\oplus 2k-2}\+\SO(3)$
	  directly and $[1]^{\oplus 2k-1} \+ \SO(2)$ 
	  centralizes $K$, which is thus regular;
	  hence $K = H_S$.
	

	\item $\big(\Spin(9),\SU(4)\mnn\big)$ and $K = \Spin(7) > \SU(4)$
		
	To see regularity,
	project down to $\SO(8)$, preserving $\SU(4)$,
	note that $\SU(4) < \U(4) < \SO(8)$ is normalized by the diagonal
	$\D_4\U(1)$, and lift to find its double cover in $\Spin(8)$
	centralizes $\SU(4)$ as well.
	That $Z_G(H)^0 \iso \U(1)$ and
	$H = H_S$ follow by a \csa.

	\item $\big(\Spin(7),\SU(3)\mnn\big)$ and $K = G_2$
	
	
	Since $\SU(3) < \SU(4)$ is normalized by $\big\{\!\diag(z,z,z,z^{-3})\big\}$,
	and $\SU(4)$ is contained in $\Spin(7)$ (see \Cref{rmk:example-discussion}\ref{rmk:S6S7}),
	we may run a \csa.

	\item 	$\big(\Sp(3),\Sp(1)\+\Sp(1)\+[1]\big)$ and $K = \Sp(2) \+ [1]$

		One has $Z_G(H)^0 = [1]^{\oplus 2} \oplus \Sp(1)$ by direct
		computation and this group also centralizes $K$,
		so $K = H_S$ is regular.
\end{itemize}

	We can also directly find a group $K$ sharing a maximal torus $S$ 
	and such that \GK is a \rsp and use
 \Cref{thm:isotropyOddSphere}, 
 to show the following \textbf{\GH are \isotf}.

\begin{itemize}

\item $\big(\SU(4), \SO(4)\mnn\big)$ and $K = \Sp(2)$ and $S = \SO(2)\+\SO(2)$ 

	The centralizer is $\pm[1]^{\oplus 4}$
	and $H' = \SU(2)^{\oplus 2}$ contains
	$S$; we saw above that $H' = H_S$.

\item $\big(\SO(7),\SO(3)\+\SO(3)\+[1]\big)$ and $K = \SO(5) \+ [1]^{\oplus 2}$ and $S = \SO(2) \+ [1] \+ \SO(2) \+ [1]^{\oplus 2}$

	The centralizer is $\pm[1]^{\oplus 6} \oplus [1]$,
	but $K$ is centralized by $[1]^{\oplus 5} \oplus \SO(2)$,
	and by a \csa, $K = H_S$.

%

\item $\big(\SO(2k),\SO(2k-2)\mnn\big)$ with $k\geq 3$ and $K = \SO(2k-1)$
and $S = \SO(2)^{\oplus k -1}$

One sees the centralizer is $\pm[1]^{\oplus 2k-2} \+ \SO(2)$
and applies a \csa ($K$ is not regular, since its centralizer is $\pm[1]^{\oplus k}$).

\end{itemize}

Not quite of this class, but closely related, is the following:

\begin{itemize}
	\item	$\big(\Spin(8),\SU(4)\mnn\big)$ \textbf{is \isotf}.
	
	Recall that in $\Spin(8)$, the triality automorphism takes
	the subgroup $\SU(4)$ to the standard $\Spin(6)$%
	~\cite[pp.~40--1]{kamerich}\cite[p.~60]{kramer2002homogeneous},
	but $\big(\Spin(8),\Spin(6)\mnn\big)$
	double-covers the just discussed \isotf pair $\big(\SO(8),\SO(6)\mnn\big)$.
	We found $H = H_S$ is regular and $Z_G(H)^0 \iso S^1$
	while discussing $\big(\Spin(9),\SU(4)\mn\big)$ above.
\end{itemize}

For a few examples, we are able to determine the \mrp $(G,H_S)$
associated to $S$ 
by observing $H$ is not regular, 
but a subgroup $H'$ is, and $H/H'$ is rationally homotopy equivalent
to an even-dimensional sphere, so that
by \Cref{thm:dichotomy} we have $H' = H_S$.
Then because $(G,H_S)$ is not a \rspp,
we may conclude by \Cref{thm:spheres} 
that \textbf{\GH is {not} \isotf} in the following cases.
Moreover, since $H > H_S$, we conclude $H$ is not regular,
and so $Z_G(H)^0 H = H$ by \Cref{thm:enlargement-corank1};
but then $Z_G(H)^0 = Z(H)^0$ is trivial.
\begin{itemize}

\item $\big(\Sp(3),\Sp(1)\+\SU(2)_{10}\big)$, with $H' = \Sp(1)\.i_{3,1}\U(1)$

	The standard maximal torus $\U(1)^{\oplus 3}$ normalizes $H'$.
	According to Berger~\cite[p.~237]{berger1961},
	$i_{3,1}\U(1)$ is a maximal torus of
	the embedded subgroup $\SU(2)_{10}$ in $\Sp(2)$,
	and so $H/H' \homeo S^2$.

\item $\big(\Sp(k) \x \Sp(2),\Sp(k-1) \x \D\Sp(1) \x \Sp(1)\mnn\big)$ 
for $k\geq 2$, with $H' = \Sp(k-1) \x \D\U(1) \x \Sp(1)$



\item $\big(G_2 \x \Sp(2), \Sp(1) \x \D\Sp(1) \x \Sp(1)\mnn\big)$,
with
$H' = \Sp(1) \x \D\U(1) \x \Sp(1)$

Note that this actually comprises two cases, depending which of the subgroups
$\SU(2)_1$ and $\SU(2)_3$ is chosen to serve as which $\Sp(1)$ in $H$.
\end{itemize}

Again not quite in this class, but related, is the following:

\begin{itemize}
\item $\big(\Sp(3),\SO(3)\.\D_3\Sp(1)\mnn\big)$ is not \isotf.

	For the centralizer, one has $Z_G\big(\SO(3)\mnn\big) = \D_3\Sp(1)$
	but $Z\big(\Sp(1)\mnn\big) = \{\pm 1\}$,
	so $H$ is not regular.
	A maximal torus $S$ is $\SO(2) \. \D_3 \U(1)$,
	and one has $H/S \homeo S^2 \x S^2$.
	This torus $S$ is also contained in $H' = \SU(2)\.\D_3 \U(1)$,
	which is normalized by the torus $\SO(2)\.\D\U(1) \+ \U(1)$,
	hence regular.
	To see $H' = H_S$,
	note that by two spectral sequence collapse arguments we have
	$\dim_\Q \H(G/S) = 16$ and $\dim_\Q \H(G/H') = 8$
	by the collapse of the corresponding \SSS.
	A larger $H'' > H'$ sharing $S$ as a maximal torus
	would have $\dim_\Q \H(G/H'') \in \{2,4\}$ 
	and $H^{11}(G/H'') \iso \Q$
	by comparing
	the Cartan algebras of $G/H''$ and $G/H'$ using
	\Cref{thm:formal} and the resulting observation
	that formality is determined by $(G,S)$ alone.
	But then $(G,H'')$ would appear in \Cref{table:OddRationalSphere},
	and it does not, 
	since $H'$ is not contained in any conjugate of $\Sp(2)$.

\end{itemize}

For the remaining examples, 
we compare $N$, $W_{H_S}$, and $W_H$,
per \Cref{thm:spheres}.
In the easier of these, 
$H$ is regular and $w_0^G$ acts as $-{\id}$ on $\ft$,
so by \Cref{thm:-Id},
\GH is \isotf if and only if 
$w_0^H$ does not act as $-{\id}$ on $\fs$.
In the following cases, both $w_0^G$ and $w_0^H$ act as $-{\id}$, so
\textbf{\GH is not \isotf}.
\begin{itemize}

	\item $\big(\SO(7),\SO(3)\+\SU(2)\mnn\big)$
	
	The block $\SU(2) < \U(2) < \SO(4)$ is normalized by the 
	standard maximal torus $\U(1)^{\oplus 2} < \U(2)$,
	so  $H$ is regular.
	That $Z_G(H)^0 \iso \U(1)$ and
	$H = H_S$ follow by a \csa.
	

\item $\big(\Sp(3),\Sp(1)\+\D_2\Sp(1)\mnn\big)$ (or $\big(\Sp(3),\Sp(1)\+\SU(2)\mnn\big)$)


	Here $W_H$ and $N$ are the same as for $\big(\SO(7),\SO(3)\+\SU(2)\mnn\big)$.
	Only $[1] \+ \Sp(2)$ centralizes the first factor $\Sp(1) \+ [1]^{\oplus 2}$ of $H$.
	Within $\Sp(2)$, only $\SO(2)$ centralizes $\D_2\Sp(1)$,
	so $Z_G(H) = [1] \+ \SO(2)$.
	That $H = H_S$ follows by a \csa.

\item $\big(F_4,\Spin(7)\mnn\big)$

	Within $\Spin(9) < F_4$, the subgroup $\Spin(2)$ normalizes
	the standardly embedded $\Spin(7) = H'$
	(as $\SO(7) \+ [1]^{\oplus 2}$
	and $[1]^{\oplus 7} \+ \SO(2)$ commute),
	so $H'$ is regular.
	But $H'$ is taken to $H$ by an outer 
	automorphism of $\Spin(8)$,
	and all automorphisms of $\Spin(8)$
	are realized through conjugation
	by an element of $N_{F_4}\big(\Spin(8)\mnn\big)$%
	~\cite[Thm.~14.2]{adamsExceptionalbook},
	so $H$ too is regular in $F_4$.
	That $Z_G(H)^0 \iso \U(1)$ and
	$H = H_S$ follow by a \csa.

\item $\big(F_4,\Sp(3)\mnn\big)$

	The subgroup $\Sp(3)$ is contained in a copy of $\Sp(3) \otimes_{\Z/2} \SU(2) < F_4$,
	within which $\SO(2)$ centralizes it.
	That $Z_G(H)^0 = \SO(2)$ and
	$H = H_S$ follow by a \csa.

\end{itemize}
	
\nd In the following cases, $H$ is regular and 
$w_0^G$ acts as $-{\id}$ while $w_0^H$ does not, so
\textbf{\GH is \isotf}.

\begin{itemize}

\item $\big(\Sp(4),\SU(4)\mnn\big)$

	The central circle $\D_4\U(1)$ 
of $\U(4) < \Sp(4)$ centralizes $\SU(4)$.
That $Z_G(H)^0 \iso \U(1)$ and
$H = H_S$ follow by a \csa.

\item $\big(\Sp(3),\SU(3)\mnn\big)$

The central circle $\D_3\U(1)$ 
of $\U(3) < \Sp(3)$ centralizes $\SU(3)$.
That $Z_G(H)^0 \iso \U(1)$ and
$H = H_S$ follow by a \csa.

\item $\big(\Sp(2), i_{p,q}\U(1)\mnn\big)$ for any coprime pair $(p,q)$

To find $H_S$ and $Z_G(H)^0$, note that $H$ is
regular with centralizer at least the maximal torus
since it is contained within it. 
To determine when it is contained in an $A_1$
subgroup or centralized by one,
recall from
the representation theory of $\Sp(1)$ and $\SO(3)$
and \Cref{thm:regularity} that
the regular $A_1$-subgroups of $\Sp(2)$ up to conjugacy 
are $\Sp(1)\+[1]$ and $\D\Sp(1)$.

\item $\big(G_2, i_{p,q} \U(1)\mnn\big)$  for any coprime pair $(p,q)$

Here $H$ lies within the maximal torus $T$ of $G_2$ from \Cref{rmk:sphere-discussion}\ref{rmk:G2},
and hence is at least centralized by $T$, hence regular.
If we have $Z_G(H)^0 H/H$ an $A_1$ group,
then $Z_G(H)$ contains an $A_1$ subgroup virtually disjoint from $S$,
and following \Cref{rmk:sphere-discussion}\ref{rmk:G2},
these fall into four conjugacy classes,
of which only those of $\SU(2)_1$ and $\SU(2)_3$ admit a centralizing
circle.
These are also the only regular $A_1$ subgroups.
Thus $H = H_S$ and $Z_G(H)^0=T$ unless $H$ is (conjugate to)
the maximal circle of either $\SU(2)$,
in which case $H_S > H$ is that $\SU(2)$ and 
$Z_G(H)^0$ is the product of $H$ and the other $\SU(2)$.
\end{itemize}

In the remaining cases we apply \Cref{thm:spheres} again,
but require subcases or \emph{sui generis} arguments
in determining whether $N > W_H$.

\bitem

\item $\big(\SU(5),\SU(2)\+\SU(3)\mnn\big)$ is not \isotf.

One has $N = W_H = \Sigma_{\{1,2\}} \x \Sigma_{\{3,4,5\}}$ in $W_G = \Sigma_{\{1,2,3,4,5\}}$
and sees $H$ is centalized by elements 
$\diag(z^3,z^3,z^{-2}, z^{-2},z^{-2})$.
That $Z_G(H)^0 \iso \U(1)$ and
$H = H_S$ follow by a \csa.

\item $\big(\Spin(10),\SU(5)\mnn\big)$ is not \isotf.

In this example the embedding is defined by 
lifting the standard block embedding of $\SU(5)$ in $\SO(10)$,
but we may work locally in $\so(10) \iso \spin(10)$.
Hence $\ft = \so(2)^{\oplus 5} \iso \R^5$ 
and $W_{\Spin(10)}$ 
is the semi-direct product of $\Sym_5$ with 
the kernel
$
\mr S \{\pm 1\}^5
$
of the multiplication map $\{\pm 1\}^5 \lt \{\pm 1\}$.
Then $\fs = \{\vec x \in \R^5 : \sum x_j=0\}$ 
is 
the nullspace of the weight
$\weight = (1,1,1,1,1)\in ({\Z^5})\dual$.
The stabilizer $N$ of $\{\pm\weight\}$ is just $\Sym_5 = W_{\SU(5)}$
since elements of $W_{\Spin(10)}$ 
can negate only an even number of coordinates
and $5$ is not even.

That the centralizer contains a nontrivial circle follows 
from lifting the standard circle $\D_5 \SO(2) < \SO(10)$ centralizing $\SU(5)$.
That $Z_G(H)^0 \iso \U(1)$ and
$H = H_S$ follow by a \csa.

\item $\big(\SU(3), i_{p,q}\U(1)\mnn\big)$ is \isotf if and only if $(p,q)$ is one of $(\pm1,0)$, $(0,\pm1)$, or $(\pm1,\mp1)$.

Since $\fs \iso \R$,
we have $N > W_H$ if and only if some element $w \in W_G$
acts as $-{\id}|_\fs$.\footnote{\ 
Isotropy-formality in the case $H$ is of rank one was already 
characterized in one of the authors' earlier works~\cite{carlson2018eqftorus},
where \isotf pairs \GH were classified.
}
This only happens for $S$ conjugate to 
$\mr{S}\big(\U(1)^{\oplus 2}\big) \+ [1]$ 
under 
the coordinate-permutating action  of $\Sigma_3 = W_{\SU(3)}$
on $T = \mathrm S\big(\U(1)^{\oplus 3}\big)$. 
If we parameterize $T$ by
$(z,w) \lmt \diag(z,w,z\-w\-)$
so that $i_{p,q}(t) = \diag(t^p,t^q,t^{-p-q})$,
this is just what we have claimed.


To find $H_S$ and $Z_G(S)^0$, 
we view $\ft$ as $\{\vec x \in \R^3 : \sum x_j = 0\}$, 
where $\R^3$ carries the standard inner product,
so that
$\fs = \R (p,q,-p-q)$ 
and $\fs^\perp = \R (p+2q,-q-2p,p-q) \eqc (\defm{p^*},\defm{q^*},-p^*-q^*)$.
Because the only  representations of
$\SU(2)$ and $\SO(3)$ of dimension $\leq 3$ are the standard ones,
all $A_1$-subgroups of $\SU(3)$ are conjugate to the standard ones.
By \Cref{thm:regularity},
$\SO(3)$ is not a regular subgroup of $\SU(3)$, 
so any $H = H_S$ is isomorphic to $S$ or $\SU(2)$,
and the semisimple virtual factor of $Z_G(H)^0$ 
can only be, respectively, $\SU(2)$ or $\trivialgroup$. 
Since any $\SU(2)$ is conjugate to the standard one,
if $H_S > S = i_{p,q}\U(1)$, then $p\.q \in \{0,-1\}$.
	For $Z_G(H)^0$ to be one of the conjugates of the standard $\SU(2)$
	on the other hand, we need $(p^*,q^*)$
	to be $(\pm 1,0)$, $(0,\pm 1)$, or $(\pm 1,\mp 1)$
	up to scaling, and one checks this happens
	if $(p,q)$ is respectively $(\pm1,\mp2)$, $(\mp2,\pm1)$, or $(\pm1,\pm1)$.
\item $\big(G' \x \Sp(1), H' \. i_{p,q}\U(1)\mnn\big)$, 
where $(G',H')$ is a rational cohomology sphere
as in \Cref{table:OddRationalSphere} with $\Cen_{G'}(H')^0\neq \{1\}$ 
and $p$ and $q$ are coprime and nonzero, is \isotf 
except when $(G',H')=\big(\SU(k+1),\SU(k)\mnn\big)$ for $k\geq 2$. 

	Let $\defm {\f u_1}$ be the Lie algebra
	of a maximal torus of the rank-$1$ group $Z_{G'}(H')$. 
	Since the rank of $H'$ is one less than that of $G'$,
	it follows the sum of $\f u_1$ with the Lie algebra $\defm{\fs_1}$ of a maximal torus $\defm{S_1}$ of $H'$ 
	spans the Lie algebra  $\defm{\ft_1}$ of a maximal torus of $G'$, i.e., $\ft_1=\fs_1\+ \f u_1$.
	We consider the Weyl group of $G'$ with respect to this particular torus.
	The Lie algebra $\defm{\f u_2}$ 
	of a maximal torus of $\Sp(1)$ is one dimensional,
	so we may make the identifications 
	\[
	\ft = \ft_1 \+ \f u_2 
	= \fs_1 \+ \f u_1 \+ \f u_2 
	\iso \fs_1 \+ \R \+ \R
	\] 
	in describing the Lie algebra $\defm\ft$ of the maximal torus of $G = G'\x \Sp(1)$.
	Under this identification, 
	the image $\defm{\f u}$ of the tangent space $\R$ of $S^1$
	under the inclusion $i_{p,q}$ of \Cref{rmk:table-explanation}
	becomes $\R\big((0,p),q\big)$ 
	and the Lie algebra $\fs$ of the maximal torus 
	of $H = H'\.i_{p,q}\U(1)$ 
	becomes $\fs_1 \+ \f u = \fs_1 \+ \R\big((0,p),q\big)$.

	Because we assume $q$ is nonzero,
	$H$ is abstractly isomorphic to $H' \x \U(1)$, 
	so we have isomorphisms $W_H \iso W_{H'} \x W_{\U(1)} \iso W_{H'}$.
	The longest word \defm{$\widetilde{w}_0$} 
	of $W_{G' \x \Sp(1)} \iso W_{G'} \x W_{\Sp(1)}$
	is the pair $(w_0,-{\id}|_{\f u_2})$,
	for $\defm{w_0} \in \Aut(\fs_1 \+ \R)$ the longest word of $W_{G'}$.

	\bitem
\item For $G'=\SO(2k+1)$, $\Sp(k)$ and $G_2$, we have $w_0|_{\ft_1}=-{\id}_{\ft_1}$, hence the word $\widetilde{w}_0$ induces $-{\id}$ on all of $\ft$. 
	Since $W_H = W_{H'}$ acts trivially on $\f u$, 
	we see $\widetilde{w}_0|_{\fs}=-{\id}_{\fs}\notin W_H$. 
	Thus, we have $N\neq W_H$.

\item For $(G',H')=\big(\SU(k+1),\SU(k)\+[1]\big)$ with $k\geq 2$, 
	the Lie algebra of the maximal torus of $G=\SU(k+1)\x\Sp(1)$
	can be written as
	\begin{align*}
		\ft = \Big\{(x_1,x_2,\ldots,x_{k+1};y) \in \R^{k+1}\+\R :  \sum_{j=1}^{k+1} x_j =0\Big\}.
	\end{align*}
	Under this identification,
	the Lie algebras of the maximal tori of $H'=\SU(k)$ and~$i_{p,q}\U(1)$ become
	respectively
	\begin{align*}
		\phantom{\mbox{and}}\qquad
		\fs_1 & = \phantom{\R}\Big\{(x_1,x_2,\ldots,x_{k},0;0) \in \R^{k+1}\+\R :  \sum_{j=1}^{k} x_j =0\Big\} \qquad\mbox{and}\\
		\f u & = \R\.(p,p,\ldots,p,-kp;q).
	\end{align*} 
	Thus the Lie algebra $\fs=\fs_1\+ \f u$ of the maximal torus of $H$ 
	is perpendicular to the vector $v=(q,q,\ldots,q,-kq;-(k^2+k)p)$. 
	Since $k\geq 2$, we have $kq \neq q$,
	so the stabilizer $\tN = W_{\{\pm v\}}$ of $\{\pm v\}$ 
	under the action of $W_G = \Sym_{k+1}\x \{\pm 1\}$ is exactly $W_v = \Sym_k$. 
	By \Cref{thm:transition}, we have $N = W_v|_\fs \iso \Sym_k$. 
	However, since $W_H = W_{H'} = \Sym_k$, we see $N = W_H$.
	\eitem
	To find $Z_{G' \x \Sp(1)} (H' \. i_{p,q}\U(1))^0$ 
	for $G' \in\big\{\SU(k+1), \SO(2k+1), \Sp(k), G_2\big\}$, 
	we first note that $Z_{G' \x \Sp(1)} (H')^0 = Z_{G'}(H')^0 \x \Sp(1)$, 
	which is $\U(1)\x\Sp(1)$ or $\Sp(1)\x\Sp(1)$ in each case
	since we assume $Z_{G'}(H') \neq \trivialgroup$. 
	The centralizer of $i_{p,q}\U(1)$ within this subgroup is  
	is $\U(1) \x \U(1)$ in all cases, since $p\.q \neq 0$. 

	To find $H_S$, 
	we first observe that the normalizer of $H' \. i_{p,q}\U(1)$ in $G' \x \Sp(1)$
	contains both this centralizer $\U(1) \x \U(1)$
	and $S_1$,
	and hence all of $T$. 
	Thus $\big(G' \x \Sp(1),H' \. i_{p,q}\U(1)\mnn\big)$ is a regular pair.
	It remains to determine when $H = H_S$.
	If instead $H' \. i_{p,q}\U(1) < H_S$,
	then $\big(G' \x \Sp(1),H_S\big)$ is a \rsp by \Cref{thm:extend-H}.
	Since $G' \x \Sp(1)$ is not simple,
	the pair $(G' \x \Sp(1), H_S)$ is not irreducible.
	As $G'$ is simple, this pair must then
	be an enlargement of $(G' , H')$ of the form 
	$(G' \x \Sp(1), H' \. \D \Sp(1))$ described in \Cref{thm:enlargement}, 
	meaning $Z_{G'}(H')^0$ is an $A_1$ subgroup and 
	$(p,q)=(1,\pm 1)$ (up to conjugacy).
	Thus $H_S$ is determined by consulting \Cref{table:OddRationalSphere}.
\qedhere
\eitem
\end{proof}

{\footnotesize\bibliography{bibshort}}

\bigskip

\nd\footnotesize{\textsc{
	}\\
	\url{jeffrey.carlson@tufts.edu}
}

\medskip

\nd\textsc{North China Electric Power University
}\\
\url{che@ncepu.edu.cn}

\end{document}